\documentclass[12pt]{amsart}
\usepackage[margin=1.15in]{geometry}
\usepackage{amscd, amssymb, amsmath, wasysym}
\usepackage{graphicx}
\usepackage{amsfonts}
\usepackage{mathrsfs}
\usepackage{eucal}
\usepackage{latexsym}
\usepackage{verbatim}
\usepackage[all]{xy}
\usepackage[dvipsnames]{xcolor}
\usepackage{bookmark}
\usepackage{subcaption}
\usepackage{tikz-cd}

\usepackage{microtype}

 \usepackage{hyperref}
 \hypersetup{
     colorlinks=true,
     linkcolor=NavyBlue,
     filecolor=NavyBlue,
     citecolor = TealBlue,
     urlcolor=magenta,
  }

\newtheorem{introthm}{Theorem}

\newcounter{thmcounter}

\numberwithin{equation}{section}
\numberwithin{thmcounter}{section}

\newtheorem{theorem}[thmcounter]{Theorem}
\newtheorem{proposition}[thmcounter]{Proposition}
\newtheorem{lemma}[thmcounter]{Lemma}
\newtheorem{corollary}[thmcounter]{Corollary}

\theoremstyle{definition}
\newtheorem{definition}[thmcounter]{Definition}

\newtheorem{example}[thmcounter]{Example}

\newtheorem{remark}[thmcounter]{Remark}

\newtheoremstyle{claim}{9pt}{3pt}{}{\parindent}{\bf}{.}{1em}{}

\theoremstyle{claim}

\newenvironment{namelist}[1]{%
\begin{list}{}
{
\settowidth{\labelwidth}{#1}%
\setlength{\labelsep}{0.3em}%
\setlength{\leftmargin}{\labelwidth}%
\addtolength{\leftmargin}{\labelsep}}}{%
\end{list}}

\newcommand{\nN}{\mathbf{N}}
\newcommand{\nZ}{\mathbf{Z}}
\newcommand{\nR}{\mathbf{R}}
\newcommand{\nC}{\mathbf{C}}

\newcommand{\kk}{\mathbf{k}}
\newcommand{\nG}{\mathbf{G}}
\newcommand{\nP}{\mathbf{P}}
\newcommand{\nA}{\mathbf{A}}

\newcommand{\sA}{\mathscr{A}}

\newcommand{\sF}{\mathscr{F}}
\newcommand{\sG}{\mathscr{G}}
\newcommand{\sO}{\mathscr{O}}

\newcommand{\sI}{\mathscr{I}}

\newcommand{\sL}{\mathscr{L}}

\DeclareMathOperator{\coker}{coker}

\DeclareMathOperator{\id}{id}

\DeclareMathOperator{\Ker}{Ker}

\DeclareMathOperator{\Pic}{Pic}

\DeclareMathOperator{\pr}{pr}

\DeclareMathOperator{\Supp}{Supp}
\DeclareMathOperator{\supp}{Supp}

\DeclareMathOperator{\Sym}{Sym}

\DeclareMathOperator{\rank}{rank}

\DeclareMathOperator{\GW}{GW}

\newcommand*{\longhookrightarrow}{\ensuremath{\lhook\joinrel\relbar\joinrel\rightarrow}}

\newcounter{rkcounter}             
\begin{document}

\title[Ulrich sheaves for higher secant varieties of curves]{Ulrich sheaves and determinantal representations for higher secant varieties of curves}

\author{Daniele Agostini}
\address{Department of Mathematics, Eberhard Karls Universit\"{a}t T\"{u}bingen, Germany}
\email{daniele.agostini@uni-tuebingen.de}

\author{Mario Kummer}
\address{Faculty of Mathematics, Technische Universit\"{a}t Dresden, Germany}
\email{mario.kummer@tu-dresden.de}

\author{Jinhyung Park}
\address{Department of Mathematical Sciences, KAIST, Daejeon, Republic of Korea}
\email{parkjh13@kaist.ac.kr}

\date{\today}

\begin{abstract}
We show that higher secant varieties of smooth projective curves have symmetric admissible determinantal representations with symmetric Ulrich sheaves of rank one if the embedding is sufficiently ample. 
For secant varieties of real curves, we give conditions for the representing matrices to be positive definite. This allows us, under mild (conjecturally vacuous) conditions, to represent the convex hull of the curve as a spectrahedron whenever it is a hyperbolicity cone.
For secant varieties of rational normal curves, we derive very explicit representations in terms of Littlewood--Richardson coefficients.
One key tool that we use are the higher Szeg\H{o} kernels and the higher Scorza correspondences associated to a non-effective theta characteristic on the curve.
\end{abstract}

\maketitle


\section{Introduction}

\subsection{Determinantal representations, Ulrich sheaves, and their isometry classes}
Let $\kk$ be a field of characteristic zero but not necessarily algebraically closed, $V$ be a $\kk$-vector space of dimension $\dim V=r+1$, and $\nP^r=\nP(V)$ be the projective space of one-dimensional quotients of $V$.
 If $X \subseteq \nP^r$ is a hypersurface of degree $d$, a \emph{determinantal representation} of $X$ is a $d\times d$ matrix $\gamma=(p_{ij})$ of linear forms $p_{ij}\in V$ such that
\begin{equation}\label{eq:det_rep_usual}
X = \{ \det(\gamma) = 0 \}. 
\end{equation}
In this case, the sheaf $\sF$ defined as the cokernel of the map
\[ 0 \longrightarrow \kk^d\otimes \sO_{\nP^r}(-1)\overset{\gamma}{\longrightarrow} \kk^d\otimes \sO_{\nP^r} \longrightarrow \sF \longrightarrow 0 \]
is a sheaf of rank one supported on $X$ which is moreover an \emph{Ulrich sheaf}, meaning that its minimal free resolution is linear of expected length $r-\dim(\Supp(\sF))$. Conversely, if $\sF$ is such an Ulrich sheaf of rank one, it yields a determinantal representation of $X$ via the matrix appearing in its minimal free resolution.
\medskip

Assume now that $\kk=\nR$ and that $\gamma$ is a symmetric determinantal representation of $X$. For any point $q\in \nP^r(\nR)\setminus X(\nR)$ we can evaluate $\gamma$ at $q$ to obtain a symmetric matrix
 $\gamma(q) \in \nR^{d\times d}$. One can show that the \emph{signature} of this matrix is, up to a sign, equal to the \emph{Brouwer degree} of the linear projection from the point $q$:
\[ \pi_q \colon X(\nR) \longrightarrow \nP^{r-1}(\nR) \]
 with respect to certain (relative) orientations. Recall that the Brouwer degree is a signed count of the points in a fiber, where the sign is given according to whether the map locally preserves orientation or not. In particular, the matrix $\gamma(q)$ is definite (either positive or negative) only if $X$ is \emph{hyperbolic} with respect to the point $q$, meaning that the preimage of a real point by $\pi_q$ consists of $d$ real points. Equivalently, any real line in $\nP^{r-1}$ through $q$ intersects $X$ only at real points.
  \medskip

 In the case of an arbitrary field $\kk$ of characteristic $\operatorname{char}(\kk)\ne 2$, it was shown in \cite{agostinikummer} that the {isometry class} of the quadratic form associated to $\gamma(q)$ is equal
 to the $\nA^1$-degree of the linear projection from $q$:
 \[ \pi_q \colon X \longrightarrow \nP^{r-1} \]
according to a certain (relative) orientation. This is again a kind of a weighted count of points in the fibers and generalizes the Brouwer degree to arbitrary fields: for more information see e.g.~\cite{kw19,Lev20,pauliwickelgren}.
\medskip

The generalization of the previous discussion to an irreducible subscheme $X\subseteq \nP^r$ of dimension $n$ and degree $d$, goes through  \emph{Livsic-type admissible determinantal representations of rank one} of $X$. Such a representation is given by  a linear map
\[ \gamma\colon \wedge^{n+1}V \longrightarrow \kk^d\otimes \kk^d \]
that can be seen as a $d\times d$ matrix with entries in $\wedge^{n+1}V^{\vee}$ with properties described in more detail in \cite{shamovichvinnikov}. The key property is that if $p_0,\ldots,p_n\in V$, then $\gamma(p_0\wedge \cdots\wedge p_n)$ is a $d\times d$ matrix with coefficients in $\kk$ and
\begin{align*} 
\det \gamma(p_0\wedge \cdots \wedge p_n) \ne 0 &\iff  X\cap \{ p_0=\cdots =p_n = 0\} = \emptyset \\
& \iff (p_0:\cdots:p_n)\colon X \longrightarrow \nP^n \text{ has no base points. } 
\end{align*}
Notice that after choosing a basis of $V$ we get an  isomorphism $\wedge^{n+1}V \cong \wedge^{r-n}V$ and if we see $\gamma$ as a matrix with coefficients in $\wedge^{r-n}V$ the previous equivalence means precisely that $\gamma$ gives a \emph{determinantal representation of the Chow form of $X$}. This is a divisor in the Grassmannian $\nG(r-n-1,r)$ of $(r-n-1)$-planes in $\nP^r$:
\[ \ \operatorname{Chow}(X) = \{ \Lambda\in \nG(r-n-1,r) \,|\, \Lambda\cap X \ne \emptyset \} \subseteq \nG(r-n-1,n) \subseteq \nP(\wedge^{r-n}V).  \]  
Notice that if $X$ is a hypersurface, then $X=\operatorname{Chow}(X)$ so we get the usual determinantal representations of \eqref{eq:det_rep_usual}. It was also proven by the second author and Shamovich in \cite{kummershamovich} that such admissible determinantal representations correspond to Ulrich sheaves on $X$. If moreover $\gamma$ is symmetric, then it was proven by the first two authors in \cite{agostinikummer} that if $p_0,\ldots,p_n\in V$ have no common zero on $X$, then the isometry class of $\gamma(p_0\wedge \cdots \wedge p_n)$ in the Grothendieck--Witt ring $\GW(\kk)$ is equal to the $\nA^1$-degree of the linear projection
\begin{equation}\label{eq:linearprojection} 
\pi = (p_0:\cdots:p_n)\colon X \longrightarrow \nP^n 
\end{equation}
with respect to a certain relative orientation. The first result of our paper is a technical statement on how to construct admissible determinantal representations of rank one and their corresponding Ulrich sheaves. If the determinantal representation is symmetric and if the linear  projection of \eqref{eq:linearprojection} has a fiber consisting entirely of $\kk$-rational points we also show how to compute the corresponding isometry class.

\begin{introthm}\label{thm:thmA}
	Let $X\subseteq \nP^r$ be a nondegenerate irreducible projective variety of dimension $n$ and degree $d$, and $\sF,\sF'$ be two rank one globally generated torsion-free sheaves on $X$ with $h^0(X,\sF)=h^0(X,\sF')=d$. Consider a linear map
	\[ \gamma\colon \wedge^{n+1}V\longrightarrow H^0(X\times X,\sF\boxtimes \sF') = H^0(X,\sF)\otimes H^0(X,\sF') \]
	and for any $p_0,\ldots,p_n\in V$ consider the corresponding projection 
	\[ \pi = (p_0:\cdots:p_n)\colon X \dashrightarrow \nP^n . \]
	Assume that the following hold:
	\begin{enumerate}
		\item[(i)] If $\pi$ has no base points and if $\{x_1,\ldots,x_d\}$ is a general fiber over $\overline{\kk}$, then $\gamma(p_0\wedge \cdots \wedge p_n)(x_i,x_j)\ne 0$ if and only if $i=j$.
		\item[(ii)] If $\pi$ has a unique base point $x\in X$ and if $\{x_1,\ldots,x_{d-1}\}$ is a general fiber over $\overline{\kk}$, then $\gamma(p_0\wedge\cdots\wedge p_n)(x,y)=0$ for all $y\in X$ and $\gamma(p_0\wedge \cdots \wedge p_n)(x_i,x_j)\ne 0$ if and only if $i=j$.
	\end{enumerate}
Then $\gamma$ is an admissible determinantal representation of $X$ of rank one, with associated Ulrich sheaf $\sF$. Assume moreover that $\sF=\sF'$ and $\gamma$ is symmetric. If the projection $\pi\colon X\to \nP^n$ has no base points and if $\{x_1,\ldots,x_d\}$ is a fiber consisting only of $\kk$-rational points, then it admits a relative orientation such that the corresponding $\nA^1$-degree is equal to the class
\[ \langle \gamma(p_0\wedge \cdots \wedge p_n) \rangle = \sum_{i=1}^{d} \langle \gamma(p_0\wedge \cdots\wedge p_n)(x_i,x_i) \rangle \text{ in } \GW(\kk), \]
where the summands on the right are defined in Section \ref{sec:signatures}.
\end{introthm}

We believe that removing the assumption on the existence of a fully $\kk$-rational fiber in Theorem \ref{thm:thmA} should be straightforward. However, the main focus of this paper does not lie in arithmetic geometry and for our main applications this assumption is satisfied.
In the rest of the paper, we apply this result to secant varieties of smooth projective curves.

\subsection{Higher secant varieties of smooth projective curves}

Many projective varieties are known to have  an Ulrich sheaf and an admissible determinantal representation. The Boij--S\"oderberg cone of such a variety is the same as that of projective space of the same dimension and the latter was characterized in \cite{eisenbudschreyer2}.
In \cite[page 543]{eisenbudschreyer} Eisenbud and Schreyer asked whether every projective variety is the support of an Ulrich sheaf, and, if so, they asked for the smallest possible rank of such a sheaf. Under some mild conditions on the embedding of a curve, we answer both questions for its (higher) secant varieties: an Ulrich sheaf of rank one exists.
We fix throughout the paper  a smooth projective curve of genus $g$ embedded in $\nP^r$
\[C \longhookrightarrow \nP^r\]
We assume that the embedding is nondegenerate, so that it is given by a very ample line bundle $L=\sO_C(1)$ and a linear system $V\subseteq H^0(C,L)$ of dimension $\dim V=r+1$. The  \emph{$k$-th secant variety} of $C \subseteq \nP^r$ is defined as
$$
\Sigma_k=\Sigma_k(C,L;V):=\overline{\bigcup_{x_1, \ldots, x_{k+1} \in C} \langle x_1, \ldots, x_{k+1} \rangle} \subseteq \nP^r.
$$
This is an irreducible projective variety, and
Lange \cite{lange} proved that $\Sigma_k$ has the expected dimension $\min\{r, 2k+1\}$. This means that, if $r \geq 2k+2$, then $\Sigma_{k}\subseteq  \nP^r$ is a proper subvariety of dimension $\dim \Sigma_k = 2k+1$. If instead $r<2k+2$ then $\Sigma_k = \mathbf{P}^r$.

\medskip

It was shown by Eisenbud and Schreyer \cite[Theorem 4.3]{eisenbudschreyer} that all Ulrich sheaves of rank one on $C$ are of the form $L\otimes \alpha$, where $\alpha$ is a line bundle on $C$ with no cohomology:
\[  h^0(C,\alpha) = h^1(C,\alpha) = 0 \]
If furthermore $\alpha$ is a theta-characteristic, so that $\alpha^{\otimes 2} \cong \omega_C$, then $L\otimes \alpha$
is a symmetric Ulrich sheaf. The same result was shown by Shamovich and Vinnikov in \cite{shamovichvinnikov}, from the point of view of admissible determinantal representations. The main result of this paper is a generalization of this construction to (higher) secant varieties. To state it, recall that the embedding $C\hookrightarrow \nP^r$ is \emph{$k$-very ample} if any $k+1$, possibly coincident, points on $C$ span a $k$-dimensional linear space in $\nP^r$.

\begin{introthm}\label{thm:main}
	Let $C \subseteq \nP^r$ be a nondegenerate smooth projective curve of genus $g$ such that the embedding is $k$-very ample and assume $r\geq 2k+2$. Then for any line bundle $\alpha$ with no cohomology, the $k$-th secant variety $\Sigma_k$ has an admissible determinantal representation with corresponding Ulrich sheaf $\sA_{k,L\otimes \alpha}$ of rank one. If $\alpha$ is a theta-characteristic, then the determinantal representation is symmetric and $\sA_{k,L\otimes \alpha}$ is a symmetric Ulrich sheaf. Finally, $\sA_{k,L\otimes\alpha}, \sA_{k,L\otimes\beta}$ are isomorphic if and only if $\alpha, \beta$ are.
\end{introthm}

The result of Theorem \ref{thm:main} was previously shown for $g=0$  and arbitrary $k$ in \cite[Lemma 4.13]{RS} and for $k=1$ and arbitrary $g$ in \cite[Theorem C]{agostinikummer}, but under the assumption that the embedding is $3$-very ample. 
\medskip

If $C$ has infinitely many $\kk$-rational points, then \cite[Proposition 4.4]{eisenbudschreyer} shows that  there is a line bundle $\alpha$ with no cohomology so that Theorem \ref{thm:main} yields  an Ulrich sheaf (or an admissible determinantal representation) of rank one. If $\kk$ is algebraically closed, then the set of line bundles with no cohomology is the complement of the theta divisor inside the variety $\operatorname{Pic}_{g-1}(C)$, so that there is a $g$-dimensional family of Ulrich sheaves of rank one. Furthermore, there is also such an $\alpha$ that is a theta-characteristic. In particular there exists a symmetric Ulrich sheaf (or an admissible determinantal representation) on $\Sigma_k$. We also point out that our constructions are explicit so that, if we are in the hypotheses of Theorem \ref{thm:main}, we can compute the minimal free resolution of $\sA_{k,L\otimes \alpha}$ as well as the corresponding admissible determinantal representation of $\Sigma_k$.  Some examples are given in Sections \ref{subsec:explicitdetrep} and \ref{subsec:explicitminres}.
\medskip

As pointed out in \cite{eisenbudschreyer}, Ulrich sheaves and their admissible determinantal representations can be used to derive so-called \emph{B\'ezout expressions for resultants}. These are determinantal criteria for global sections of line bundles on projective varieties to have a common zero. In our case, these expressions can also be interpreted directly on the curve $C \subseteq \nP^r$. If the embedding is $k$-very ample, then Theorem \ref{thm:main} yields a matrix $M$ with entries in $\wedge^{2k+2}V^{\vee}$ such that for any $p_0,\ldots,p_{2k+1}\in V$ the matrix $M(p_0\wedge \cdots \wedge p_{2k+1})$ is non-singular if and only if the projection
\[ \pi = (p_0:\cdots:p_{2k+1})\colon C \longrightarrow \nP^{2k+1} \]
is $k$-very ample as well. 
\medskip

\subsection{Rational normal curves and Littlewood--Richardson coefficients}
We use the construction of Theorem \ref{thm:main} to compute a very explicit admissible symmetric determinantal representation for the secant varieties of the rational normal curve 
\[ \nP^1 \longhookrightarrow \nP^n, \quad t\longmapsto (1:t:t^2:\cdots:t^n), \]
purely in terms of the \emph{Littlewood--Richardson coefficients} from combinatorics and representation theory. Assume that $n\geq 2k+2$, so that the secant $\Sigma_k=\Sigma_k(\nP^1,\sO_{\nP^1}(n))$ is  a proper variety of degree $\binom{n-k}{k+1}$. Hence we are looking for a symmetric $\binom{n-k}{k+1}\times \binom{n-k}{k+1}$  matrix with coefficients in $\wedge^{2k+2}V^{\vee}$, where $V$ has a basis $1,t,t^2,\ldots,t^{n}$. First recall that $\wedge^{2k+2}V$ has a basis indexed by partitions $\lambda = (\lambda_1,\ldots,\lambda_{2k+2})$  fitting in a $(2k+2)\times (n-2k-1)$ box, meaning that $n-2k-1 \geq \lambda_1 \geq \lambda_2 \geq \ldots \geq \lambda_{2k+2} \geq 0$. The basis is given by
\[ \wedge^{\lambda} t := t^{\lambda_1+2k+1} \wedge t^{\lambda_2+2k} \wedge \cdots \wedge t^{\lambda_{2k+2}}. \]
We then denote by $x_{\lambda}$ the elements of the dual basis in $\wedge^{2k+2}V^{\vee}$. Observe now that there are $\binom{n-k}{k+1}$ distinct partitions $\mu = (\mu_1,\ldots,\mu_{k+1})$ fitting in a $(k+1)\times (n-2k-1)$ box, so that there is a symmetric matrix of linear forms $M^{n,k}(x_{\lambda})$ with entries indexed by such partitions $\mu,\nu$ and defined by
\[ M^{n,k}_{\mu,\nu}(x_{\lambda}) := \sum_{\lambda} c^{\lambda}_{\mu,\nu}\cdot x_{\lambda}, \]
where the sum runs over all partitions $\lambda$ fitting in a $(2k+2)\times (n-2k-1)$ box and the $c^{\lambda}_{\mu,\nu}$ are the Littlewood--Richardson coefficients.

\begin{introthm}\label{thm:rat_norm}
	If $n\geq 2k+2$, the matrix $M^{n,k}(x_{\lambda})$ gives a symmetric admissible determinantal representation of rank one of the $k$-th secant $\Sigma_k$ of the rational normal curve of degree $n$. In particular, it gives a determinantal representation of the Chow form of $\Sigma_k$.
\end{introthm}

For example, the first secant of a rational normal curve of degree $5$ has a symmetric admissible determinantal representation given by
\[
\begin{pmatrix}
	x_{{\emptyset}} & x_{{(1)}} & x_{{(2)}} & x_{{(1,1)}} & x_{{(2,1)}} & x_{{(2,2)}} \\
	x_{{(1)}} & x_{{(2)}} + x_{{(1,1)}} & x_{{(2,1)}} & x_{{(2,1)}} + x_{{(1,1,1)}} & x_{{(2,2)}} + x_{{(2,1,1)}} & x_{{(2,2,1)}} \\
	x_{{(2)}} & x_{{(2,1)}} & x_{{(2,2)}} & x_{{(2,1,1)}} & x_{{(2,2,1)}} & x_{{(2,2,2)}} \\
	x_{{(1,1)}} & x_{{(2,1)}} + x_{{(1,1,1)}} & x_{{(2,1,1)}} & x_{{(2,2)}} + x_{{(2,1,1)}} + x_{{(1,1,1,1)}} & x_{{(2,2,1)}} + x_{{(2,1,1,1)}} & x_{{(2,2,1,1)}} \\
	x_{{(2,1)}} & x_{{(2,2)}} + x_{{(2,1,1)}} & x_{{(2,2,1)}} & x_{{(2,2,1)}} + x_{{(2,1,1,1)}} & x_{{(2,2,2)}} + x_{{(2,2,1,1)}} & x_{{(2,2,2,1)}} \\
	x_{{(2,2)}} & x_{{(2,2,1)}} & x_{{(2,2,2)}} & x_{{(2,2,1,1)}} & x_{{(2,2,2,1)}} & x_{{(2,2,2,2)}}
\end{pmatrix}.
\]
More intrinsically, Theorem \ref{thm:rat_norm} says that the coproduct of symmetric functions gives an admissible determinantal representation for the secant of a rational normal curve. See Section \ref{subsec:explicitdetrep} for more details.

\subsection{Hyperbolic secant varieties of real algebraic curves}

Finally, we turn to real algebraic geometry. Assume that $C\subseteq \nP^r$  is a real algebraic curve. If $r\geq 2k+2$ and if $p_0,\ldots,p_{2k+1}\in V$ have no common zero on $\Sigma_k$, we consider the projection
\begin{equation}\label{eq:real_projection} 
\pi = (p_0:\cdots:p_{2k+1})\colon \Sigma_k \longrightarrow \nP^{2k+1} 
\end{equation}
and we say that this is \emph{real-fibered} if $\pi^{-1}(\nP^{2k+1}(\nR)) \subseteq \Sigma_k(\nR)$. Equivalently, one says that the secant $\Sigma_k$ is \emph{hyperbolic} with respect to the center $\Lambda = \{p_0=\cdots=p_{2k+1}=0\}$ of the projection. It was proven by the second author and Sinn in \cite{kummersinn}, that this condition is equivalent to the linear system $\langle p_0,\ldots,p_{2k+1} \rangle \subseteq V$ being \emph{vastly real}, meaning that any nonzero section in it has at most $2k$ nonreal zeros, counted with multiplicity.

\medskip

The notion of hyperbolicity in real algebraic geometry arose from the Lax conjecture for real plane curves. In the language introduced before, the conjecture states that a real plane curve $C\subseteq \nP^2$ is hyperbolic with respect to a real point $p\in \nP^2(\nR)$ if and only if it has a determinantal representation by a symmetric matrix $M$ of linear forms such that $M(p)$ is a (positive or negative) definite matrix. The Lax conjecture was proven by Helton and Vinnikov in \cite{heltonvinnikov}. We extend this result to higher secant varieties of curves, and we answer \cite[Question (4)]{kummersinn} affirmatively.

\begin{introthm}\label{thm:thmD}
	In the previous setting, assume that the embedding $C\subseteq \nP^r$ is $k$-very ample. Then the following are equivalent:
	\begin{enumerate}
		\item The projection $\pi\colon \Sigma_k \rightarrow \nP^{2k+1}$ is real-fibered.
		\item The linear system $\langle p_0,\ldots,p_{2k+1} \rangle$ is vastly real.
		\item There is an admissible determinantal representation $\phi$ of $\Sigma_k$ of rank one such that $\phi(p_0\wedge \cdots \wedge p_{2k+1})$ is (positive or negative) definite. 
	\end{enumerate}
	Furthermore, if $C$ is an $M$-curve and $V=H^0(C,L)$, then $V$ is $k$-very ample.
\end{introthm} 

Along the way, we also prove some structural results on vastly real linear systems, generalizing results from \cite{mikhalkinorevkov} to higher dimensions.

\medskip

For certain real curves $C\subseteq \nP^{2k+2}$ studied in \cite{kummersinn} the hypersurface $\Sigma_k = \{P=0\}$ is the algebraic boundary of the convex hull $K$ of one oval of $C(\nR)$. The determinantal representation of $\Sigma_k$ allows to rapidly evaluate $P$ in practice, and then $-\log(P)$ can serve as a   self-concordant barrier function for interior point methods to optimize over $K$, see \cite[Section 4]{kummersinn} and \cite{guler}. The complexity of evaluating $P$ is then crucial for the runtime of this optimization method. The situation becomes even better in the setting of Theorem \ref{thm:thmD}, because then the determinantal representation is definite at a point in the interior of $K$ and one can apply the method of \emph{semidefinite programming}.

\subsection{Higher Szeg\H{o} kernels, higher Scorza correspondence, and higher writhes}

 In order to construct the admissible determinantal representations for the secant variety $\Sigma_k$ of Theorem \ref{thm:main}, the key object is the $(k+1)$-\emph{Szeg\H{o} kernel} associated to a line bundle $\alpha$ on $C$ with no cohomology. This is a section of a certain line bundle on $C_{k+1}\times C_{k+1}$, where $C_{k+1}$ is the symmetric product of $C$, and its corresponding divisor is the $(k+1)$-\emph{Scorza correspondence}:
\[ \{ (\xi,\xi') \in C_{k+1}\times C_{k+1} \,|\, h^0(\alpha(\xi'-\xi))>0 \}. \]
When $k=0$, this is the Scorza correspondence in classical algebraic geometry, and the first Szeg\H{o} kernel was used in \cite{shamovichvinnikov} to construct an admissible determinantal representation of the curve $C$. Szeg\H{o}  kernels  appear also in mathematical physics \cite{Fay,raina} and string theory \cite{DHP,Wi}. We refer to \cite{FI} for a recent overview of these objects, with applications to moduli spaces of curves. Another recent generalization of Scorza correspondences, this time inside cartesian products of curves, is studied by Ruggiero \cite{ruggiero}, where the focus is on numerical classes and on local properties, such as smoothness.

\medskip

 We also use the $(k+1)$-Szeg\H{o} kernels and  Theorem \ref{thm:thmA} to compute the  $\nA^1$-degree of a linear projection $\pi\colon \Sigma_k \rightarrow \nP^{2k+1}$ as a sum of  local $\nA^1$-degrees. These local degrees have explicit formulas involving Gaussian-like maps on the curve $C$ and their explicit expression is the key to proving Theorem \ref{thm:thmD} for real curves. For $k=1$ these degrees were studied by the first two authors in \cite{agostinikummer} under the name of \emph{arithmetic writhes}, so that for $k\geq2$ one could call them \emph{higher arithmetic writhes}, but the focus of this work will not lie on their properties.
 
 \subsection{Structure of the paper} We briefly outline the structure of the paper. In Section \ref{sec:prelim} we discuss general admissible determinantal representations and Ulrich sheaves, and we prove Theorem \ref{thm:thmA}, apart from the statement on the isometry class in the Grothendieck--Witt ring. In Section \ref{sec:prelimsymmsec} we collect various results on symmetric products of curves, Szeg\H{o} kernels, and
 secant varieties of curves. In Section \ref{sec:constdetrep}, we show how to construct admissible determinantal representation of the secant $\Sigma_k$ under a certain assumption on its degree. This assumption is easily verified for rational normal curves, and this proves Theorem \ref{thm:rat_norm}. For arbitrary curves instead, the assumption is harder to verify directly, so that in Section \ref{sec:constulrich} we instead construct an Ulrich sheaf  on $\Sigma_k$ by cohomological vanishings, concluding the proof of Theorem \ref{thm:main}.  In Section \ref{sec:signatures} we compute the isometry classes in the Grothendieck--Witt ring of the determinantal representations obtained before, in particular completing the proof of Theorem \ref{thm:thmA}. Finally, we apply these techniques to real curves and we prove Theorem \ref{thm:thmD}.
 \medskip
 
 \noindent
 \textbf{Acknowledgements:} We thank Gavril Farkas, Maria Teresa Ruggiero and Rainer Sinn for useful discussions and comments.
 D.A.~thanks the Department of Mathematical Sciences at KAIST for the hospitality during the KAIST Thematic Program on Syzygies and Secants in 2024, where part of this paper was developed.  D.A.~was supported by the DFG under the joint ANR-DFG program “Positivity on K-trivial varieties” (DFG Project Nr. 530132094.) and by the SFB-TRR 195.
 M.K.~was partially supported by DFG grant 502861109.
 J.P.~was partially supported by the National Research Foundation (NRF) funded by the Korea government (MSIT) (RS-2026-25478877).

\section{Ulrich sheaves and admissible determinantal representations}\label{sec:prelim}
\noindent In this section, we set up basic notation and recall relevant facts on Ulrich sheaves and determinantal representation, following \cite{eisenbudschreyer} and \cite{kummershamovich}. We will then prove a part of Theorem \ref{thm:thmA} (= Theorem \ref{thm:admrepresentation}), which gives a way to construct admissible determinantal representations.

\subsection{Ulrich sheaves and admissible determinantal representations}
Let $\kk$ be a field of characteristic zero, $V$ a vector space over $\kk$ of dimension $\dim V=r+1$ and $\nP^r = \nP(V)$ the corresponding projective space of quotients of $V$. The symmetric algebra $S = \Sym V$ is the homogeneous coordinate ring of $\nP^r$. Let $X\subseteq \nP^r$ be a nondegenerate projective variety of pure dimension $n$: this means that we can see $V$ as a linear system $V \subseteq H^0(X,\sO_X(1))$.  A coherent sheaf $\mathscr{F}$ with $\supp(\mathscr{F})=X$ is called an \emph{Ulrich sheaf} if
$$
H^i(X, \mathscr{F}(-i))=0~\text{ for $i>0$}~~\text{ and }~~H^i(X, \mathscr{F}(-i-1))=0~\text{ for $i<n$}.
$$
Recall that the \emph{degree} of $\mathscr{F}$ is defined as
$$
\deg (\mathscr{F}):=\rank (\mathscr{F}) \cdot \deg(X).
$$

\begin{proposition}[{\cite[Proposition 2.1]{eisenbudschreyer}}]\label{prop:ulrich}
For a coherent sheaf $\mathscr{F}$ on $X$, the following are equivalent:
\begin{enumerate}
\item $\mathscr{F}$ is an Ulrich sheaf on $X$.
\item If $\pi \colon X \to \nP^n$ is a finite linear projection of $X \subseteq \nP^r$, then $\pi_* \mathscr{F} = \sO_{\nP^n}^{\oplus \deg(\sF)}$.
\item The finitely generated graded $S$-module $F:=\bigoplus_{m \in \mathbf{Z}} H^0(X, \mathscr{F}(m))$ is an Ulrich module, i.e., $F$
has a linear minimal free resolution of length $r-n$.
\item $F$ is a Cohen--Macaulay module of dimension $n+1$ whose number of generators is equal to $\deg(\mathscr{F})$.
\end{enumerate}
Furthermore, in this case $h^0(X, \mathscr{F}) = \deg(\mathscr{F}) = \rank{(\mathscr{F})}\cdot \deg(X)$.
\end{proposition}

\medskip

Let $U,U'$ be two $\kk$-vector spaces of dimension $\dim U = \dim U' = d$, and consider an element $\gamma \in \wedge^{n+1}V^{\vee}\otimes U\otimes U' = \operatorname{Hom}_{\kk}(\wedge^{n+1}V,U\otimes U')$. Once we fix bases $e_1,\dots,e_d$ and $e_1',\dots,e_d'$ of $U$ and $U'$ respectively, $\gamma$ can be seen as a $d \times d$ matrix with entries in $\wedge^{n+1}V^{\vee}$. We also have an induced linear map $V^{\vee} \rightarrow \wedge^{n+2}V^{\vee} \otimes U \otimes U',  z\mapsto z\wedge \gamma$, that at the level of matrices can be realized by wedging every entry by $z$. It can be seen as a morphism  of vector bundles on $\mathbf{P}^r$:\footnote{The morphism that we denote by $\widetilde{\gamma}$ here is the transpose of the one denoted in the same way in \cite{kummershamovich}. That's because our focus here is on its cokernel, rather than on its kernel, as in \cite{kummershamovich}.} 
\[  \widetilde{\gamma}\colon \wedge^{n+2}V \otimes U'^{\vee} \otimes \sO_{\nP^r}(-1) \longrightarrow U\otimes \sO_{\nP^r}.  \]
One says that $\gamma$ is a \emph{Livsic-type determinantal representation} of $X \subseteq \nP^r$ if $X$ coincides with the locus where $\widetilde{\gamma}$ does not have maximal rank (see \cite[Definition 4.2]{kummershamovich}). In other words, $X$ is the support of $\operatorname{coker}(\widetilde{\gamma})$:
\[ X = \{ x \in \nP^r \,|\, \operatorname{coker} \widetilde{\gamma}(x) \ne 0 \} = \operatorname{Supp}(\operatorname{coker} \widetilde{\gamma} ).  \]

In this case, the \emph{degree} of $\gamma$ is defined as
$$
\deg (\gamma):=
\rank \big( \coker(\widetilde{\gamma}^t) \big) \cdot \deg(X).
$$
We say that $\gamma \in \wedge^{n+1} V^{\vee} \otimes U\otimes U'$ is an \emph{admissible determinantal representation} if $\deg(\gamma) = d$ (see \cite[Definition 4.3]{kummershamovich}).
Finally a determinantal representation is called \emph{symmetric} if $U'=U$ and if $\gamma \in \wedge^{n+1}V^{\vee}\otimes \operatorname{Sym}^2 U \subseteq \wedge^{n+1}V^{\vee}\otimes U\otimes U$. Equivalently, for any basis of $U$, then $\gamma$ can be written as a symmetric matrix with entries in $\wedge^{n+1}V^{\vee}$.

\begin{theorem}[{\cite[Theorem 4.7 and Proposition 4.11]{kummershamovich}}]\label{thm:detrep<=>ulrich}
There is an admissible determinantal representation of degree $d$ on $X$ if and only if there is an Ulrich sheaf of degree $d$ on $X$. More precisely, there is a one-to-one correspondence between isomorphism classes of Ulrich sheaves of degree $d$ on $X$  and similarity classes\footnote{For the definition of similarity between determinantal representations, see \cite[before Proposition 4.11]{kummershamovich}} of  admissible determinantal representations of degree $d$ on $X$.
\end{theorem}

The correspondence is explicit: if $\gamma \in \wedge^{n+1}V^{\vee} \otimes U\otimes U'$ is an admissible determinantal representation, then the corresponding Ulrich sheaf $\sF$ is the cokernel of $\widetilde{\gamma}$, so it fits in an exact sequence:
\[   \wedge^{n+2}V\otimes U^{\vee} \otimes \sO_{\nP^r}(-1) \xrightarrow{~\widetilde{\gamma}~} U'\otimes \sO_{\nP^r} \longrightarrow \mathscr{F} \longrightarrow 0. \]

\begin{remark}\label{rmk:chowform}
	Following \cite[Remark 4.4]{kummershamovich}, we explain how to construct a linear determinantal representation of the Chow form of $X$, assuming that  $\gamma$ is an admissible determinantal representation of degree $d$ with associated Ulrich sheaf of rank one. This means that $d=\deg(X)$. Notice that
	$$
	\wedge^{n+1} V^{\vee} \otimes U\otimes U' \cong \wedge^{r-n} V \otimes U\otimes U',
	$$
	and once we fix bases of $U,U'$, we may think of $\gamma$ as a matrix having linear forms on $\nP(\wedge^{r-n} V)$ as entries. Since $\deg (\det \gamma)= d = \deg X$, it follows from the arguments in \cite[Remark 4.4]{kummershamovich} that $\det \gamma$ is the Chow form of $X$. More precisely,  the determinantal hypersurface $\{\det \gamma  = 0 \}\subseteq \nP(\wedge^{r-n} V)$ is the Chow divisor of $X$ on the Pl\"{u}cker embedding of the Grassmannian  $\mathbf{G}(r-n-1, r)$ of $(r-n-1)$-planes in $\nP^r$, under the Pl\"ucker embedding $\mathbf{G}(r-n-1,r) \subseteq \nP (\wedge^{r-n} V)$.
\end{remark}

\subsection{Construction of admissible determinantal representations of rank one}

Let $X\subseteq \nP^r = \nP(V)$ be a nondegenerate variety of pure dimension $n$ and degree $\deg(X)=d$. We consider two sheaves $\sF,\sF'$ of rank one on $X$ such that $h^0(X,\sF) = h^0(X,\sF') = d$. We want to find an admissible determinantal representation of degree $d$ of the form
\[ \gamma\colon \wedge^{n+1}V \longrightarrow H^0(X\times X, \sF\boxtimes \sF') = H^0(X,\sF) \otimes H^0(X,\sF') \]
such that the corresponding Ulrich sheaf is $\sF$.
We start with some observations:

\begin{lemma}\label{lem:easymadecomplicated}
	With the previous notation, let  $\psi \in H^0(X,\sF)\otimes H^0(X,\sF') = H^0(X\times X,\mathscr{F}\boxtimes \mathscr{F}')$.
    \medskip

    \noindent
	$(1)$ If there are points $x_1,\dots,x_d\in X$ such that $\sF,\sF'$ are locally free at these points and also
		\[
		\psi(x_i,x_j) = \begin{cases} \text{nonzero} &\text{ if } i= j \\ \text{zero} &\text{ if } i\ne j  \end{cases},
		\]
		then the restrictions $s_i = \psi(-,x_i)$ and $s_i'=\psi(x_i,-)$ for $i=1,\dots,d$ are bases of $H^0(X,\sF)$ and  $H^0(X,\sF')$, respectively, such that
		\[
		s_i(x_j) = \begin{cases} \text{nonzero} &\text{ if } i= j \\ \text{zero} &\text{ if } i\ne j  \end{cases}, \qquad
		s'_i(x_j) = \begin{cases} \text{nonzero} &\text{ if } i= j \\ \text{zero} &\text{ if } i\ne j  \end{cases},
		\]
		and furthermore
		\[ \psi = \lambda_1\cdot s_1\otimes s_1' + \dots + \lambda_d \cdot s_d\otimes s_d' \quad \text{ for } \lambda_i\in \kk^{\times}. \]
		In particular,
		$\psi$ can be represented as a nonsingular diagonal matrix.
		\medskip

		\noindent
		$(2)$ In the hypotheses of the previous point, assume that $\mathscr{F} = \mathscr{F}'$ and $\psi$ is symmetric.  Then we can assume $s_i = s_i'$, so that
		\[ \psi = \lambda_1\cdot s_1\otimes s_1 + \dots + \lambda_d \cdot s_d\otimes s_d \quad \text{ for } \lambda_i\in \kk^{\times}. \]
		\noindent
		$(3)$ If there is a point $x\in X$ such that $\psi(x,x')=0$ for any $x'\in X$, $\sF$ is locally free around $x$, and $x$ is not a base point of $\mathscr{F}$, then any matrix representation of $\psi$ is singular.
\end{lemma}
\begin{proof}
$(1)$ For any $i=1,\dots,d$, define $s_i := \psi(-,x_i), s_i' := \psi(x_i,-)$ via the restriction of $\psi$ to $X\times \{x_i\}$ and $\{x_i\}\times X$ respectively. Then $s_i(x_j) = 0$ if and only if $x_j=x_i$, so that they are linearly independent sections, and since they are the right number, they form a basis of $H^0(X,\mathscr{F})$. The same holds for the $s_i'$ and then it is straightforward to show that $\psi = \sum_{h=1}^d \lambda_h s_h\otimes s_h'$ for certain $\lambda_h\in \kk^{\times}$.

\medskip

\noindent $(2)$ This is immediate from the proof of the previous point.

\medskip

\noindent $(3)$ Let $s_1,\dots,s_d$ and $s_1',\dots,s_d'$ be any bases of $H^0(X,\mathscr{F})$ and $H^0(X,\mathscr{F}')$ and write $\psi = \sum_{h,k} \lambda_{h,k} s_h\otimes s_k'$. By assumption, the restriction $\psi(x,-)$ to $\{x\}\times X$ is zero, meaning that $\psi = \sum_{h,k} \lambda_{h,k} s_h(x) s_k' = 0$.  Since the $s_k'$ are a basis, we can rewrite this in matrix form as
		\[ \begin{pmatrix} s_1(x) & \dots & s_d(x) \end{pmatrix} \cdot \begin{pmatrix} \lambda_{11} & \lambda_{12} & \dots & \lambda_{1d} \\
			\lambda_{21} & \lambda_{22} & \dots & \lambda_{2d} \\
			\vdots & \vdots & \ddots & \ddots \\
			\lambda_{d1} & \lambda_{d2} & \dots & \lambda_{dd} \\
		\end{pmatrix} = 0. \]
		Since $x$ is not a base point of $\mathscr{F}$, the vector $(s_1(x)\dots s_d(x))$ is nonzero, and then the matrix of the $\lambda_{ij}$ must be singular.
\end{proof}

The next technical result gives a way to obtain determinantal representations and Ulrich sheaves:

\begin{theorem}\label{thm:admrepresentation} 
Let $X\subseteq \nP^r$ be an irreducible nondegenerate projective variety of dimension $n$ as before and let $\sF,\sF'$ two rank one and torsion-free sheaves on $X$ with $h^0(X,\sF)=h^0(X,\sF')=d = \deg(X)$. Consider a linear map
\[ \gamma\colon \wedge^{n+1}V \longrightarrow H^0(X\times X,\sF\boxtimes \sF') = H^0(X,\sF)\otimes H^0(X,\sF')\]
with the following properties:
\begin{itemize}
	\item[(i)] If $p_0,\dots,p_n \in V$ have no common zero on $X$, then a general fiber of the projection $\pi = (p_0:\dots :p_n)\colon X\to \nP^n$ consists of $d$ distinct points $\{x_1,\dots,x_d\}$ such that
	\[ \gamma(p_0\wedge \dots \wedge p_n)(x_i,x_j) = \begin{cases} \text{ nonzero} & \text{if } i=j \\ \text{ zero} &\text{if } i\ne j \end{cases}.\]
	\item[(ii)] If $x\in X$ is a general point and if  $p_0,\dots,p_n\in V$ are general elements vanishing at $x$, then
	\[ \gamma(p_0\wedge \dots \wedge p_n)(x,-)\footnote{Here we see the restriction $\gamma(p_0\wedge \dots \wedge p_n)(x,-)$ to $\{x\}\times X$ as a section of $H^0(X,\sF')$} = 0 \quad \text{ in } H^0(X,\sF'). \]
	Furthermore,  a general fiber of the projection $\pi = (p_0:\dots:p_n)\colon X\dashrightarrow \nP^{n}$ consists of $d-1$ points $\{x_1,\dots,x_{d-1}\}$ such that
	\[ \gamma(p_0\wedge \dots \wedge p_n)(x_i,x_j) = \begin{cases} \text{ nonzero} & \text{if } i=j \\ \text{ zero} &\text{if } i\ne j \end{cases}.\]
\end{itemize}
Then $\gamma$ is an admissible determinantal representation of degree $d$ of $X$, and the corresponding Ulrich sheaf is the image of the evaluation map $H^0(X,\sF)\otimes \sO_X \to \sF$. In particular it coincides with $\sF$ if and only if $\sF$ is globally generated.
\end{theorem}

	To show that $\gamma$ is an admissible determinantal representation, we need to prove that the map of sheaves
	\begin{equation}\label{eq:gammatilda}
	\widetilde{\gamma}\colon \wedge^{n+2}V\otimes H^0(X,\sF')^{\vee} \otimes \sO_{\nP^r}(-1) \longrightarrow H^0(X,\sF)\otimes \sO_{\nP^r}
	\end{equation}
	fails to be surjective precisely on points $x\in X$. We start with an observation:

\begin{lemma}\label{lem:imagegammatilda}
With the notation of Theorem \ref{thm:admrepresentation}, let $x\in \nP^r$ be a point. Then the image of the map
	\begin{equation*}
	\widetilde{\gamma}_{|x}\colon \wedge^{n+2}V\otimes H^0(X,\sF')^{\vee} \otimes \sO_{\nP^r}(-1)_{|x} \longrightarrow H^0(X,\sF)
\end{equation*}
is spanned by the elements $\gamma(p_0\wedge \dots \wedge p_n)(-,x')$ for $p_0,\dots,p_n\in V$ vanishing on $x$ and $x'\in X$ a general point.
\end{lemma}
\begin{proof}
	If $x'\in X$ is a point where $\sF'$ is locally free, the evaluation $\operatorname{ev}_{x'}$  at $x'$ is a well-defined element of $H^0(X,\sF')^{\vee}$, up to a nonzero scalar multiple. These elements span $H^0(X,\sF')^{\vee}$ since $\sF'$ is torsion-free of rank one. If $x\in \nP^r$ is any point, then the map $\widetilde{\gamma}_{|x}$ at the point $x$ can be described as
	\begin{equation}\label{eq:mapgammatildaexplicit}
	\widetilde{\gamma}_{|x}\colon  \left( p_0\wedge \dots \wedge p_{n+1} \right) \otimes \operatorname{ev}_{x'} \longmapsto \sum_{i=0}^{n+1}(-1)^i \gamma(p_0 \wedge \dots \wedge \hat{p_i} \wedge \dots \wedge p_{n+1})(-,x') \cdot p_i(x).
	\end{equation}
  Observe that the expression \eqref{eq:mapgammatildaexplicit} vanishes if $p_0(x) = \dots = p_{n+1}(x)=0$, while if $p_0(x)=\dots=p_{n}(x)=0$ and $p_{n+1}(x)\ne 0$, then it reduces, up to a nonzero scalar,  to
	\begin{equation}\label{eq:mapgammatildaexplicit2}
		\widetilde{\gamma}_{|x}\colon  \left( p_0\wedge \dots \wedge p_{n+1} \right) \otimes \operatorname{ev}_{x'} \longmapsto (-1)^{n+1} \gamma(p_0 \wedge \dots \wedge p_{n})(-,x').
	\end{equation}
Since we can find a basis of $V$ whose  elements all vanish on $x$ apart from one, the statement follows.
\end{proof}

Now we can prove Theorem \ref{thm:admrepresentation}.

\begin{proof}[Proof of Theorem \ref{thm:admrepresentation}]
	 We first show that $\gamma$ induces an admissible determinantal representation: this statement is invariant under field extensions, so we can assume that $\kk$ is algebraically closed. First, let $x\in \nP^r\setminus X$, and take linearly independent $p_0,\dots,p_n \in V$ vanishing on $x$ but with no common zero on $X$. Consider the projection $\pi = (p_0:\dots:p_n) \colon X\to \nP^r$ and let $\{x_1,\dots,x_d\}$ be a general fiber. By assumption (i) and Lemma \ref{lem:easymadecomplicated}, the sections $s_i = \gamma(p_0\wedge \dots \wedge p_n)(-,x_i)$ are a basis of $H^0(X,\sF)$, and then Lemma \ref{lem:imagegammatilda} shows that the map $\widetilde{\gamma}$ is surjective at $x$.
	\medskip

	\noindent
	Assume now that $x \in X$ is a general point: then assumption (ii) and Lemma \ref{lem:imagegammatilda} show that the image of $\widetilde{\gamma}$ at $x$ is contained in the subspace $H^0(X,\sF\otimes \mathfrak{m}_x) \subseteq H^0(X,\sF)$ of sections vanishing at $x$. Notice that since $x\in X$ is general, the fact that the image of $\widetilde{\gamma}_{|x}$ is contained in  $H^0(X,\sF\otimes \mathfrak{m}_x)$ actually holds for every $x\in X$, since this is a closed condition. Furthermore, if $x$ is a general point, $H^0(X,\sF\otimes \mathfrak{m}_x)$ is a general subspace of $H^0(X,\sF)$ of dimension $d-1$. Since the $p_0,\dots,p_n$ are general in the space of sections vanishing at $x$, we can assume that they vanish on $X$ only at $x$.  Then, using the second part of assumption (ii) we can see, reasoning as before, that the image of $\widetilde{\gamma}$ at $x$ is precisely $H^0(X,\sF\otimes \mathfrak{m}_x)$.
	\medskip

	\noindent
	The previous discussion proves that the cokernel $\sG$ of $\widetilde{\gamma}$ is a sheaf supported on $X$ of rank $1$. Hence, $\gamma$ is an admissible determinantal representation, and $\sG$ is an Ulrich sheaf. Furthermore, the previous discussion also proves that the image of $\widetilde{\gamma}$ is contained in the kernel of the  evaluation map $H^0(X,\sF)\otimes \sO \to \sF$ (this statement is also invariant under field extensions). In particular, if $\widetilde{\sF}$ is the image of the evaluation map $H^0(X,\sF)\otimes \sO_X \to \sF$, we see that  $\widetilde{\sF}$ is a quotient of $\sG$. This means that there is a surjective map of sheaves $\varphi\colon \sG \to \widetilde{\sF}$ that is an isomorphism at a general point of $X$, since both $\sG$ and $\widetilde{\sF}$ are of rank one. This shows that the kernel of $\sG \to \widetilde{\sF}$ is a torsion sheaf, but since $\sG$ is Ulrich, it is Cohen--Macaulay with support equal to the reduced variety $X$, so that it is torsion-free and then the kernel must be zero.
	\end{proof}

\section{Symmetric products, Szeg\H{o} kernels, and secant varieties of curves}\label{sec:prelimsymmsec}
\noindent In this section, we set up basic notation and review basic results of symmetric products and secant varieties of curves following  \cite{agostini, ENP, ENP3, NP}, to which we refer  for more details. We also discuss the existence, uniqueness, and basic properties of higher Szeg\H{o} kernels associated with a line bundle with no cohomology, which are essential for constructing admissible determinantal representations of secant varieties of curves.

\subsection{Symmetric products of curves}\label{subsec:symcur} 
Recall that $C$ is a smooth projective curve of genus $g$ over an algebraically closed field $\kk$ of characteristic zero. For an integer $k\geq 0$, we write $C_{k+1}$ for the $(k+1)$-th symmetric product of $C$:
\[ C_{k+1} = \{ x_1+\dots+x_{k+1} \,|\, x_i \in C \}. \]
This parametrizes effective divisors (or finite subschemes) $\xi\subseteq C$ of degree $k+1$.  For integers $r,s \geq 1$, we define the addition map
$$
\sigma_{r,s} \colon C_r \times C_s \longrightarrow C_{r+s}, \quad (\xi,\eta) \longmapsto \xi+\eta,
$$
which is a finite flat morphism. We also define the map
$$
j_{r-1,s-1}\colon C\times C_{r-1} \times C_{s-1} \longrightarrow C_r\times C_s, \quad (x,\xi,\eta) \longmapsto (x+\xi,x+\eta),
$$
which is finite and birational onto the image $D_{r,s}$:
$$
D_{r,s} = \{ (\xi,\eta) \in C_r\times C_s \,|\, \Supp(\xi) \cap \Supp(\eta) \ne \emptyset \}.
$$
The subscheme $D_{r,s}\subseteq D_r\times D_s$ is a divisor and  $D_{1,k+1} \subseteq C\times C_{k+1}$, together with the map ${\pr_{C_{k+1}}}_{|D_{1,k+1}}\colon D_{1,k+1} \to C_{k+1}$ is the universal divisor over $C_{k+1}$. Since the map $j_{1,k+1}\colon C\times C_k \to D_{1,k+1}$ is an isomorphism, we can also consider the composition $\pr_{C_{k+1}}\circ j_{0,k} = \sigma_{1,k}\colon C\times C_{k} \to C_{k+1}$ to be the universal divisor over $C_{k+1}$. In general,  for
integers $m_1, \ldots, m_t \geq 1$ we get divisors in $C_{m_1}\times \dots \times C_{m_t}$ by
$$
D_{0, \ldots, 0, m_i, 0, \ldots, 0, m_j , 0, \ldots, 0}:=\{(\xi_1, \ldots, \xi_n) \in C_{m_1} \times \cdots \times C_{m_n} \mid \Supp(\xi_i) \cap \Supp(\xi_j) \neq \emptyset\}
$$
For a line bundle $L$ on $C$, we set
$$
E_{k+1,L}:= \pr_{C_{k+1},*}((\sO_{C_{k+1}}\boxtimes L) \otimes \sO_{D_{1,k+1}}) = \sigma_{k,1,*} (\sO_{C_k} \boxtimes L),
$$
which is a vector bundle of rank $k+1$ on $C_{k+1}$, whose fiber ${E_{k+1,L}}_{|\xi}$ at a point $
\xi\in C_{k+1}$ is identified with $H^0(\xi, L|_{\xi})$. Recall that a linear system $V\subseteq H^0(C,L)$  \emph{separates} $\xi\in C_{k+1}$ if the evaluation map $V\to H^0(\xi,L_{|\xi})$ is surjective. Then $V$ is called $k$-very ample if it separates all $\xi \in C_{k+1}$. In this case, we have an exact sequence
\[ 0 \longrightarrow M_{k+1,V} \longrightarrow V\otimes \sO_{C_{k+1}} \longrightarrow E_{k+1,L} \longrightarrow 0 \]
that defines a bundle $M_{k+1,V}$. If $V=H^0(C,L)$ is the complete linear system, we write $M_{k+1,L}:= M_{k+1,H^0(C,L)}$. Furthermore, $H^0(C,L) \cong H^0(C_{k+1},E_{k+1,L})$ so that $E_{k+1,L}$ is globally generated if and only if the $L$ is $k$-very ample (meaning that the complete linear system of $L$ is). We denote by $\delta_{k+1}$ the divisor class on $C_{k+1}$  such that $\sO_{C_{k+1}}(-\delta_{k+1}) = \det E_{k+1, \sO_C}$. One has that
$$
2\delta_{k+1} \sim  \{ \xi \in C_{k+1} \mid \text{$\xi$ is nonreduced} \}.
$$
Consider the quotient map $q_{k+1} \colon C^{k+1} \to C_{k+1}=C^{k+1}/\mathfrak{S}_{k+1}$, where $\mathfrak{S}_{k+1}$ is the symmetric group naturally acting on $C^{k+1}$. The $(k+1)$-fold box product $L^{\boxtimes (k+1)}$ has a natural action of $\mathfrak{S}_{k+1}$ and we obtain a line bundle on $C_{k+1}$ by taking invariant descent
\[ S_{k+1,L} := q_{k+1,*}^{\mathfrak{S}_{k+1}}(L^{\boxtimes (k+1)}). \]
One has that $S_{k+1,L}\otimes S_{k+1,M} \cong S_{k+1,L\otimes M}$. 
If instead we fix $\eta\in C_s$ and we consider $C_{k+1}\times \{\eta\} \subseteq C_{k+1}\times C_s$, then we see that
\[ \sO_{C_{k+1} \times C_{k+1}} (D_{k+1,s})_{|C_{k+1}\times \{\eta\}} \cong S_{k+1,\sO_C(\eta)}. \]
We further define two more line bundles
$$
N_{k+1, L}:=S_{k+1,L}(-\delta_{k+1})~~\text{ and }~~A_{k+1,L}:=S_{k+1,L}(-2\delta_{k+1})
$$
on $C_{k+1}$. It holds that $N_{k+1,L} \cong \det E_{k+1,L}$ and $N_{k+1,\omega_C} \cong \omega_{C_{k+1}}$.
We can compute the pullbacks of the various line bundles along the addition map:

\begin{lemma}[{\cite[Lemma 3.1]{agostini}, \cite[Lemma 2.3]{NP}}]\label{lem:pullbackdeltaviaaddmap}
For the addition map $\sigma_{r,s} \colon C_r \times C_s \to C_{r+s}$, and any line bundle $L$ on $C$ it holds that
\[
\sigma_{r,s}^*S_{r+s,L} \cong S_{r,L}\boxtimes S_{s,L}, \quad
\sigma_{r,s}^*N_{r+s,L} \cong (N_{r,L}\boxtimes N_{s,L})(-D_{r,s}).
\]
In particular,
\[
 \sigma_{r,s}^*\sO_{C_{r+s}}(\delta_{r+s})  \cong (\sO_{C_r}(\delta_r) \boxtimes \sO_{C_s}(\delta_s)) \otimes \sO_{C_r \times C_s}(D_{r,s}) .
\]
\end{lemma}

Another technical statement that we will need is about the restriction of the divisor $D_{k+1,k+1} \subseteq C_{k+1}\times C_{k+1}$ to the diagonal $\Delta_{C_{k+1}} \subseteq C_{k+1}\times C_{k+1}$.

\begin{lemma}\label{lem:D|_Delta}
	Under the diagonal embedding $\Delta_{C_{k+1}}\colon C_{k+1} \hookrightarrow C_{k+1}\times C_{k+1}$ it holds that
	$$
   \sO_{C_{k+1} \times C_{k+1}} (D_{k+1,k+1})_{|\Delta_{C_{k+1}}} \cong \Delta^*_{C_{k+1}}\sO_{C_{k+1} \times C_{k+1}} (D_{k+1,k+1}) \cong A^{-1}_{k+1, \omega_C}.
	$$
\end{lemma}

\begin{proof}
	We proceed by induction on $k$. If $k=0$, then  $D_{1,1}=\Delta_C$. By the adjunction formula
	$
	\omega_{C} \cong \omega_{C \times C}(D_{1,1})|_{\Delta_C} \cong \omega_C^2(D_{1,1}|_{\Delta_C}),
	$
	so we get $\sO_C(D_{1,1}|_{\Delta_C}) \cong \omega_C^{-1}$ as desired. Assume next that $k \geq 1$ and consider the commutative diagram
	\[
	\xymatrix{
		C_k \times C \cong \Delta_{C_k \times C} \ar@{^{(}->}[r] \ar[d]_-{\sigma_{k,1}} & C_k \times C \times C_k \times C  \ar[d]^-{\sigma_{k,1} \times \sigma_{k,1}} \\
		\Delta_{C_{k+1}} \ar@{^{(}->}[r] & C_{k+1} \times C_{k+1},
	}
	\]
	where the horizontal maps are closed embeddings. It is enough to prove that
	$$
	\sO_{C_k \times C}((\sigma_{k+1} \times \sigma_{k+1})^* D_{k+1,k+1} )|_{\Delta_{C_k \times C}} \cong \sigma_{k,1}^* A^{-1}_{k+1, \omega_C} .
	$$
	Notice that
	$
	(\sigma_{k+1} \times \sigma_{k+1})^* D_{k+1,k+1} = D_{k,0,k,0} + D_{k,0,0,1} + D_{0,1,k,0}+D_{0,1,0,1}.
	$
	We have
	$
	D_{k,0,0,1}|_{\Delta_{C_k \times C}} =  D_{0,1,k,0}|_{\Delta_{C_k \times C}} = D_{k,1}.
	$
	and by induction, we find
	$\sO_{C_k \times C}(D_{k,0,k,0}|_{\Delta_{C_k \times C}}) \cong A^{-1}_{k, \omega_C} \boxtimes \sO_C$ and $\sO_{C_k \times C}(D_{0,1,0,1}|_{\Delta_{C_k \times C}}) \cong \sO_{C_k} \boxtimes \omega_C^{-1}.
	$
	Thus
	\[
	\sO_{C_k \times C}((\sigma_{k+1}\times \sigma_{k+1})^* D_{k+1,k+1} )|_{\Delta_{C_k \times C}} \cong (A^{-1}_{k, \omega_C} \boxtimes \omega_C^{-1})(2D_{k,1}) \cong \sigma_{k,1}^* A^{-1}_{k+1, \omega_C}. \qedhere
	\]
\end{proof}

The cohomology groups of the line bundles $S_{k+1,L}$ and $N_{k+1,L}$ can be computed explicitly:

\begin{lemma}[{\cite[Lemma 2.4]{agostini}, \cite[Lemma 3.7]{ENP}}]\label{lem:H^i(N)}
	Let $L$ be any line bundle on $C$. For any $i \geq 0$, it holds that
	\begin{align*}
		&H^i(C_{k+1}, N_{k+1,L}) \cong \wedge^{k+1-i} H^0(C, L) \otimes S^i H^1(C, L)\\
		&H^i(C_{k+1}, S_{k+1,L}) \cong S^{k+1-i} H^0(C, L) \otimes \wedge^i H^1(C, L).
	\end{align*}
\end{lemma}

The isomorphism $\det{}_{k+1,L}\colon \wedge^{k+1}H^0(C,L) \xrightarrow{\simeq} H^0(C_{k+1},N_{k+1,L})$
is realized by taking global sections in the natural map $\wedge^{k+1}H^0(C,L) \otimes \sO_{C_{k+1}} \to \wedge^{k+1}E_{k+1,L} \cong N_{k+1,L}$. In particular we have the following lemma. 

\begin{lemma}\label{lem:A(p_1...p_m)(D)=0}
	Let $p_0, \ldots, p_{k} \in H^0(C, L)$, let $W = \langle p_0,\dots,p_{k} \rangle$ the space that they generate,  and $\xi \in C_{k+1}$. Then the following are equivalent:
	\begin{enumerate}
		\item $\xi$ is not a base point of $\wedge^{k+1}W\subseteq H^0(C_{k+1},N_{k+1,L})$.
		\item  $\det_{k+1,L}(p_0 \wedge \cdots \wedge p_{k})(\xi) \neq 0$.
		\item $s_1|_{\xi}, \ldots, s_{k+1}|_{\xi}$ is a basis of $H^0(\xi, L|_{\xi})$.
		\item $W$ separates $\xi$.
		\item The evaluation map $W\to H^0(\xi,L_{|\xi})$ is surjective (and then an isomorphism).
		\end{enumerate}
\end{lemma}

We will also need a result, essentially due to Rathmann \cite{Rathmann}, for the cohomology of $A_{k+1,L}$:

\begin{proposition}\label{prop:rathmannvanishing}
	Let $B$ be a $k$-very ample line bundle on $C$ such that $H^1(C,L\otimes B^{-1})=0$. Then $H^i(C_{k+1},\wedge^j M_{k+1,B}\otimes A_{k+1,L}) = 0$ for all $i>0,j\geq 0$.
\end{proposition}
\begin{proof}
	By \cite[Lemma 3.5]{ENP}, this can be deduced from
	\begin{equation}\label{eq:rathmann}
		H^i(C^{k+1}, \wedge^j q_{k+1}^* M_{k+1,B} \otimes L^{\boxtimes k+1} (-\Delta_{k+1}))=0~~\text{ for $i > 0$ and $j \geq 0$,}
	\end{equation}
	where $q_{k+1} \colon C^{k+1} \to C_{k+1}$ is the quotient map and $\Delta_{k+1}$ is the sum of all pairwise diagonals of $C^{k+1}$. Observe that since $B$ is $k$-very ample, it is effective, so that $H^1(C,L\otimes B^{-1})=0$ implies $H^1(C,L)=0$ as well. Then the desired cohomology vanishing (\ref{eq:rathmann}) follows from \cite[Theorem 3.1]{Rathmann}\footnote{It is only claimed in \cite[Theorem 3.1]{Rathmann} that the cohomology vanishing (\ref{eq:rathmann}) holds for $i>0$ and $j>0$. But the proof actually also covers the case of $j=0$.} (see also \cite[Theorem 4.1]{ENP}).
\end{proof}

\subsubsection{Line bundles from the Jacobian}
We want to connect the line bundles on the symmetric products introduced before with those arising from the  Jacobian.
Fix an effective divisor  $\xi'\in C_{k+1}$ and consider the Abel--Jacobi map
\[ u_{-\xi'} \colon C_{k+1} \longrightarrow \operatorname{Pic}_0(C), \quad  \xi\longmapsto u(\xi-\xi').\]
For any line bundle $\alpha\in \Pic_{g-1}(C)$ we have the corresponding theta divisor
\[ \Theta_{\alpha} = \{\eta\in \Pic_0(C) \,|\, h^0(C,\eta\otimes\alpha) > 0\}. \]
This induces a principal polarization, so that the map
\[ \Pic_0(C) \longrightarrow \Pic_0(\Pic_0(C)), \quad \eta \longmapsto \widehat{\eta} := \sO_{\Pic_0(C)}(\Theta_{\alpha\otimes\eta} - \Theta_{\alpha})\]
is an isomorphism, which is moreover independent of $\alpha$.  We  define 
\[ \alpha' := \omega_C\otimes \alpha^{-1} \]
so that $h^0(C,\alpha')=h^1(C,\alpha')$ by Riemann--Roch and  $\alpha\cong \alpha'$ precisely when $\alpha$ is a theta characteristic.

\begin{lemma}\label{lem:alphamap1}
	Assume that $h^0(C,\alpha)=0$. For any $\xi',\xi\in C_{k+1}$ it holds that $h^0(C,\alpha(\xi'))=k+1, h^1(C,\alpha(\xi'))=0$ and the evaluation map
	\[ H^0(C_{k+1},N_{k+1,\alpha(\xi')}) \longrightarrow H^0(C_{k+1},N_{k+1,\alpha(\xi')} \otimes \kappa(\xi))\]
	is nonzero  (hence an isomorphism) if and only if $h^0(C,\alpha(\xi'-\xi))=0$.
\end{lemma}
\begin{proof}
	The vanishing $h^1(C,\alpha(\xi'))=0$ follows from $h^1(C,\alpha)=0$, and then $h^0(C,\alpha(\xi'))=k+1$ is a consequence of Riemann--Roch. Since $H^0(C_{k+1},N_{k+1,\alpha(\xi')})\cong \wedge^{k+1}H^0(C,\alpha(\xi')) \cong \kk$, the evaluation map is an isomorphism if and only if it is nonzero.  By Lemma \ref{lem:A(p_1...p_m)(D)=0} this is equivalent to the evaluation map on $C$ itself $H^0(C,\alpha(\xi')) \to H^0(C,\alpha(\xi')\otimes \sO_{\xi})$ being an isomorphism. This last statement is equivalent to $h^0(C,\alpha(\xi'-\xi))=0$.
\end{proof}

We can now compute the effect of the Abel--Jacobi maps on line bundles:

\begin{proposition}\label{prop:pullback_theta}
	With the previous notation, it holds that
	\[ u_{-\xi'}^*\sO_{C_{k+1}}({\Theta_{\alpha}}) \cong N_{k+1,\alpha'(\xi')}  \quad \text{ and } \quad u_{-\xi'}^*(\widehat{\eta}) \cong S_{k+1,\eta^{-1}} .\]
\end{proposition}
\begin{proof}
	Assume first that the first isomorphism holds for  $h^0(C,\alpha)=h^0(C,\alpha')=0$. Then  for any $\eta\in \Pic_0(C)$, we have 
	\[ u_{-\xi}^*(\widehat{\eta}) \cong N_{k+1,\alpha'\otimes\eta^{-1}(\xi')}\otimes N_{k+1,\alpha'(\xi')}^{-1} \cong S_{k+1,\eta^{-1}.} \]
	Furthermore, if $h^0(C,\alpha)\ne 0$ we can find $\eta\in \Pic_0(C)$ such that $h^0(C,\alpha\otimes\eta)=0$ and then $\sO_{C_{k+1}}(\Theta_{\alpha\otimes\eta})\cong \sO_{C_{k+1}}(\Theta_{\alpha})\otimes \widehat{\eta}$ and what we already proved shows that
	\[  N_{k+1,\alpha'\otimes\eta^{-1}(\xi')} \cong u^*_{-\xi}\sO_{C_{k+1}}(\Theta_{\alpha})\otimes S_{k+1,\eta^{-1}} .\]
	Hence it is enough to prove the first isomorphism when $h^0(C,\alpha)=0$.
	The support of $u^*_{-\xi}(\Theta_{\alpha})$ is equal to the set
	\[  \{ \xi\in C_{k+1} \,|\, h^0(C,\alpha(\xi-\xi')) > 0 \} = \{ \xi \in C_{k+1} \,|\, h^0(C,\alpha'(\xi'-\xi))>0\}, \]
	where the second equality is given by Riemann--Roch. By Lemma \ref{lem:alphamap1}, this is the same support as that of the unique divisor  in the linear system of the line bundle $N_{k+1,\alpha'(\xi')}$. Furthermore, the two classes  $[u^*_{-\xi'}\Theta_{\alpha}]$ and $[N_{k+1,\alpha'(\xi')}]$ in the Neron-Severi group $\operatorname{NS}(C_{k+1})$ coincide because of \cite[Lemma VIII.3.2]{ACGH}. In particular, if the support above is irreducible, then it must be that $u_{-\xi'}^*\sO_{C_{k+1}}(\Theta_{\alpha}) \cong N_{k+1,\alpha'(\xi')}$. This is true if $k+1\geq 2g-1$ because then the Abel--Jacobi map $u_{-\xi'}\colon C_{k+1}\rightarrow \Pic_0(C)$ is a projective bundle and $\Theta_{\alpha}$ is irreducible divisor. For an arbitrary $k$, fix an effective divisor $E$ of degree $e\geq 2g-1$ and consider the embedding $t_{E}\colon C_{k+1} \hookrightarrow C_{k+1+e},t_{E}(\xi)=\xi+E$. The map $u_{-\xi}$ is equal to the composition
	\[ C_{k+1} \overset{t_{E}}{\longrightarrow} C_{k+1+e} \overset{u_{-\xi'-E}}{\longrightarrow} \Pic_0(C).  \]
	By the previous discussion, what we want to prove is that
	$t_E^*N_{k+1+e,\alpha'(\xi'+E)}\cong N_{k+1,\alpha'(\xi')}$ but this follows from example from Lemma \ref{lem:pullbackdeltaviaaddmap}, since the map $t_E$ can also be seen as the restriction to $C_{k+1}\times\{E\}$ of the summation map $\sigma_{k+1,e}\colon C_{k+1}\times C_e\to C_{k+1+e}$.
\end{proof}

\subsection{First order tautological bundles and higher Gauss maps}

Consider the universal divisor $D_{k+1,1}\subseteq C_{k+1}\times C$ and let $2D_{k+1,1}$ be the divisor corresponding to the square of the ideal sheaf of  $D_{k+1,1}$. The \emph{first order tautological bundle} associated to $L$ is
\[ E^{(1)}_{k+1,L} := \pr_{C_{k+1},*}((\sO_{C_{k+1}} \boxtimes L) \otimes \sO_{2D_{k+1,1}} ). \]
This is a bundle of rank $2k+2$ and the fiber at $\xi\in C_{k+1}$ is given by $E^{(1)}_{k+1,L} \otimes \kappa(\xi) \cong H^0(C,L\otimes \sO_{2\xi})$.

\begin{remark}\label{rem:jetbundle}
	The vector bundle $E^{(1)}_{1,L}$ on $C$ is usually called the first jet bundle associated to $L$. It always fits into an exact sequence
	\[ 0 \longrightarrow \omega_C\otimes L \longrightarrow E^{(1)}_{1,L} \longrightarrow L \longrightarrow 0 \]
	and if $L=\sO_C$ the sequence splits so that $E^{(1)}_{1,\sO_C} \cong \sO_C\oplus \omega_C$. Assume now that $L$ is an arbitrary line bundle, let $U\subseteq C$ be an open subset and $S$ a rational section of $L$ which is regular and never zero on $U$. Then $S$ induces an isomorphism $L_{|U}\cong \sO_U$ and ${E^{(1)}_{1,L}}_{|U} \cong E^{(1)}_{1,\sO_U} \cong \sO_U \oplus \omega_U$. In concrete terms, any section of $H^0(U,E^{(1)}_{1,L})$ can be identified with a pair $(f\cdot S,\eta\cdot S)$, where $f \in H^0(U,\sO_C)$ and $\eta \in H^0(U,\omega_U)$.  In particular, if $p\in H^0(U,L)$ is a section, we can write $p=f\cdot S$ for an $f\in H^0(U,\sO_U)$ and then the total differential gives an element $df\in H^0(U,\omega_U)$. This yields an element $(f\cdot S,df\cdot S) \in H^0(U,E^{(1)}_{1,L})$. This construction globalizes on $C$ to yield a natural evaluation map
	\[ \operatorname{ev}^{(1)}_{1,L}\colon  H^0(C,L)\otimes \sO_C \longrightarrow E^{(1)}_{1,L}, \]
	which is defined as before on open subsets where $L$ is trivial.
\end{remark}

The last evaluation map extends to higher symmetric products:
\begin{equation}\label{eq:evmap2}
\operatorname{ev}^{(1)}_{k+1,L}\colon H^0(C,L)\otimes \sO_{C_{k+1}} \longrightarrow E^{(1)}_{k+1,L}
\end{equation}
that we can describe explicitly around a point  $\xi=x_1+\dots+x_{k+1} \in C_{k+1}$ with the $x_i$ pairwise distinct, using the arguments of Remark \ref{rem:jetbundle}.  Choose for each $x_i$ a rational section $S_i$ of $L$ that is regular and nonzero on a neighborhood $U_i$ of $x_i$. Any section of $H^0(C,L)$ can be written as  $p = f_i\cdot S_i$ for certain rational functions $f_i$, regular on $U_i$. The total differential $df_i$ is a rational section of $\omega_C$, regular on $U_i$, so that $df_i\cdot S_i$ is a rational section of $\omega_C\otimes L$, regular on $U_i$.
 Then the evaluation map \eqref{eq:evmap2} can be described in an appropriate neighborhood of $x_1+\dots+x_{k+1}$ as
\[ \operatorname{ev}^{(1)}_{k+1,L}(p) = \begin{pmatrix} f_j\cdot S_j & df_j\cdot S_j \end{pmatrix}_{1\leq j \leq k+1} = (f_1 \cdot S_1 \,\, df_1 \cdot S_1 \,\, \ldots \,\, f_{k+1}\cdot S_{k+1} \,\, df_{k+1}\cdot S_{k+1}) . \]
In particular, if  $x_1'+\dots+x_{k+1}'$ is in a neighborhood of $x_1+\dots+x_{k+1}$, then
\[ \operatorname{ev}^{(1)}_{k+1,L}(p)(x'_1+\dots+x'_{k+1}) = \begin{pmatrix} f_j(x'_j)S_j(x'_j) & df_j(x'_j)S_j(x'_j) \end{pmatrix}_{1\leq j \leq k+1},  \]
where $f_j(x'_j)S_j(x'_j) \in H^0(C,L\otimes \kappa(x'_j))$ and $df_j(x'_j)S_j(x_j') \in H^0(C,L\otimes \omega_C\otimes \kappa(x'_i))$ and
\begin{align*}
E^{(1)}_{k+1,L}\otimes \kappa(\xi) &= H^0(C,L\otimes \sO_{2\xi}) = \bigoplus_{j=1}^{k+1} E^{(1)}_{1,L} \otimes \kappa(x_j) \\
&\cong  \bigoplus_{j=1}^{k+1} H^0(C,L\otimes \kappa(x_j)) \oplus H^0(C,\omega_C\otimes L\otimes \kappa(x_i)).
\end{align*}

Consider now the \emph{doubling map}:
\begin{equation}\label{eq:doublingmap}
	\operatorname{db}_{k+1}\colon C_{k+1} \longrightarrow C_{2k+2}, \quad \xi \longmapsto 2\xi, 
\end{equation}
which is also the composition of the two maps
\[ \Delta_{C_{k+1}}\colon C_{k+1} \longrightarrow C_{k+1}\times C_{k+1} \quad \text{ and } \quad \sigma_{k+1,k+1}\colon C_{k+1}\times C_{k+1} \longrightarrow C_{2k+2} \]

\begin{lemma}\label{lem:doublingpullback}
	The first order tautological bundle is the pullback of the usual tautological bundle under the doubling map
	\[ E^{(1)}_{k+1,L} \cong \operatorname{db}_{k+1}^*E_{2k+2,L}. \]
\end{lemma}
\begin{proof}
	Consider the map $\id_C\times \operatorname{db}_{k+1}\colon C\times C_{k+1} \to C\times C_{2k+2}$. We claim that
	\begin{equation}\label{eq:pullbackdoubling}
	(\id_C\times \operatorname{db}_{k+1})^*\sO_{D_{1,2k+2}} = \sO_{2D_{1,k+1}}.
	\end{equation}
	 For $i=1,\dots,k+1$ choose points $x_i\in C$ and corresponding local coordinates $z_i$. Points of $C_{k+1}$ in a neighborhood of $\xi = x_1+\dots+x_{k+1}$ can be identified with polynomials  of degree $k+1$:
	\[P_1(t) = (t-z_1) \cdots (t-z_{k+1}).  \]
	 Take now a copy $z_i'$ for each local coordinate $z_i$. Then points of $C_{2k+2}$ in a neighborhood of $2\xi$ can be identified with  polynomials of degree $2k+2$:
	\[ P_2(t) = (t-z_1)(t-z_1') \cdots (t-z_{k+1})(t-z_{k+1}'). \]
	The doubling map is then given in terms of these as $\operatorname{db}\colon P_1(t)\mapsto P_1(t)^2$. Choose now another point $x_0\in C$ and a local coordinate $z_0$ around it. The universal divisor $D_{1,2k+2}$ in $C\times C_{2k+2}$ is given in a neighborhood of $(x_0,2\xi)$ as $D_{1,2k+2} = \{ (z_0,P_2(t)) \,|\, P_2(z_0)=0 \}$. Hence around $(x_0,\xi)\in C\times C_{k+1}$ we see that
	\[ (\id_C\times \operatorname{db})^*D_{1,2k+2} = \{(z_0,P_1(t)) \,|\, P_1(z_0)^2 = 0 \} = 2D_{1,k+1} \]
	and this proves the claim \eqref{eq:pullbackdoubling}. Consider now the two cartesian diagrams, where the second is a consequence of  \eqref{eq:pullbackdoubling}:
	\begin{equation}
		\begin{tikzcd}
			C\times C_{k+1} \arrow{r}{\id_C\times \operatorname{db}_{k+1}} \arrow[swap]{d}{\pr_{C_{k+1}}}& C\times C_{2k+2} \arrow{d}{\pr_{C_{2k+2}}} \\
			C_{k+1} \arrow{r}[below]{\operatorname{db}_{k+1}} & C_{2k+2}
		\end{tikzcd}
		\quad
		\begin{tikzcd}
			2D_{1,k+1} \arrow{r}{\id_C\times \operatorname{db}_{k+1}} \arrow[swap]{d}{\pr_{C_{k+1}}}& D_{1,2k+2} \arrow{d}{\pr_{C_{2k+2}}} \\
			C_{k+1} \arrow{r}[below]{\operatorname{db}_{k+1}} & C_{2k+2}.
		\end{tikzcd}
	\end{equation}
Base  change with respect to the second diagram gives a natural map $\operatorname{db}_{k+1}^*E_{2k+2,L}\to E^{(1)}_{k+1,L}$ which is furthermore  an isomorphism, since all maps involved in the diagram are finite morphisms, hence affine. This is what we wanted to prove.
 \end{proof}

In particular this allows us to compute the determinant of the first order tautological bundles:

\begin{corollary}
	There is an isomorphism $\wedge^{2k+2} E^{(1)}_{k+1,L} \cong A_{k+1,L}^{\otimes 2}\otimes S_{k+1,\omega_C}$.
\end{corollary}
\begin{proof}
	Taking the $(2k+2)$-alternating power in the isomorphism of  \ref{lem:doublingpullback}, it is enough to compute $ \operatorname{db}_{k+1}^*N_{2k+2,L}$. If we write $\operatorname{db}_{k+1} = \sigma_{k+1,k+1}\circ \Delta_{C_{k+1}}$, we see from Lemma \ref{lem:pullbackdeltaviaaddmap} that $\sigma_{k+1,k+1}^*N_{2k+2,L} \cong (N_{k+1,L}\boxtimes N_{k+1,L})(-D_{k+1,k+1})$. Then Lemma \ref{lem:D|_Delta} shows that
	\[\Delta_{C_{k+1}}^* (N_{k+1,L}\boxtimes N_{k+1,L})(-D_{k+1,k+1}) \cong (N_{k+1,L}\otimes N_{k+1,L})\otimes A_{k+1,\omega_C} \cong A_{k+1,L}^{\otimes 2} \otimes S_{k+1, \omega_C}. \] \qedhere
\end{proof}

In view of this, if we apply $\wedge^{2k+2}$ to the evaluation map \eqref{eq:evmap2} and take global sections, we obtain a map, that we call \emph{(k+1)-Gaussian map}:
\begin{equation}\label{eq:k+1gaussianmap}
G_{k+1,L}\colon \wedge^{2k+2} H^0(C,L) \longrightarrow H^0(C_{k+1},A_{k+1,L}^{\otimes 2}\otimes S_{k+1,\omega_C}).
 \end{equation}
 We describe this explicitly around a point $\xi=x_1+\dots+x_{k+1} \in C_{k+1}$ with the $x_i$ mutually distinct.  Choose for any $j$ a rational section $S_j$ of $L$ which is regular and nonzero around  $x_j$. We then take $p_0,\dots,p_{2k+1}\in H^0(C,L)$ and write $p_i = f_{ij} S_j$ for certain rational functions $f_{ij}$, regular around $x_j$. Then the previous discussion shows that around $x_1+\dots+x_{k+1}$ we can write
\begin{align}
&G_{k+1,L}(p_0\wedge \dots \wedge p_{2k+1}) = \det \begin{pmatrix} f_{ij}\cdot S_j & df_{ij}\cdot S_j \end{pmatrix}_{\substack{0\leq i\leq 2k+1,\\ 1\leq j \leq k+1} }
\notag \\
& = \det
\begin{pmatrix}
f_{01}\cdot S_1 & df_{01}\cdot S_1 & \dots & f_{0,k+1}\cdot S_{k+1} & df_{0,k+1}\cdot S_{k+1}\\
f_{11}\cdot S_1 & df_{11}\cdot S_1 & \dots & f_{1,k+1}\cdot S_{k+1} & df_{1,k+1}\cdot S_{k+1}\\
\vdots & \vdots & \ddots & \vdots & \vdots \\
f_{2k+1,1}\cdot S_1 & df_{2k+1,1}\cdot S_1 & \dots & f_{2k+1,k+1}\cdot S_{k+1} & df_{2k+1,k+1}\cdot S_{k+1} \label{eq:highergaussian}
\end{pmatrix}.
\end{align}
In particular, if $x_1'+\dots+x_{k+1}'$ is in a neighborhood of $x_1+\dots+x_{k+1}$, then
\begin{equation*}
	G_{k+1,L}(p_0\wedge \dots \wedge p_{2k+1})(x_1'+\dots+x_{k+1}') = \det \begin{pmatrix} f_{ij}(x'_j)\cdot S_j(x'_j) & df_{ij}(x'_j)\cdot S_j(x'_j) \end{pmatrix}_{\substack{0\leq i\leq 2k+1 \\ 1\leq j \leq k+1 }}.
\end{equation*}
Notice that this expression is symmetric in the $x_i'$, as it should be, because if $x_i',x_j'$ are exchanged, then two columns for $x_i'$ are  exchanged with two columns for $x_j'$, so that the determinant is invariant.

\begin{remark}
	If $k=0$ then the $1$-Gaussian map is $G_{1,L}\colon \wedge^2 H^0(C,L) \to H^0(C,L^{\otimes 2}\otimes \omega_C)$
	\[ G_{1,L}(p_0\wedge p_1)  = \det \begin{pmatrix} f_0 \cdot S_1 & df_0\cdot S_1 \\ f_1\cdot S_1 & df_1 \cdot S_1 \end{pmatrix} = (f_0df_1 - f_1df_0) S^2_1. \]
	This is the usual Gaussian map or Wahl map of \cite{wahl}.   The $(k+1)$-Gaussian map has image in
	\[H^0(C_{k+1}, A_{k+1,L}^{\otimes 2}\otimes S_{k+1,\omega_C}) \hookrightarrow H^0(C_{k+1},S_{k+1,L^{\otimes 2}\otimes \omega_C}) = \operatorname{Sym}^{k+1}H^0(C,L^{\otimes 2}\otimes \omega_C), \]
	so it can also be seen as a map
	\[G_{k+1,L}\colon \wedge^{2k+2}H^0(C,L) \longrightarrow \operatorname{Sym}^{k+1}H^0(C,L^{\otimes 2} \otimes \omega_C). \]
	Indeed, it can be written explicitly in terms of a polynomial of degree $k+1$ in the $G_{1,L}(p_i\wedge p_j)$. For example one can check explicitly that (removing the $L$ from the notation $G_{k+1,L}$ for simplicity):
	\begin{small}
	\[
	G_{2}(p_0\wedge p_1\wedge p_2\wedge p_3) = 4\cdot(G_{1}(p_0\wedge p_1)G_{1}(p_2\wedge p_3) - G_{1}(p_0\wedge p_2)\cdot G_{1}(p_1\wedge p_3) + G_{1}(p_0\wedge p_3)\cdot G_{1}(p_1\wedge p_2)). \]
	\end{small}
	
\noindent This expression has a natural generalization.
\end{remark}

\begin{lemma}\label{lem:pfaffianexpression}
	For $p_0,\dots,p_{2k+1}\in H^0(C,L)$ the corresponding Gaussian map is, up to a scalar, the Pfaffian of a skew-symmetric matrix with entries $G_{1,L}(p_i\wedge p_j)$:
	\[ G_{k+1,L}(p_0\wedge \dots \wedge p_{2k+1}) = (k+1)^{(k+1)}\cdot \operatorname{Pf} \left( G_{1,L}(p_i\wedge p_j) \right)_{0\leq i,j \leq 2k+1} . \]
\end{lemma}
\begin{proof}
Let $M$ be the $(2k+2)\times (2k+2)$ matrix appearing in \eqref{eq:highergaussian}, so that $G_{k+1,L}(p_0\wedge \dots \wedge p_{2k+1}) = \det(M)$. Consider also the $(2k+2)\times (2k+2)$ block diagonal matrix
\[ J =  \begin{pmatrix} 0 & 1 \\ -1 & 0 \end{pmatrix}^{\oplus (k+1)}. \]
This is skew-symmetric and then $\det(M) = \operatorname{Pf}(MJM^t)$ by properties of the Pfaffian. For $0\leq i<j \leq 2k+1$ the corresponding entry is (see for example \cite[Equation (2.2)]{pfaffians})
\begin{align*}
	(MJM^t)_{ij} = \sum_{h=1}^{k+1} \det\begin{pmatrix}
		f_{ih}\cdot S_{h} & df_{ih}\cdot S_h \\
	    f_{jh}\cdot S_{h} & df_{jh}\cdot S_h
	 \end{pmatrix}
	 = (k+1)\cdot G_{1,L}(p_i\wedge p_j)
\end{align*}
and this concludes the proof.
\end{proof}

\begin{remark}
	We can give a geometric interpretation of the first order tautological bundles and also of the Gaussian map. Assume that $L$ is very ample so that it defines an embedding $C\subseteq \nP(H^0(C,L))$. The evaluation map
	$H^0(C,L)\otimes \sO_{C_{k+1}} \to E^{(1)}_{k+1,L}$ induces a rational map
	\[ \gamma_{k+1,L}\colon C_{k+1} \dashrightarrow \operatorname{Gr}(2k+2,H^0(C,L)) \]
	that associates to $\xi=x_1+\dots+x_{k+1}$ the osculating space in $\nP(H^0(C,L))$ of first order to $C$ at the $x_i$. When $k=0$ this is simply the usual Gauss map that associates to each point on $C$ its tangent line in $\nP(H^0(C,L))$. Then the $(k+1)$-Gauss map is the one induced by the Pl\"ucker embedding
	\[ \gamma_{k+1,L}\colon C_{k+1} \dashrightarrow \operatorname{Gr}(2k+2,H^0(C,L)) \hookrightarrow \nP(\wedge^{2k+2}H^0(C,L)) \]
	via pullback of sections.
\end{remark}

\subsection{Higher Szeg\H{o} kernels and Scorza correspondences}\label{sec:szego}
For our smooth projective curve $C$ of genus $g$, consider the divisor $D_{k+1,k+1}\subseteq C_{k+1}\times C_{k+1}$ and the corresponding line bundle $\sO_{C_{k+1} \times C_{k+1}}(D_{k+1,k+1})$ on $C_{k+1}\times C_{k+1}$. The short exact sequence of sheaves on $C_{k+1}\times C_{k+1}$
\[ 0 \longrightarrow \sO_{C_{k+1} \times C_{k+1}} \longrightarrow \sO_{C_{k+1} \times C_{k+1}}(D_{k+1,k+1}) \longrightarrow \sO_{D_{k+1,k+1}}(D_{k+1,k+1}) \longrightarrow 0 \]
shows that there is, up to scalar multiplication, a unique nonzero global section of $\sO_{C_{k+1} \times C_{k+1}}(D_{k+1,k+1})$ that vanishes on $D_{k+1,k+1}$. We call it the \emph{(k+1)-prime form} of $C$ and we denote it by $E=E(\xi,\xi')$. We can also think of it as a rational function on $C_{k+1}\times C_{k+1}$ that has a simple pole along $D_{k+1,k+1}$ and is regular everywhere else.

\begin{remark}
The \emph{1-prime form} is called simply the \emph{prime form} of $C$. If $E(x,y)$ is the prime form, we can give an explicit expression for the $(k+1)$-prime form as follows. For two general points $\xi=x_1+\dots+x_{k+1},\xi'=x'_1+\dots+x'_{k+1} \in C_{k+1}$, define
\[ E(\xi,\xi') := \prod_{1 \leq i,j \leq k+1 } E(x_i,x'_j). \]
The product on the right is symmetric on the $x_i$ and $x_i'$ separately, so that it is well defined as an expression of $(\xi,\xi')$. Furthermore, it is straightforward to see that it has a simple pole along $D_{k+1,k+1}$ and that it is  regular everywhere else, so that this is the $(k+1)$-prime form of $C$. It is also well-known, that the $1$-prime form is antisymmetric, in the sense that $E(x,x')=-E(x',x)$. It follows that for the $k+1$-prime form
\[ E(\xi,\xi') = (-1)^{k+1} E(\xi',\xi). \]
\end{remark}

\noindent
Consider now a line bundle $\alpha$ such that
\[  h^0(C,\alpha) = h^1(C,\alpha) = 0 \]
and define
\[ \alpha' := \omega_C\otimes \alpha^{-1}. \]
Notice that $h^0(C,\alpha')=h^1(C,\alpha')=0$ and  $\alpha\cong \alpha'$ precisely when $\alpha$ is a theta characteristic, necessarily even, since $h^0(C,\alpha)=0$.

We define the following line bundle on $C_{k+1}\times C_{k+1}$:
\[ \sL_{(\alpha,\alpha')} := (N_{k+1,\alpha}\boxtimes N_{k+1,\alpha'})(D_{k+1,k+1}). \]
By Lemma \ref{lem:D|_Delta}, there is an isomorphism
\begin{equation}\label{eq:isoLaa'}
{\sL_{(\alpha,\alpha')}}_{|\Delta_{C_{k+1}}} \cong N_{k+1,\alpha'}^{\otimes 2} \otimes A_{k+1,\omega_C^{-1}}  \cong S_{k+1,\alpha \otimes \alpha' \otimes \omega_C^{-1}} \cong \sO_{\Delta_{C_{k+1}}}.
\end{equation}
So we define a \emph{$(k+1)$-Szeg\H{o} kernel associated to $(\alpha,\alpha')$} to be a section $S_{(\alpha,\alpha')} = S_{(\alpha,\alpha')}(\xi,\xi')$ in $H^0(C_{k+1}\times C_{k+1}, \sL_{(\alpha,\alpha')})$ such that ${S_{(\alpha,\alpha')}} = 1$ when restricted to $\Delta_{C_{k+1}}$ via the isomorphism of \eqref{eq:isoLaa'}. The name follows the literature in mathematical physics and string theory \cite{DHP,Fay,raina,Wi}. When $k=0$ and $\alpha$ is a theta-characteristic, the zero locus of the  Szeg\H{o} kernel is the \emph{Scorza correspondence} on $C\times C$ associated to $\alpha$. This was classically introduced by Scorza for curves of genus $3$. For a recent study of this correspondence in arbitrary genus and for an throughout exposition of previous results, we refer to the paper  \cite{FI}. Here we consider the case of arbitrary $k$, with the following result, that is probably well known but we include a proof of it via symmetric products for completeness.

\begin{proposition}\label{prop:szego}
	With the above notation, there is a unique $(k+1)$-Szeg\H{o} kernel $S_{(\alpha,\alpha')}$ associated to $(\alpha,\alpha')$. Furthermore
	\[ S_{(\alpha,\alpha')}(\xi,\xi')= 0 \qquad  \text{ if and only if } \qquad  h^0(C,\alpha(\xi'-\xi)) = h^0(C,\alpha'(\xi-\xi')) >0.   \]
	and under the involution of $C_{k+1}\times C_{k+1}$ that exchanges the factors, it holds that
	\[ S_{(\alpha,\alpha')}(\xi,\xi') = S_{(\alpha',\alpha)}(\xi',\xi) \qquad  \text{ for all } \xi,\xi'\in C_{k+1}. \]
\end{proposition}

\begin{proof}
	Consider the diagonal $\Delta_{C_{k+1}}\subseteq C_{k+1}\times C_{k+1}$.  and the restriction map $\mathscr{L}_{(\alpha,\alpha')} \longrightarrow \mathscr{L}_{(\alpha,\alpha')|\Delta_{C_{k+1}}} \cong \sO_{\Delta_{C_{k+1}}}$. If we can prove that this map is an isomorphism on global sections, we will show that a Szeg\H{o} kernel exists and is unique. If we denote by $\pr_2\colon C_{k+1}\times C_{k+1 } \to C_{k+1}$ the second projection, we will show that
	\begin{equation}\label{eq:mappushforward}
	\pr_{C_2,*}\mathscr{L}_{(\alpha,\alpha')} \longrightarrow \pr_{C_2,*}(\mathscr{L}_{(\alpha,\alpha')}\otimes \sO_{\Delta_{C_{k+1}}})
	\end{equation}
	is an isomorphism, which implies what we want. Observe that both $\mathscr{L}_{(\alpha,\alpha')}$ and $\mathscr{L}_{(\alpha,\alpha')}\otimes \sO_{\Delta_{C_{k+1}}}$ are flat with respect to $\pr_2$ and if we restrict $\mathscr{L}_{(\alpha,\alpha')}\to \mathscr{L}_{(\alpha,\alpha')}\otimes \sO_{\Delta_{C_{k+1}}}$ to the fiber $C_{k+1}\times \{\xi'\}$ we obtain the map
	\begin{align}
	H^0(C_{k+1},N_{k+1,\alpha(\xi')}) &\longrightarrow H^0(C_{k+1},N_{k+1,\alpha(\xi')}\otimes \kappa(\xi')).
	\end{align}
	Since $h^0(C,\alpha)=0$, Lemma \ref{lem:alphamap1} shows that this map is an isomorphism. As this holds for any $\xi'\in C_{k+1}$ we see that map \eqref{eq:mappushforward} is an isomorphism as well.

\medskip

\noindent	Now we want to show that $S_{(\alpha,\alpha')}(\xi,\xi')\ne 0$ if and only if $h^0(C,\alpha(\xi'-\xi))=0$. Observe that if we fix $\xi'\in C_{k+1}$, then $S_{(\alpha,\alpha')}(-,\xi')$ is a  section of $H^0(C_{k+1}\times \{\xi'\},\mathscr{L}_{C_{k+1}\times \{\xi'\}}) \cong H^0(C_{k+1},N_{\alpha(\xi')})$, which is moreover nonzero, since it does not vanish at $(\xi',\xi')$. Since $h^0(C_{k+1},N_{k+1,\alpha(\xi')})=1$, we see that $S_{(\alpha,\alpha')}(\xi,\xi')\ne 0$ if and only if $\xi$ is not a base point of $N_{k+1,\alpha(\xi')}$. Then the conclusion follows from Lemma \ref{lem:alphamap1}.

\medskip

\noindent	Finally, consider the involution $\iota$ on $C_{k+1}\times C_{k+1}$ that exchanges the two factors. Then there is a canonical  isomorphism  $\iota^*\mathscr{L}_{(\alpha,\alpha')} \cong \sL_{(\alpha',\alpha)}$, and by construction  ${\iota^*S_{(\alpha,\alpha')}}_{|\Delta_{C_{k+1}}}= 1$. This means that $\iota^*S_{(\alpha,\alpha')} \cong S_{(\alpha',\alpha)}$, which is what we wanted to prove.
\end{proof}

Let now $\alpha$ be a theta-characteristic on $C$ with $h^0(C,\alpha)=h^1(C,\alpha)=0$. An \emph{algebraic orientation} associated to $\alpha$ is an isomorphism
\begin{equation}\label{eq:algorientation}
\rho\colon \alpha\otimes \alpha \longrightarrow \omega_C .
\end{equation}
Once we fix such an algebraic orientation, we also have an isomorphism $\alpha \cong \alpha' = \omega_C\otimes \alpha^{-1}$ and a corresponding isomorphism $\sL_{(\alpha,\alpha)} \cong \sL_{(\alpha,\alpha')}$. We then define the Szeg\H{o} kernel $S_{(\alpha,\alpha)}$ as the section of $H^0(C_{k+1}\times C_{k+1},\sL_{(\alpha,\alpha)})$ corresponding to $S_{(\alpha,\alpha')}$ under this isomorphism. Looking at \eqref{eq:algorientation}, this means that $S_{(\alpha,\alpha)}$ is the unique section of $\sL_{(\alpha,\alpha)}$ such that ${S_{(\alpha,\alpha)}}_{|\Delta_{C_{k+1}}} = 1$ under the isomorphism
\[ {\sL_{(\alpha,\alpha)}}_{|\Delta_{C_{k+1}}} \cong S_{k+1,\alpha^{\otimes 2} \otimes \omega_C^{-1}} \overset{\simeq}{\longrightarrow} \sO_{\Delta_{C_{k+1}}}, \]
which is induced by the chosen algebraic orientation.

\begin{remark}\label{rem:szegosymm}
 If $\alpha$ is a theta characteristic with no cohomology, Proposition \ref{prop:szego} shows that $S_{(\alpha,\alpha)}$ is symmetric: $S_{(\alpha,\alpha)}(\xi,\xi') = S_{(\alpha,\alpha)}(\xi',\xi)$ for all $\xi,\xi'\in C_{k+1}$.
\end{remark}

Let now $\alpha$ be any line bundle on $C$ with no cohomology. We can consider the Szeg\H{o}  kernel $S_{(\alpha,\alpha')}$ as a rational section of $N_{k+1,\alpha}\boxtimes N_{k+1,\alpha'}$ with simple poles along $D_{k+1,k+1}$ and regular everywhere else.  Following the terminology of \cite{shamovichvinnikov}, we define the \emph{$(k+1)$-Cauchy kernel} of $(\alpha,\alpha')$ as
\[ K_{(\alpha,\alpha')}(\xi,\xi') := \frac{S_{(\alpha,\alpha')}(\xi,\xi')}{E(\xi,\xi')}, \]
which we see as a rational section of $N_{k+1,\alpha}\boxtimes N_{k+1,\alpha'}$ for which now $D_{k+1,k+1}$ is not a pole anymore.

\begin{remark}
The terminology for Szeg\H{o} and Cauchy kernels is not uniform. For example, what we call Cauchy kernel here is called the Szeg\H{o} kernel in \cite{Fay,raina}.
\end{remark}

Let $\alpha$ be a theta characteristic on $C$ with no sections. After Proposition \ref{prop:szego}, it is natural to define the \emph{$(k+1)$-Scorza correspondence} associated to $\alpha$ to be the divisor
\[ \{(\xi,\xi') \in C_{k+1}\times C_{k+1}\,|\, h^0(C,\alpha(\xi'-\xi))>0\} \]
with the scheme structure given as the zero locus of the $(k+1)$-Szeg\H{o} kernel $S_{(\alpha,\alpha)}$. In particular, it has class given by $\sL_{(\alpha,\alpha)}\cong (N_{k+1,\alpha}\boxtimes N_{k+1,\alpha})(D_{k+1,k+1})$.A variant of this Scorza correspondence via cartesian rather than symmetric products is considered by Ruggiero in \cite{ruggiero}, with a focus on the numerical class and the local property for the case $k=1$.

As for the usual Scorza correspondence, we can describe the $(k+1)$-Scorza correspondence and the corresponding Szeg\H{o} kernel, in terms of theta functions.
Consider the difference map:
\[ \operatorname{diff}_{k+1}\colon C_{k+1} \times C_{k+1} \longrightarrow \operatorname{Pic}_0(C), \qquad (\xi,\xi') \longmapsto \sO_C(\xi'-\xi) \]
and the theta divisor  $\Theta_{\alpha} \subseteq \Pic_0(C)$ corresponding to $\alpha$:
\[ \Theta_{\alpha} = \{ \eta \in \Pic_0(C) \,|\, h^0(C,\eta\otimes \alpha) > 0\} .\]
This induces a principal polarization on $\Pic_0(C)$, so, up to a scalar multiplication, there is a unique nonzero section $\theta_{\alpha} \in H^0(\Pic_0(C),\sO_{\Pic_0(C)}(\Theta_{\alpha}))$ which is called a \emph{theta function with characteristic $\alpha$}. The fact that $\alpha$ is not effective means that $\theta_{\alpha}(0)\ne 0$. Then it holds that:

\begin{proposition}\label{prop:szego_theta}
	With the previous notation, it holds that
	\[ S_{(\alpha,\alpha')} = \frac{\operatorname{diff}^*\theta_{\alpha}}{\theta_{\alpha}(0)}, \quad \text{ i.e. } \quad S_{(\alpha,\alpha')}(\xi,\xi') = \frac{\theta_{\alpha}(\xi'-\xi)}{\theta_{\alpha}(0)} \text{ for } \xi,\xi' \in C_{k+1}.   \] 
\end{proposition}
\begin{proof}
	Proposition \ref{prop:szego}, shows that $\operatorname{diff}^*(\Theta_{\alpha}) = \{ S_{(\alpha,\alpha')} = 0\}$ at the level of sets and if we show that it is true at the level of divisor classes, we are done. By the seesaw principle is it enough to show that these two divisor classes coincide when restricted to all fibers of the two projections $C_{k+1}\times C_{k+1} \to C_{k+1}$. Because of Proposition \ref{prop:szego},this means that for all $\xi,\xi'\in C_{k+1}$ we have
	\[ (\operatorname{diff}^*\Theta_\alpha)_{|C_{k+1}\times \{\xi'\}} \cong N_{k+1,\alpha(\xi')}, \quad (\operatorname{diff}^*\Theta_\alpha)_{|\{\xi\}\times C_{k+1}} \cong N_{k+1,\alpha'(\xi)} ,\] 
	and by symmetry it is enough to prove the first isomorphism. The restriction of the difference map to $C_{k+1}\times\{\xi'\}$ is the composition 
	\[ C_{k+1} \overset{u_{-\xi'}}{\longrightarrow}  \operatorname{Pic}_0(C) \overset{\iota}{\longrightarrow} \operatorname{Pic}_0(C)  \]
	where the first map is the Abel--Jacobi map $u_{-\xi'}(\xi)=\sO_C(\xi-\xi')$ and the second one is the inversion. It holds that $\iota^*\Theta_{\alpha}\cong \Theta_{\alpha'}$ so what we have to prove is that $u_{-\xi'}^*\Theta_{\alpha'} \cong N_{k+1,\alpha(\xi')}$. which is given by Proposition \ref{prop:pullback_theta}. 
\end{proof}

\begin{remark} 
The fact that $S_{(\alpha,\alpha)}$ is symmetric if $\alpha$ is an even theta characteristic reflects the fact that the theta function $\theta_{\alpha}$ is even with respect to the natural involution on $\operatorname{Pic}_0(C)$.
\end{remark}

\subsection{Secant varieties of curves}\label{subsec:secvar}
Here we recall some facts on secant varieties of curves, referring  to \cite{ENP, ENP3} for further details.
Consider again the smooth projective curve $C \subseteq \nP^r=\nP(V)$, nondegenerate and of genus $g$. The line bundle $L=\sO_C(1)$ is very ample and $V$ can be seen as a linear system $V\subseteq H^0(C,L)$. We set
$$
B^{k}(L):=\nP(E_{k+1,L})
$$
equipped with the canonical projection $\pi_k \colon B^{k} \rightarrow C_{k+1}$. If the line bundle $L$ is clear from the context, then we simply write $B^k$ for $B^k(L)$. The tautological line bundle $\sO_{B^k}(1)$ has sections $H^0(B^k,\sO_{B^k}(1)) \cong H^0(C_{k+1},E_{k+1,L}) = H^0(C,L)$, and the linear system $V\subseteq H^0(C,L)$ gives a rational map
$$
\beta_{V,k} \colon B^k \dashrightarrow  \nP^r.
$$
Notice that the  (closure of the) image of $\beta_{V,k}$ is the $k$-th secant variety $\Sigma_k=\Sigma_k(C,L;V)$ of $C \subseteq \nP^r$. Recall that $\dim \Sigma_k(C,L;V) = \min\{r, 2k+1\}$ (see \cite{lange}). By the general position theorem \cite[p.109]{ACGH}, if $\xi\in C_{k+1}$ is general, then $\beta_{V,k}(\pi_k^{-1}(\xi)) = \langle \xi \rangle$ is the $k$-plane spanned by $\xi$ in $\nP^r$. Furthermore, this $k$-plane is exactly $k+1$-secant: $\langle \xi \rangle \cap C = \xi$.  

\medskip

If $V$ is $k$-very ample, then it generates $E_{k+1,L}$ so that $\beta_{V,k}$ is a morphism. In this case,
$$
B^k \cong \{ (\xi, p)  \mid p \in \langle \xi \rangle\} \subseteq C_{k+1} \times \nP^r,
$$
and $\beta_{V,k}(\pi_k^{-1}(\xi)) = \langle \xi \rangle$ is the $k$-plane secant along $\xi$ in $\nP^r$, for any $\xi\in C_{k+1}$. Finally, recall from \cite{ENP} that if $V$ is $k$-very ample there is a natural map
\[ B^{k-1}\times C \longrightarrow B^k, \quad ((\xi',p),q) \longmapsto (p+\xi',q)\]
whose image is a prime divisor $Z_{k-1}\subseteq B^k$ of class $\sO_{B^k}(Z_{k-1})\cong \sO_{B^k}(k+1)\otimes \pi_k^*A_{k,L}^{-1}$ and such that $\beta_{V,k}(Z_{k-1})=\Sigma_{k-1} \subseteq \Sigma_k$.

\begin{lemma}[{cf. \cite[Lemma 4.4]{kummersinn}}]\label{lem:beta=bir}
Suppose that $r \geq 2k+2$. Then $\beta_{V,k} \colon B^k \dashrightarrow \Sigma_k$ is birational. If furthermore $V$ is $k$-very ample, then $Z_{k-1}$ is the unique exceptional divisor of this map.
\end{lemma}

\begin{proof}
	For $1 \leq m \leq r$ and $0 \leq s \leq m-2$, we claim that
	$$
	\dim \{ \xi \in C_m \mid \dim \langle \xi \rangle \leq s \} \leq s.
	$$
	The family of hyperplanes of $\nP^r$ containing a plane of dimension at most $s$ has dimension at least $r-s-1$, and a hyperplane of $\nP^r$ contains only finitely many planes of the form $\langle \xi \rangle$ for $\xi \in C_m$. Then the claim follows from the general position theorem \cite[p.109]{ACGH}, which states that any $r$ points in $C \cap H$ are linearly independent for a general hyperplane $H$ of $\nP^r$. Next, we claim that
	$$
	\dim \{ (\xi, \xi', x) \in C_{k+1} \times C_{k+1} \times \nP^r \mid \xi \neq \xi', x \in \langle \xi \rangle \cap \langle \xi' \rangle, x \not\in \Sigma_{k-1} \} \leq 2k. 
	$$
	For $\xi, \xi' \in C_{k+1}$ with $\xi \neq \xi'$, let $m:=\deg (\xi \cup \xi') \leq 2k+2$, and $s:=\dim \langle \xi \cup \xi' \rangle$. Notice that $x \in \langle \xi \rangle \cap \langle \xi' \rangle, x \not\in \Sigma_{k-1}$ if and only if $s \leq m-2$. In this case, the first claim shows that there is a family of pairs $(\xi, \xi')$ with  $\xi \neq \xi', \deg (\xi \cup \xi')=m, \dim \langle \xi \cup \xi' \rangle = s$ having dimension at most $s$. As $\dim (\langle \xi \rangle \cap \langle \xi' \rangle) \leq 2k-s$, the second claim holds. Now, the lemma easily follows from the second claim. We know that $\beta_{V,k}$ is generically finite. If $\beta_{V,k}$ is not birational, then $\beta_{V,k}^{-1}(x)$ for a general point $x \in \Sigma_k$ gives  $\xi, \xi' \in C_{k+1}$ with $\xi \neq \xi'$ and $x \in \langle \xi \rangle \cap \langle \xi' \rangle$. Thus there is a family of triples $(\xi, \xi', x) \in C_{k+1} \times C_{k+1} \times \nP^r$ with $\xi \neq \xi', x \in \langle \xi \rangle \cap \langle \xi' \rangle, x \not\in \Sigma_{k-1}$ having dimension at least $2k+1$, and we get a contradiction to the second claim. 
	Suppose now that $V$ is $k$-very ample and assume that there is a prime divisor $E$ on $B^k(L)$ such that $E \neq Z_{k-1}$ and $\ell:=\dim \beta_{V,k}(E) \leq 2k-1$. For a general point $x \in \beta_{V,k}(E)$, let $F:=\beta_{V,k}^{-1}(x) \cap E$. Then $\dim F = 2k-\ell$. Note that $x \not\in \Sigma_{k-1}$. Two distinct points on $F$ give $\xi, \xi' \in C_{k+1}$ with $\xi \neq \xi'$ and $x \in \langle \xi \rangle \cap \langle \xi' \rangle$. Thus there is a family of triples $(\xi, \xi', x) \in C_{k+1} \times C_{k+1} \times \nP^r$ with $\xi \neq \xi', x \in \langle \xi \rangle \cap \langle \xi' \rangle, x \not\in \Sigma_{k-1}$ having dimension at least $\ell + 2(2k-\ell) = 4k-\ell \geq 2k+1$, and we get a contradiction to the second claim. 
\end{proof}

In the case where $V=H^0(C,L)$ is the complete linear system, we denote the secant variety $\Sigma_k(C,L;H^0(C,L))$ by $\Sigma_k(C,L)$ and  $\beta_{H^0(C,L),k}$  by $\beta_{k}\colon B^k \dashrightarrow \Sigma_k(C,L)$. If $\deg L \geq 2g+2k+1$, then $\beta_{k}\colon B^k \to \Sigma_k(C,L)$ is a resolution of singularities. Furthermore, $\Sigma_k(C,L)$ has normal Cohen--Macaulay Du Bois singularities (see \cite[Theorem 1.1]{ENP}). If $\deg L \geq 2g+2k+2$, then the defining ideal $I_{\Sigma_k(L)/\nP^r}$ is generated by degree $k+2$ forms by \cite[Theorem 1.2]{ENP}. If $\deg L \geq 2g+2k+3$, then \cite[Theorem 5.2]{ENP} shows   that
$$
H^0(\nP^r, \sI_{\Sigma_k(C,L)/\nP^r}(k+2)) = H^0(\Sigma_{k+1}(C,L), \sI_{\Sigma_k(C,L)/\Sigma_{k+1}(C,L)}(k+2))=H^0(C_{k+2}, A_{k+2,L}).
$$

\section{Admissible determinantal representations on secants of curves}\label{sec:constdetrep}
\noindent The aim of this section is to use Szeg\H{o} kernels and Theorem \ref{thm:admrepresentation} to construct admissible determinantal representations on secant varieties of curves (see Theorem \ref{thm:detrep}). We also prove Theorem \ref{thm:rat_norm}.

\subsection{Construction of admissible determinantal representations on secant varieties of curves}
As before, we let $C \subseteq \nP^r = \nP(V)$ be a smooth projective curve of genus $g$ over an algebraically closed field $\kk$ of characteristic zero. And we set $L = \sO_C(1)$ a very ample line bundle on $C$ with a linear system $V\subseteq H^0(C,L)$. We assume that
$$
r\geq 2k+2
$$
so that  $\beta_{V,k} \colon B^k(L) \dashrightarrow \Sigma_k$ is birational by Lemma \ref{lem:beta=bir}. We also fix a  line bundle $\alpha$ on $C$ such that $h^0(C,\alpha)=h^1(C,\alpha)=0$ and we set $\alpha':=\omega_C\otimes \alpha^{-1}$.
\medskip

Denote by $\sigma:=\sigma_{k+1,k+1} \colon C_{k+1} \times C_{k+1} \to C_{2k+2}$ the addition map $(\xi, \xi') \mapsto \xi+\xi'$. Note that
\[\sigma^*N_{2k+2,L} \cong (N_{k+1,L}\boxtimes N_{k+1,L})(-D_{k+1,k+2}) \]
so we can define a map $P\colon \wedge^{2k+2}V \to H^0(C_{k+1}\times C_{k+1},(N_{k+1,L}\boxtimes N_{k+1,L})(-D_{k+1,k+1}))$ as the composition
\[
	\wedge^{2k+2} V  \xrightarrow{~\det_{L,k+1}~}  H^0(C_{2k+2}, N_{2k+2,L}) \\
	 \xrightarrow{~\sigma^*~}  H^0(C_{k+1}\times C_{k+1},(N_{k+1,L}\boxtimes N_{k+1,L})(-D_{k+1,k+1})).
\]
Let also $S_{(\alpha,\alpha')}$ be the Szeg\H{o} kernel given by Proposition \ref{prop:szego}. Recall that this is a nonzero section in $H^0(C_{k+1}\times C_{k+1},(N_{k+1,\alpha}\otimes N_{k+1,\alpha'})(D_{k+1,k+1}))$. Consider the composition $\phi=\phi_{V,k,\alpha}$ of the maps
\begin{align*}
	\wedge^{2k+2} V & \xrightarrow{P}  H^0(C_{k+1}\times C_{k+1},(N_{k+1,L}\boxtimes N_{k+1,L})(-D_{k+1,k+1})) \\
	& \xrightarrow{~\cdot S_{(\alpha,\alpha')}~}  \underbrace{H^0(C_{k+1} \times C_{k+1}, A_{k+1, L \otimes \alpha} \boxtimes A_{k+1, L \otimes \alpha'})}_{=H^0(C_{k+1}, A_{k+1, L \otimes \alpha}) \otimes H^0(C_{k+1}, A_{k+1, L \otimes \alpha'})} .
\end{align*}
This is a linear map
$$
\phi \colon \wedge^{2k+2} V \longrightarrow H^0(C_{k+1}, A_{k+1, L \otimes \alpha}) \otimes H^0(C_{k+1}, A_{k+1, L \otimes \alpha'})
$$
defined by
\[ \phi(p_0\wedge \dots \wedge p_{2k+1})(\xi,\xi') = P(p_0\wedge \dots \wedge p_{2k+1})(\xi,\xi)\cdot S_{(\alpha,\alpha')}(\xi,\xi'). \]
We want to use Theorem \ref{thm:admrepresentation} to show that this is an admissible determinantal representation of $\Sigma_k$. We first need to determine when $\phi$ vanishes.

\begin{lemma}\label{lem:phivanishes}
	Let $p_0,\dots,p_{2k+1} \in V$ and let $\langle p_0,\dots,p_{2k+1} \rangle \subseteq V$ be the subspace that they generate. For $\xi,\xi'\in C_{k+1}$ it holds that
	\[ \phi(p_0\wedge \dots \wedge p_{2k+1})(\xi,\xi') = 0 \]
	if and only if one or both of the following two condition hold:
	\begin{itemize}
		\item[$(a)$]  $\langle p_0,\dots,p_{2k+1} \rangle$ does not separate $\xi+\xi'$.
		\item[$(b)$] $h^0(C,\alpha(\xi'-\xi))>0$.
	\end{itemize}
\end{lemma}
\begin{proof}
	By construction we see that $\phi(p_1\wedge \dots \wedge p_{2k+2})(\xi,\xi') = 0$ if and only if $P(p_0\wedge \dots \wedge_{2k+1})(\xi,\xi')=\det_{k+1,L}(p_0\wedge \dots \wedge p_{2k+1})(\xi+\xi')=0$ or $S_{(\alpha,\alpha')}(\xi,\xi')=0$. Lemma \ref{lem:A(p_1...p_m)(D)=0} shows that the first case corresponds to $(a)$ while Proposition \ref{prop:szego} shows that the second case corresponds to $(b)$.
\end{proof}

The following is the main result of this section.

\begin{theorem}\label{thm:detrep}
	Assume that $r\geq 2k+2$ and that \[d:=h^0(C_{k+1},A_{k+1,L\otimes \alpha}) = h^0(C_{k+1},A_{k+1,L\otimes \alpha'}) = \deg(\Sigma_k).\] The linear map
	\[  \phi = \phi_{k,V,\alpha} \colon \wedge^{2k+2} V \longrightarrow H^0(C_{k+1},A_{k+1,L\otimes \alpha}) \otimes H^0(C_{k+1},A_{k+1,L\otimes \alpha'}) \]
	is an admissible determinantal representation of $\Sigma_k$ of degree $d$, whose corresponding Ulrich sheaf , that we denote by $\sA_{k,L\otimes \alpha}$, has rank one. If $\alpha$ is a theta-characteristic, this is a symmetric representation and $\sA_{k,L\otimes \alpha}$ is a symmetric Ulrich sheaf.
\end{theorem}
\begin{proof}
	We first treat the case when $V$ is $k$-very ample so that there is  the well-defined morphism $\beta_k\colon B_k \longrightarrow \Sigma_k$ and the  coherent sheaf $\widetilde{\sA}_{k,L\otimes \alpha} = \beta_{k,*}(\pi_k^* A_{k+1,L\otimes \alpha})$ on $\Sigma_k$. By construction, $H^0(\Sigma_k,\widetilde{\sA}_{k,L\otimes \alpha}) \cong H^0(B_k,\pi_k^*A_{k+1,L\otimes \alpha}) \cong H^0(C_{k+1},A_{k+1,L\otimes \alpha})$. Furthermore, if $x \in \Sigma_k$ is a general point, then the fiber $\beta^{-1}(x)$ is a single point, meaning that there is a unique $\xi\in C_{k+1}$ such that $x\in \langle \xi \rangle$. The evaluation of a section of $H^0(\Sigma_k,\widetilde{\sA}_{k,L\otimes \alpha})$ at $x$ corresponds to the evaluation of a section of $H^0(C_{k+1},A_{k+1,L\otimes \alpha})$ at $\xi$. To apply Theorem \ref{thm:admrepresentation}, we first observe that $\widetilde{\sA}_{k,L\otimes \alpha}$ is torsion-free and of rank one because $\beta_k$ is birational.  Now we check the two properties:
	\medskip

	\noindent
	(i)  Suppose that $p_0, \ldots, p_{2k+1}$ have no common zero on $\Sigma_k$ and let
	$
	\pi  \colon \Sigma_k \longrightarrow \nP^{2k+1},
	$
	be the corresponding outer projection,
	which is a finite morphism of degree $d$. Consider a general fiber $\{y_1, \ldots, y_d\}$. We can assume that there are unique $\xi_1, \ldots, \xi_d \in C_{k+1}$ such that $y_i \in \langle \xi_i \rangle$ for all $1 \leq i \leq d$. We may also assume that each $\xi_i$ consists of $k+1$ distinct points and that $\pi$ is an immersion around $y_1, \ldots, y_d$. We want to show that
	$$
	\phi(p_0 \wedge \cdots \wedge p_{2k+1})(\xi_i, \xi_j) =
	\begin{cases} \text{nonzero} & \text{if $i=j$} \\ \text{zero} & \text{if $i \neq j$} \end{cases}.
	$$
	First, assume that $i \neq j$. Note that $\pi$ sends $\langle \xi_i + \xi_j \rangle$ to a smaller dimensional linear subspace of $\nP^{2k+1}$. This means that $\langle p_0,\dots,p_{2k+1}\rangle$ does not separate $\xi_i+\xi_j$. By Lemma \ref{lem:phivanishes} we get $\phi(p_0 \wedge \cdots \wedge p_{2k+1})(\xi_i, \xi_j) = 0$. Next, assume that $i=j$. Terracini's lemma tells us that the projective tangent space to $\Sigma_{k}$ at $y_i$ is $\langle 2\xi_i \rangle$. As $\pi$ is an immersion at $x_i$, this means that $\langle p_0,\dots,p_{2k+1} \rangle$ separates $2\xi_i$. By Lemma \ref{lem:phivanishes} it is then enough to check that $S_{(\alpha,\alpha')}(\xi_i,\xi_i)\ne 0$ and this follows from Proposition \ref{prop:szego}.
	\medskip

	\noindent
	(ii) Suppose that $p_0, \ldots, p_{2k+1}$ have a common zero at a point $y$ on $\Sigma_k$. There is $\xi \in C_{k+1}$ such that $y\in \langle \xi \rangle$. This means that $\langle p_0,\dots,p_{2k+1} \rangle$ does not separate $\xi$, so that it does not separate $\xi+\xi'$ either, for any  $\xi'\in C_{k+1}$. Then Lemma \ref{lem:phivanishes} shows that
	$$
	\phi(p_0 \wedge \cdots \wedge p_{2k+1})(\xi, \xi')=0~~\text{ for any $\xi' \in C_{k+1}$}.
	$$
	Furthermore, if $x\in X$ is a general point and if $p_0,\dots,p_{2k+1}$ have a unique common zero $x$ on $\Sigma_k$, then they induce an inner projection $\pi\colon \Sigma_k \dashrightarrow \nP^{2k+1}$ of degree $d-1$. If $\{y_1,\dots,y_{d-1}\}$ is a general fiber, and if $\xi_1,\dots,\xi_{d-1} \in C_{k+1}$ are the unique schemes such that $y_i \in \langle \xi_i \rangle$, then the same arguments used for point (i) show that
	$$
	\phi(p_0 \wedge \cdots \wedge p_{2k+1})(\xi_i, \xi_j) =
	\begin{cases} \text{nonzero} & \text{if $i=j$} \\ \text{zero} & \text{if $i \neq j$} \end{cases}.
	$$
	which is what we needed to prove.
	\medskip

	\noindent
	The previous discussion shows that we can apply Theorem \ref{thm:admrepresentation}, so that $\phi$ induces an admissible determinantal representation of degree $d$ of $\Sigma_d$, whose corresponding Ulrich sheaf $\sA_{k,L\otimes \alpha}$ is the image of the evaluation map
	\[  H^0(C_{k+1},A_{k+1,L\otimes \alpha}) \otimes \sO_{\Sigma_k} \longrightarrow \widetilde{\sA}_{k,L\otimes \alpha}. \]
	\medskip

	\noindent
	Assume now that $V$ is not necessarily $k$-very ample. Then the map $\beta_k\colon B_k \dashrightarrow \Sigma_k$ is only rational, but it can be resolved via two birational projective maps $\varepsilon_k \colon \widetilde{B}_k \to B_k$ and $\widetilde{\beta}_k\colon \widetilde{B}_k \to \Sigma_k$, such that $\widetilde{\beta}_k = \beta_k \circ \varepsilon_k$. Notice that  $\varepsilon_{k,*}\sO_{\widetilde{B_k}} = \sO_{B_k}$, so that  $H^0(\widetilde{B}_k,\varepsilon_k^*\pi_k^*A_{k+1,L\otimes \alpha}) \cong H^0(B_k,\pi_k^*A_{k+1,L\otimes \alpha}) \cong H^0(C_{k+1},A_{k+1,L\otimes \alpha})$. Consider now the coherent sheaf $\widetilde{\sA}_{k,L\otimes \alpha} = \widetilde{\beta}_{k,*}(\varepsilon_k^* \pi_k^* A_{k+1,L\otimes \alpha})$ on $\Sigma_k$. We see that $H^0(\Sigma_k,\widetilde{\sA}_{k,L\otimes \alpha}) \cong H^0(C_{k+1},A_{k+1,L\otimes \alpha})$, and then we can use the same proof of the case when $V$ is $k$-very ample.
	\medskip

	\noindent
	Finally, if $\alpha$ is a theta characteristic, Remark \ref{rem:szegosymm}, shows that the Szeg\H{o} kernel $S_{(\alpha,\alpha)}$ is symmetric. Furthermore, the addition map $\sigma\colon C_{k+1}\times C_{k+1} \to C_{2k+2}$ is also symmetric, so that for any $p_0,\dots,p_{2k+1}\in V$ we see that \[ \phi(p_0\wedge \dots \wedge p_{2k+1}) = \sigma^*\det_{2k+2,L}\left( p_0 \wedge \dots \wedge p_{2k+1}  \right)\cdot S_{(\alpha,\alpha)} \]  is a symmetric element of $H^0(C_{k+1},A_{k+1,L\otimes \alpha}) \otimes H^0(C_{k+1},A_{k+1,L\otimes \alpha})$. This proves that $\phi$ yields a symmetric determinantal representation.
\end{proof}

\begin{remark}
	We have defined $\phi$ via the Szeg\H{o} kernels but we could have used Cauchy kernels instead. Indeed
	\begin{align*}
	\phi(p_0\wedge \dots \wedge p_{2k+1})(\xi,\xi') &=  P(p_0\wedge \dots \wedge p_n)(\xi,\xi') \cdot S_{(\alpha,\alpha')}(\xi,\xi') \\
	& = P(p_0\wedge \dots \wedge p_n)(\xi,\xi')  \cdot E(\xi,\xi')\cdot K_{(\alpha,\alpha')}(\xi,\xi') \\
	& = Q(p_0\wedge \dots \wedge p_{2k+1})(\xi,\xi')\cdot K_{(\alpha,\alpha')}(\xi,\xi'),
	\end{align*}
	where
	\[ Q(p_0\wedge \dots \wedge p_{2k+1})(\xi,\xi') =P(p_0\wedge \dots \wedge p_n)(\xi,\xi')  \cdot E(\xi,\xi').\]
This is the way taken by \cite{shamovichvinnikov} in the case $k=0$, and it is of course equivalent to the definition with the Szeg\H{o} kernels, but in the rest of the paper we will keep using Szeg\H{o} kernels rather than Cauchy kernels.
\end{remark}

Theorem \ref{thm:detrep}, gives a way to construct an admissible determinantal representation, or equivalently an Ulrich sheaf, if  $\deg(\Sigma_k)=h^0(C_{k+1},A_{k+1,L\otimes \alpha})=h^0(C_{k+1},A_{k+1,L\otimes \alpha'})$.  This condition can be checked directly for rational normal curves.

\subsection{Explicit admissible determinantal representations for secant varieties of rational normal curves} \label{subsec:explicitdetrep}
Here we show Theorem \ref{thm:rat_norm}.
On $C = \nP^1$, let $t$ be the rational function which has a simple pole at $(0:1)$ and a simple zero at $(1:0)$ but has no other pole or zero. Then $\nP^1 = \nA^1 \cup \{\infty\}$ where $\infty = \{(0:1)\}$ and $t$ serves as an affine coordinate on $\nA^1$. We can naturally identify $H^0(\nP^1,\sO_{\nP^1}(n))$ with the vector space $\kk[t]_{\leq n}$ of univariate polynomials of degree at most $n$. For any $m\geq 1$, there is an isomorphism $(\nP^1)_m \cong \nP^m$, which is described as follows: if $t_1,\dots,t_m \in \nA^1$ we consider the polynomial
	\[ (t-t_1)\cdots (t-t_m) = t^m - e_1(t_1,\dots,t_m)t^{m-1}+ \dots +(-1)^m e_m(t_1,\dots,t_m), \]
	where the $e_i(t_1,\dots,t_m)$ are the elementary symmetric functions. Then the isomorphism is given by
	\[ (\nP^1)_m \longrightarrow \nP^m, \qquad t_1+\dots+t_m \longmapsto [1,e_1(t_1,\dots,t_m),\dots,e_m(t_1,\dots,t_m)] . \]
	More precisely, this is only defined over $(\nA^1)_m$ but it can be extended to the whole of $(\nP^1)_m$. The fact that this is indeed an isomorphism is the fundamental theorem of symmetric functions.
	This realizes $\nP^m$ as the projective space whose points correspond to polynomials $P(t)\in \kk[t]_{\leq m}$, up to rescaling. The locus $\Delta_m \subseteq (\nP^1)_m$ corresponds to polynomials with a double root, meaning that the resultant of the polynomial with its derivatives vanishes. This resultant is an expression of degree $2(m-1)$ in the coefficients, so that $\sO_{(\nP^1)_m}(\delta_m) \cong \sO_{\nP^m}(m-1)$. Furthermore $S_{m,\sO_{\nP^1}(n)} \cong \sO_{\nP^m}(n)$ so that $N_{m,\sO_{\nP^1}(n)} \cong \sO_{\nP^m}(n-m+1)$ and $A_{m,\sO_{\nP^1}(n)} \cong \sO_{\nP^m}(n-2m+2)$. Finally, we have an identification
	\begin{align*}
		H^0(\nP^m,\sO_{\nP^m}(d)) \cong \kk[e_1,\dots,e_m]_{\leq d}.
	\end{align*}
	This space has a particularly nice basis corresponding to partitions $\lambda$ fitting into an $m\times d$ box, meaning that
	\[ \lambda = (\lambda_1,\dots,\lambda_m), \quad d\geq \lambda_1 \geq \lambda_2  \geq \dots \geq \lambda_m \geq 0. \]
	Given such a partition, the corresponding Schur polynomial $s_{\lambda}(t_1,\dots,t_m)$ is a symmetric function of $t_1,\dots,t_m$ that can be expressed explicitly in terms of the elementary symmetric polynomials: let $\lambda' = (m\geq \lambda'_1 \geq \dots \geq \lambda_d' \geq 0)$ be the partition obtained by transposing the Young diagram of $\lambda$. Then the dual Jacobi--Trudi formula states that
	\begin{equation}\label{eq:dualjacobitrudi}
		s_{\lambda} =
		\det
		\begin{pmatrix}
			e_{\lambda'_1} & e_{\lambda'_1+1} & \dots & e_{\lambda'_1+(d-1)} \\
			e_{\lambda'_2-1} & e_{\lambda'_2} & \dots & e_{\lambda'_2+(d-2)} \\
			\vdots & \vdots & \ddots & \vdots \\
			e_{\lambda'_d-(d-1)} & e_{\lambda'_d-(d-2)} & \dots & e_{\lambda'_d}
		\end{pmatrix},
	\end{equation}
	where one sets $e_0(t_1,\dots,t_m)=1$ and $e_{i}(t_1,\dots,t_m)=0$ if $i<0$ or $i>m$. The expression \eqref{eq:dualjacobitrudi} shows that if $\lambda$ is a partition fitting in an $m\times d$ box, then $s_{\lambda}$ is a polynomial of degree at most $d$ in $e_1,\dots,e_m$. Furthermore, there are $\binom{d+m}{m}$ such partitions and the corresponding Schur polynomials $s_{\lambda}$ are linearly independent. This shows that
	\[ \left\{ s_{\lambda} \,|\, \lambda \text{ in a } (m\times d)\text{-box } \right\} \text{ is a basis of } \kk[e_1,\dots,e_m]_{\leq d} \cong H^0(\nP^m,\sO_{\nP^m}(d)). \]
	Consider now the embedding $C\hookrightarrow \nP^n = \nP(V_n)$ given by $V_n = H^0(\nP^1,\sO_{\nP^1}(n))$. The image is a rational normal curve of degree $n$. We assume $n\geq 2k+2$, so that the secant $\Sigma_k=\Sigma_k(\nP^1,\sO_{\nP^1}(n))$ is a proper subvariety of $\nP^n$ of dimension $2k+1$. We also fix the unique line bundle $\alpha=\sO_{\nP^1}(-1)$ with no cohomology and we want to compute the map  
	\[ \phi=\phi_{V_n,k,\alpha} \colon \wedge^{2k+2}V_n \longrightarrow \begin{matrix} H^0((\nP^1)_{k+1},A_{k+1,\sO_{\nP^1}(n-1)}) \\ \otimes \\ H^0((\nP^1)_{k+1},A_{k+1,\sO_{\nP^1}(n-1)}) \end{matrix} \cong 
	\begin{matrix} H^0(\nP^{k+1},\sO_{\nP^{k+1}}(n-1-2k)) \\ \otimes \\ H^0(\nP^{k+1},\sO_{\nP^{k+1}}(n-1-2k) \end{matrix}\] 
	of Theorem \ref{thm:admrepresentation}. The first step is the isomorphism:
	\[ {\det}_{2k+2}\colon \wedge^{2k+2}V_n \longrightarrow H^0((\nP^1)_{2k+2},N_{2k+2,\sO_{\nP^1}(n)}) \cong H^0(\nP^{2k+2},\sO_{\nP^{2k+2}}(n-2k-1)) . \]
	Let $\lambda=(\lambda_1,\dots,\lambda_{2k+2})$ be a partition fitting into a $(2k+2)\times (n-2k-1)$ box. For any such partition we set
	\[  \wedge^{\lambda}t := t^{\lambda_1+2k+1}\wedge t^{\lambda_2+2k} \wedge \dots \wedge t^{\lambda_{2k+1}+1} \wedge t^{\lambda_{2k+2}} \in \wedge^{2k+2}V_n\]
	and in this way we get a basis of $\wedge^{2k+2}V_n$. Then it holds that
	\[ {\det}_{2k+2}(\wedge^{\lambda}t) = \frac{\det
		\begin{pmatrix}
			t_1^{\lambda_1+2k+1} & t_2^{\lambda_1+2k+1} & \dots & t_{2k+2}^{\lambda_1+2k+1} \\
			t_1^{\lambda_2+2k} & t_2^{\lambda_2+2k} & \dots & t_{2k+2}^{\lambda_2+2k} \\
			\vdots & \vdots & \ddots & \vdots  \\
			t_1^{\lambda_{2k+2}} & t_2^{\lambda_{2k+2}} & \dots & t_{2k+2}^{\lambda_{2k+2}}
	\end{pmatrix}}{\prod_{1\leq i < j \leq 2k+2}(t_i-t_j)}  = s_{\lambda}(t_1,\dots,t_{2k+2}),
	\]
	where the last equality is the bialternant formula for Schur polynomials. Now we want to compute the pullback $\sigma_{k+1,k+1}^*s_{\lambda}$ along the addition map $\sigma_{2k+2}\colon \nP^{k+1} \times \nP^{k+1} \longrightarrow \nP^{2k+2}$. If we use the basis of $H^0((\nP^1)_{k+1},A_{k+1,\sO_{\nP^1}(n-1)}) \cong H^0(\nP^{k+1},\sO_{\nP^{k+1}}(n-2k-1))$ given by the partitions $\mu = (\mu_1,\dots,\mu_{k+1})$ contained in a $(k+1)\times (n-2k-1)$ box, we see from the coproduct formula of \cite[Formula (5.9), page 72]{macdonald} that
	\[ s_{\lambda}(t_1,\dots,t_{k+1},t_{k+2},\dots,t_{2k+2}) = \sum_{\mu,\nu} c^{\lambda}_{\mu\nu} \cdot s_{\mu}(t_1,\dots,t_{k+1})s_{\nu}(t_{k+2},\dots,t_{2k+2}),\]
	where the $c^{\lambda}_{\mu,\nu}$ are the Littlewood--Richardson coefficients. This shows that $\sigma_{k+1,k+1}^*(s_{\lambda}) = \sum_{\mu,\nu} c^{\lambda}_{\mu,\nu}\cdot ( s_{\mu}\boxtimes s_{\nu})$. Finally the Szeg\H{o} kernel in our case is a section of $(S_{k+1,\sO_{\nP^1}(-1)}\boxtimes S_{k+1,\sO_{\nP^1}(-1)}) \otimes \sO_{\nP^{k+1} \times \nP^{k+1}}(D_{k+1,k+1}) \cong \sO_{\nP^{k+1}\times \nP^{k+1}}$ so that $S_{(\alpha,\alpha)}=1$ and then
	\[ \phi(\wedge^{\lambda}t) = \sum_{\mu,\nu} c^{\lambda}_{\mu,\nu} \cdot (s_{\mu}\boxtimes s_{\nu}). \]
	
	\begin{remark}\label{rmk:coproduct}
	More intrinsically, we can rephrase this by saying that the map $\phi$ corresponds to the coproduct of symmetric functions that takes a symmetric function in $2k+2$ variables and expresses it in terms of symmetric functions in $k+1$ variables.
	\end{remark}
	
	We can represent the map $\varphi$ by an explicit matrix. Let us fix the basis $(x_{\lambda})$ of $(\wedge^{2k+2}V_{n})^{\vee}$ which is dual to the basis $(\wedge^{\lambda}t)$ indexed by partitions fitting in a $(2k+2)\times (n-2k-1)$ box. We also fix the basis $s_{\mu}$ of $H^0(\nP^{k+1},\sO_{\nP^{k+1}}(n-2k-1))$ indexed by partitions fitting in a $(k+1)\times (n-2k-1)$ box. Then we can represent $\phi$ via the matrix $M^{n,k}(x_{\lambda})$ of linear forms such that
	\begin{equation}
		M^{n,k}_{\mu,\nu}(x_{\lambda}) = \sum_{\lambda} c^{\lambda}_{\mu,\nu}\cdot x_{\lambda}.
	\end{equation}
	This gives an admissible determinantal representation of $\Sigma_k=\Sigma_k(\nP^1,\sO_{\nP^1}(n))$, which is the statement of Theorem \ref{thm:rat_norm} from the introduction.
	
	\begin{theorem}[Theorem \ref{thm:rat_norm}]
		If $n\geq 2k+2$, the matrix $M^{n,k}(x_{\lambda})$ gives a symmetric admissible determinantal representation of rank one of $\Sigma_k$ 
	\end{theorem}
	
\begin{proof}This follows from the previous discussion and from Theorem \ref{thm:admrepresentation} if we can show that $\deg(\Sigma_k)=h^0((\nP^1)_{k+1},A_{k+1,\sO_{\nP^1}(n-1)})$. We have seen in the previous discussion that $h^0((\nP^1)_{k+1},A_{k+1,\sO_{\nP^1}(n-1)})=h^0(\nP^{k+1},\sO_{\nP^{k+1}})(n-2k-1)) = \binom{n-k}{k+1}$, and it is a classical result that this is equal to the degree of $\Sigma_k$: see for example \cite[Proposition 1]{soule}\footnote{We note that \cite[Proposition 1]{soule} holds in the assumption that the embedding of the curve is $k$-very ample, even if the author does not mention this assumption explicitly. This is because the proof uses the fact that the map $\beta_k\colon B^k\to \Sigma_k$ is a morphism}.
\end{proof}

\begin{example}Following the previous notation, an admissible determinantal representation for the first secant variety of the rational normal curve of degree $5$ is given by the matrix
	\[
	 \begin{pmatrix}
		x_{{\emptyset}} & x_{{(1)}} & x_{{(2)}} & x_{{(1,1)}} & x_{{(2,1)}} & x_{{(2,2)}} \\
		x_{{(1)}} & x_{{(2)}} + x_{{(1,1)}} & x_{{(2,1)}} & x_{{(2,1)}} + x_{{(1,1,1)}} & x_{{(2,2)}} + x_{{(2,1,1)}} & x_{{(2,2,1)}} \\
		x_{{(2)}} & x_{{(2,1)}} & x_{{(2,2)}} & x_{{(2,1,1)}} & x_{{(2,2,1)}} & x_{{(2,2,2)}} \\
		x_{{(1,1)}} & x_{{(2,1)}} + x_{{(1,1,1)}} & x_{{(2,1,1)}} & x_{{(2,2)}} + x_{{(2,1,1)}} + x_{{(1,1,1,1)}} & x_{{(2,2,1)}} + x_{{(2,1,1,1)}} & x_{{(2,2,1,1)}} \\
		x_{{(2,1)}} & x_{{(2,2)}} + x_{{(2,1,1)}} & x_{{(2,2,1)}} & x_{{(2,2,1)}} + x_{{(2,1,1,1)}} & x_{{(2,2,2)}} + x_{{(2,2,1,1)}} & x_{{(2,2,2,1)}} \\
		x_{{(2,2)}} & x_{{(2,2,1)}} & x_{{(2,2,2)}} & x_{{(2,2,1,1)}} & x_{{(2,2,2,1)}} & x_{{(2,2,2,2)}}
	\end{pmatrix}.
	\]
\end{example}

If the curve $C$ has higher genus, it is difficult to check directly that the conditions of Theorem \ref{thm:admrepresentation} hold. What we do instead is to prove that a sheaf is Ulrich via cohomology vanishing, when the embedding is $k$-very ample.

\section{Ulrich sheaves on secants of curves}\label{sec:constulrich}
\noindent This section is devoted to showing the existence of rank one Ulrich sheaves on secant varieties of curves and thereby completing the proof of Theorem \ref{thm:main}. We also explain how to explicitly construct the minimal free resolution of the Ulrich sheaves in Subsection \ref{subsec:explicitminres}.

\subsection{Existence of Ulrich sheaves on secant varieties of curves}
We assume that  $C\subseteq \nP^r=\nP(V)$ is a smooth projective curve of genus $g$ over a field $\kk$ of characteristic zero,  embedded by a linear system $V\subseteq H^0(C,L)$ of a very ample line bundle $L=\sO_C(1)$. In this section, we assume $r\geq2k+2$ and that \emph{$V$ is $k$-very ample}.
We also assume to have a line bundle $\alpha$ on $C$ with $h^0(C,\alpha)=h^1(C,\alpha)=0$.

\medskip

Since $V$ is $k$-very ample, it generates the bundle $E_{k+1,L}$ on $C_{k+1}$. Thus
$$
\beta_{V, k} \colon B^k \longrightarrow \Sigma_k = \Sigma_k(C, V; L)
$$
is a morphism, where we recall $B^k=B^k(L)=\nP(E_{k+1,L})$ with the canonical projection
$$
\pi_k \colon B^k \longrightarrow C_{k+1}.
$$
The main result of this section is the following theorem.

\begin{theorem}\label{thm:ulrichdeg>=2g+2k+1}
If $V$ is $k$-very ample, then the torsion-free sheaf 
$$
\widetilde{\mathscr{A}}_{k, L \otimes \alpha}:=\beta_{V,k,*} \pi_k^* A_{k+1,L \otimes \alpha}.
$$
is an Ulrich sheaf on $\Sigma_k$. If $r \geq 2k+2$, then this Ulrich sheaf has rank one and it coincides with the Ulrich sheaf $\sA_{k,L\otimes \alpha}$ of Theorem \ref{thm:detrep}.
\end{theorem}

To prove the theorem, we need to show that
$$
H^i(\nP^r, \widetilde{\mathscr{A}}_{k, L \otimes \alpha}(-i))=H^{i-1}(\nP^r, \widetilde{\mathscr{A}}_{k, L \otimes \alpha}(-i))=0~\text{ for $1 \leq i \leq 2k+1$}.
$$
We verify the required cohomology vanishing with a series of lemmas. 

\begin{lemma}\label{lemma:vanishingwedgeM}
Fix integers $i \geq 0$ and $j \geq 0$ and let $\mathscr{G}$ be a coherent sheaf on $C_{k+1}$ such that
$
H^{i+\ell}(C_{k+1},\wedge^{\ell} M_{k+1,L}\otimes \mathscr{G})=0 ~~\text{ for $0 \leq \ell \leq j$}
$.
Then
$
H^i(C_{k+1}, S^j E_{k+1,L} \otimes \mathscr{G})=0.
$
\end{lemma}

\begin{proof}
Since the linear system $V$ is  $k$-very ample, the complete linear system is $k$-very ample as well, and there is a short exact sequence of vector bundles
$$
0 \longrightarrow  M_{k+1,L} \longrightarrow  H^0(C,L)\otimes \sO_{C_{k+1}} \longrightarrow  E_{k+1,L} \to 0.
$$
For every $j \geq 0$, this induces an exact complex
\begin{multline*}
0 \longrightarrow \wedge^j M_{k+1,L} \longrightarrow H^0(C,L) \otimes \wedge^{j-1} M_{k+1,L} \longrightarrow
\cdots
\longrightarrow  S^{j-2} H^0(C,L) \otimes \wedge^{2} M_{k+1,L} \\
\longrightarrow  S^{j-1}H^0(C,L) \otimes M_{k+1,L}
\longrightarrow  S^j H^0(C,L) \otimes \sO_{C_{k+1}} \longrightarrow S^j E_{k+1,L} \longrightarrow 0.
\end{multline*}
which stays exact after tensoring with $\mathscr{G}$, since it is a complex of locally free sheaves.  Then the assertion follows from chasing through complexes.
\end{proof}

\begin{lemma}\label{lemma:cohvanMA}
For any $i>0,j\geq 0$ it holds that
$
H^i(C_{k+1}, \wedge^j M_{k+1,L} \otimes A_{k+1,L \otimes \alpha})=0.
$
\end{lemma}

\begin{proof}
Since $L$ is $k$-very ample and $H^1(C,(L\otimes \alpha)\otimes L^{-1}) = H^1(C,\alpha)=0$, the assertion follows from Proposition \ref{prop:rathmannvanishing}.
\end{proof}

\begin{lemma}\label{lemma:wedgeMSvanishing}
	For any $j\geq 0$ and $0\leq \ell\leq k$ it holds that
	$$
	H^{j+\ell}(C_{k+1}, \wedge^{\ell} M_{k+1,L} \otimes S_{k+1,\alpha'})=0.
	$$
\end{lemma}

\begin{proof}
	By \cite[Proposition 3.1]{NP}, it is sufficient to check the following:
	\begin{enumerate}
		\item $H^{j+\ell}(C_{k+1} \times C_{\ell}, (S_{k+1, \alpha'} \boxtimes N_{k+1,L})(-D_{k+1, \ell})=0$.
		\item $H^{j+\ell-m}(C_{k+1} \times C_{\ell-n}, (S_{k+1,\alpha'} \boxtimes N_{\ell-n, L})(-D_{k+1, \ell-n})=0$ for all integer $m,n$ with $1 \leq n \leq j+\ell$ and $n+1 \leq m \leq 2n$.
	\end{enumerate}
	For an integer $0 \leq n \leq \ell$, let $\pr_2^{\ell-n} \colon C_{k+1} \times C_{\ell-n} \to C_{\ell-n}$ be the projection to the second component. For any $\xi \in C_{\ell-n}$, consider the fiber $C_{k+1}\times C_{\ell-n}$. Note that
	$$
	(S_{k+1,\alpha'} \boxtimes N_{\ell-n, L})(-D_{k+1, \ell-n})|_{C_{k+1}\times \{\xi\}} \cong S_{k+1, \alpha'(-\xi)}.
	$$
	Since $h^0(C, \alpha'(-\xi))=0$ and $h^1(C, \alpha'(-\xi))=\ell-n \leq k$, it follows from Lemma \ref{lem:H^i(N)} that
	$$
	H^i(C_{k+1}, S_{k+1, \alpha'(-\xi)})=0~~\text{ for $i \geq 0$}.
	$$
	Thus we obtain
	$$
	R^i \pr_{2,*}^{\ell-n} (S_{k+1,\alpha'} \boxtimes N_{\ell-n, L})(-D_{k+1, \ell-n}) = 0~~\text{ for $i \geq 0$}.
	$$
	By the Leray spectral sequence for $\pr_2^{\ell-n}$, we get
	$$
	H^{j+\ell-m}(C_{k+1} \times C_{\ell-n}, (S_{k+1,\alpha'} \boxtimes N_{\ell-n, L})(-D_{k+1, \ell-n})=0~~\text{ for $m \geq 0$ and $n \geq 0$},
	$$
	which verifies the above conditions $(1)$ and $(2)$.
\end{proof}

\begin{lemma}[{cf. \cite[Theorem 5.2 (2)]{ENP}}]\label{lemma:DuBois-like-cond}
For any $i>0$ it holds that
$
R^i \beta_{V,k,*} \pi_k^*A_{k+1, L \otimes \alpha} = 0.
$
\end{lemma}

\begin{proof}
Notice that the lemma is equivalent to
$$
H^i(B^k, \pi_k^* A_{k+1, L \otimes \alpha} \otimes \sO_{B^k}(\ell)) = 0~~\text{ for $i>0$ and $\ell \gg 0$}.
$$
As $R^j \pi_{k,*}\sO_{B^k}(\ell) = 0$ for $j > 0$ and $\ell \geq 0$, we have
$$
H^i(B^k, \pi_k^* A_{k+1, L \otimes \alpha} \otimes \sO_{B^k}(\ell)) \cong H^i(C_{k+1}, S^{\ell} E_{k+1, L} \otimes A_{k+1, L \otimes \alpha})~~\text{ for $i \geq 0$ and $\ell \geq 0$.}
$$
This cohomology vanishes by  Lemmas \ref{lemma:vanishingwedgeM} and \ref{lemma:cohvanMA}.
\end{proof}

\begin{lemma}\label{lemma:cohfromBtoC}
For any $j\geq 0$ and $1\leq i\leq 2k+1$ it holds that
\begin{align*}
H^j(\nP^r, \widetilde{\mathscr{A}}_{k,L}(-i)) &\cong H^j(B^k, \pi_k^*A_{k+1, L \otimes \alpha} \otimes \sO_{B^k}(-i))\\
& \cong H^{2k+1-j}(C_{k+1},S^{i-k-1}E_{k+1,L}\otimes S_{k+1,\alpha'} )^{\vee}.
\end{align*}
\end{lemma}

\begin{proof}
The first equality  follows immediately from Lemma \ref{lemma:DuBois-like-cond}. For the second equality, we use  Serre duality:
$$
H^j(B_k, \pi_k^*A_{k+1, L \otimes \alpha} \otimes \sO_{B^k}(-i)) \cong H^{2k+1-j}(B^k, \pi_k^* S_{k+1,\alpha'} \otimes \sO_{B^k}(k+1-i))^{\vee},
$$
where we recall $\omega_{B^k} \cong \pi_k^* A_{k+1, \omega_C \otimes L} \otimes \sO_{B^{k}}(-k-1)$.
As $1 \leq i \leq 2k+1$, we have $R^{\ell} \pi_{k,*} \sO_{B^k}(k+1-i) = 0$ for all $\ell \geq 0$. Thus
\[
H^{2k+1-j}(B^k, \pi_k^* S_{k+1,\alpha'} \otimes \sO_{B^k}(k+1-i))^{\vee}
\cong H^{2k+1-j}(C_{k+1}, S^{k+1-i} E_{k+1,L} \otimes S_{k+1,\alpha'})^{\vee}. \qedhere
\]
\end{proof}

We are ready to finish the proof of Theorem \ref{thm:ulrichdeg>=2g+2k+1}.

\begin{proof}[Proof of Theorem \ref{thm:ulrichdeg>=2g+2k+1}]
In view of Lemma \ref{lemma:cohfromBtoC}, we need to establish the following:
$$
H^{2k+1-i}(C_{k+1}, S^{i-k-1} E_{k+1,L} \otimes S_{k+1,\alpha'})=H^{2k+2-i}(C_{k+1}, S^{i-k-1}E_{k+1,L} \otimes S_{k+1,\alpha'})=0
$$
for $k+1 \leq i \leq 2k+1$. It is equivalent to
$$
H^{k-i}(C_{k+1}, S^i E_{k+1,L} \otimes S_{k+1, \alpha'}) = H^{k+1-i}(C_{k+1}, S^i E_{k+1,L} \otimes S_{k+1, \alpha'}) = 0
$$
for $0 \leq i \leq k$. By Lemma \ref{lemma:vanishingwedgeM}, we reduce the problem to
$$
H^{k-i+\ell}(C_{k+1}, \wedge^{\ell} M_{k+1,L} \otimes S_{k+1,\alpha'})=H^{k+1-i+\ell}(C_{k+1}, \wedge^{\ell} M_{k+1,L} \otimes S_{k+1,\alpha'})=0
$$
for $0 \leq i \leq k$ and $0 \leq \ell \leq i$. This follows from Lemma \ref{lemma:wedgeMSvanishing}. We have shown that $\widetilde{\mathscr{A}}_{k, L \otimes \alpha} = \beta_{V,k,*} \pi_k^* A_{k+1, L \otimes \alpha}$ is an Ulrich sheaf on $\Sigma_k$. For the second statement of the theorem, we suppose that $r \geq 2k+2$.
Then Lemma \ref{lem:beta=bir} says that $\beta_{V,k}$ is birational. In this case, the Ulrich sheaf $\widetilde{\mathscr{A}}_{k, L \otimes \alpha}$ has rank one so that $\deg(\Sigma_k)=h^0(C_{k+1},A_{k+1,L\otimes \alpha})$. The same principle shows that $\deg(\Sigma_k) = h^0(C_{k+1},A_{k+1,L\otimes \alpha'})$. Then Theorem \ref{thm:detrep} provides an admissible determinantal representation of $\Sigma_k$ with corresponding Ulrich sheaf $\sA_{k,L\otimes \alpha}$. The proof of Theorem \ref{thm:detrep} shows that $\sA_{k,L\otimes \alpha}$ is the image of the evaluation map
\[ H^0(C_{k+1},A_{k+1,L\otimes \alpha})\otimes \sO_{\Sigma_k} \longrightarrow \widetilde{\sA}_{k,L\otimes \alpha}, \]
but since we already proved that $\widetilde{\sA}_{k,L\otimes \alpha}$ is an Ulrich sheaf, we know that it is globally generated, so that $\sA_{k,L\otimes \alpha} = \widetilde{\sA}_{k,L\otimes\alpha}$.
\end{proof}

To give a complete proof of Theorem \ref{thm:main}, we need a final observation:

\begin{proposition}\label{prop:isoUlrich}
	Let $M,  M'$ be any line bundles on $C$. Then we have $\beta_{V,k,*} \pi_k^* A_{k+1, M} \cong \beta_{V,k,*} \pi_k^* A_{k+1, M'}$ if and only if $M \cong M'$.
\end{proposition}

\begin{proof}
By Lemma \ref{lem:beta=bir}, $\beta_{V,k} \colon B^k \to \Sigma_k$ is a birational morphism, which is factored as
$$
B^k  \xrightarrow{~\overline{\beta}_{V,k}~} \overline{\Sigma}_k \xrightarrow{~\eta_k~} \Sigma_k,
$$
where $\eta_k$ is the normalization map of $\Sigma_k$. As $\overline{\beta}_{k,*}  \pi_k^* A_{k+1, M},  \overline{\beta}_{k,*} \pi_k^* A_{k+1, M'}$ are torsion-free sheaves on $\overline{\Sigma}_k$, we see that $\beta_{V, k,*} \pi_k^* A_{k+1, M} \cong \beta_{V, k,*} \pi_k^* A_{k+1, M'}$ if and only if $\overline{\beta}_{k,*}  \pi_k^* A_{k+1, M} \cong  \overline{\beta}_{k,*} \pi_k^* A_{k+1, M'}$. Now, by Lemma \ref{lem:beta=bir}, there is an open subset $U \subseteq B^k$ such that $\overline{\beta}_k \colon U \to \overline{\beta}_k(U)$ is an isomorphism and the kernel of the group homomorphism $\operatorname{Cl}(B^k) \to \operatorname{Cl}(U)$ is generated by $Z_{k-1}$.\footnote{If $\Sigma_k$ is normal or more generally the singular locus of $\Sigma_k$ has codimension at least two, then the argument can be applied directly to $\beta_{V,k}$ without passing to the normalization of $\Sigma_k$. However, even under the assumption that $V$ is $k$-very ample, $\Sigma_k$ may have a singular locus of codimension one (see e.g., \cite[Example 4.17]{kummersinn}), so it is necessary to consider the normalization of $\Sigma_k$.}
If $\overline{\beta}_{k,*} \pi_k^* A_{k+1, M} \cong \overline{\beta}_{k,*} \pi_k^* A_{k+1, M'}$ (i.e., $\beta_{V, k,*} \pi_k^* A_{k+1, M} \cong \beta_{V, k,*} \pi_k^* A_{k+1, M'}$), then $(\pi_k^* A_{k+1, M})|_U \cong (\pi_k^* A_{k+1,M'})|_U$ so that $(\pi_k^* S_{k+1, M \otimes {M'}^{-1}})|_U \cong \sO_U$. Thus $\pi_k^* S_{k+1, M \otimes {M'}^{-1}} \cong \sO_{B^k}(mZ_{k-1})$ for some $m \in \mathbf{Z}$. But it is clear that $m=0$, so $\pi_k^* S_{k+1, M \otimes {M'}^{-1}} \cong \sO_{B^k}$, which implies that $ S_{k+1, M \otimes {M'}^{-1}} \cong \sO_{C_{k+1}}$. As the group homomorphism $\Pic(C) \to \Pic(C_{k+1}),~N \mapsto S_{k+1,N}$ is injective, we obtain $M \cong M'$.
\end{proof}

\begin{proof}[Proof of Theorem \ref{thm:main}]
	This follows from Theorems \ref{thm:detrep}, \ref{thm:ulrichdeg>=2g+2k+1}, and Proposition \ref{prop:isoUlrich}.
\end{proof}

\subsection{The minimal free resolution of $\mathscr{A}_{k,L\otimes \alpha}$}\label{subsec:explicitminres}

The minimal free resolution of the Ulrich sheaf $\mathscr{A}_{k,L\otimes \alpha}$ can be computed explicitly. As always, let $C\subseteq \nP^r = \nP(V)$ be a smooth projective curve of genus $g$ embedded by a very ample line bundle $L=\sO_C(1)$ and a linear system $V\subseteq H^0(C,L)$ of dimension $\dim V=r+1$. We assume that $r\geq 2k+2$ and that the embedding is $k$-very ample. We also fix a line bundle $\alpha$ on $C$ with no cohomology so that  Theorem \ref{thm:main} yields an Ulrich sheaf of rank one $\sA_{k,L\otimes \alpha}$ on the secant variety $\Sigma_k$.
\medskip

\noindent
Assume that we are able to do the following:
\begin{itemize}
	\item Compute a basis of $H^0(C,L\otimes \alpha)$.
	\item Compute the multiplication map $V\otimes H^0(C,L\otimes \alpha) \to H^0(C,L^{\otimes 2}\otimes \alpha)$.
	\item Compute a basis of $H^0(C_{k+1},A_{k+1,L\otimes \alpha})$ as a subspace of $H^0(C_{k+1},S_{k+1,L\otimes \alpha}) = S^{k+1}H^0(C,L\otimes \alpha)$.
\end{itemize}

\noindent
Then we explain how to construct explicitly the minimal free resolution of the Ulrich sheaf $\mathscr{A}_{k,L\otimes \alpha}$.
If we denote by $c:=r-2k-1$ the codimension of $\Sigma_k$, the minimal free resolution has the shape
\begin{equation}\label{eq:mfrulrich}
	0 \longrightarrow K_{c} \otimes \sO_{\mathbf{P}^r}(-c) \overset{M_{c}}{\longrightarrow} \cdots \longrightarrow K_1 \otimes \sO_{\mathbf{P}^r}(-1)  \overset{M_1}\longrightarrow K_0 \otimes \sO_{\mathbf{P}^r} {\longrightarrow} \mathscr{A}_{k,L\otimes \alpha} \longrightarrow 0,
\end{equation}
where  the vector space $K_h$ is the kernel of the Koszul differential:
\[ K_h = \Ker \left( \wedge^h V \otimes H^0(\nP^r,\sA_{k,L\otimes \alpha}) \overset{d_h}{\longrightarrow} \wedge^{h-1} V \otimes H^0(\nP^r,\sA_{k,L\otimes \alpha}(1)) \right). \]

\noindent
From Theorem \ref{thm:ulrichdeg>=2g+2k+1} we have
\begin{align*}
	&H^0(\nP^r,\mathscr{A}_{k,L\otimes \alpha})  \cong H^0(B^k,\pi_k^*A_{k+1,L\otimes \alpha}) \cong H^0(C_{k+1},A_{k+1,L\otimes \alpha}) \\
	&H^0(\nP^r,\mathscr{A}_{k,L\otimes \alpha}(1))  \cong  H^0(B^k,\pi_k^*A_{k+1,L\otimes \alpha} \otimes \sO_{B^k}(1)) \cong H^0(C_{k+1},A_{k+1,L\otimes \alpha} \otimes E_{k+1,L})
\end{align*}
so that the  Koszul differential $d_h$ can also be seen as a map
\begin{equation}\label{eq:firstkoszul}
	d_h\colon  \wedge^h V\otimes H^0(C_{k+1},A_{k+1,L\otimes \alpha}) \longrightarrow \wedge^{h-1}V\otimes H^0(C_{k+1},A_{k+1,L\otimes \alpha} \otimes E_{k+1,L}).
\end{equation}
We can use this to compute the kernels $K_h$ and the maps $M_h$ in the minimal free resolution \ref{eq:mfrulrich} explicitly. We will discuss the procedure only in the case $h=1$ for simplicity, but this already give a presentation of $\mathscr{A}_{k,L\otimes \alpha}$, and the general case is a straightforward generalization.
\medskip

\noindent
\textbf{Step 1}: Fix a basis $y_1,\dots,y_\ell$ of $H^0(C,L\otimes \alpha)$ and choose $F_1,\dots,F_d$ with
\[ F_i = F_i(y_1,\dots,y_{\ell}) \in S^{k+1}H^0(C,L\otimes \alpha) \quad \text{ for } i=1,\dots,d \]
that form a basis of $H^0(C_{k+1},A_{k+1,L\otimes \alpha})$. For example, if $\deg(L)\geq g+2k+4$, then \cite[Theorem 5.2]{ENP} shows that we can take $F_1,\dots,F_d$ to be a basis of the degree $k+1$ part of the homogeneous ideal of $\Sigma_{k-1}(C,L\otimes \alpha)$.
\medskip

\noindent
\textbf{Step 2}: If $\sigma_{1,k}\colon C \times C_k \to C_{k+1}$ is the addition map, then $\sigma_{1,k}^*(A_{k+1,L\otimes \alpha}) \subseteq (L\otimes \alpha) \boxtimes S_{k,L\boxtimes \alpha}$. Hence, the pullback induces an inclusion
\begin{align*}
	\sigma_{1,k}^*\colon H^0(C_{k+1},A_{k+1,L\otimes \alpha}) \longhookrightarrow  H^0(C,L\otimes \alpha) \otimes S^k H^0(C,L\otimes \alpha)
\end{align*}
that can be computed explicitly on our basis of $H^0(C_{k+1},A_{k+1,L\otimes \alpha})$ by
\[ \sigma_{1,k}^*F_i = \sum_{j=1}^{\ell} y_j \otimes \frac{\partial F_i}{\partial y_j} \qquad \text{ for } i=1,\dots,d. \]

\noindent
\textbf{Step 3}: We also see that $\sigma_{1,k}$ induces an isomorphism $H^0(C_{k+1},A_{k+1,L\otimes \alpha} \otimes E_{k+1,L}) \cong H^0(C\times C_k,(L\boxtimes \sO_{C_k}) \otimes \sigma_{1,k}^*A_{k+1,L\otimes \alpha})$ and reasoning as in Step 2, we see that  this is a subspace of $H^0(C,L^{\otimes 2}\otimes \alpha) \otimes S^kH^0(C,L\otimes \alpha)$. The composition of \eqref{eq:firstkoszul} for $h=1$ with this inclusion is a map
\[ d_1'\colon   V\otimes H^0(C_{k+1},A_{k+1,L\otimes \alpha}) \longrightarrow  H^0(C,L^{\otimes 2}\otimes \alpha) \otimes S^{k-1}H^0(C,L\otimes \alpha) \]
that is given on the basis $F_i,i=1,\dots,d$ by
\[ d_1'(p_1 \otimes F_i) =  \sum_{j=1}^{\ell}  (p_1\cdot y_j) \otimes \frac{ \partial F_i}{\partial y_j} \qquad \text{ for any } p_1\in V. \]
We can use this to compute a basis $s_1,\dots,s_d$ of $K_1 = \operatorname{Ker} d_1 = \operatorname{Ker} d_1'$, which we can write in the form
\[ s_j = \sum_{i=1}^{d} p_{ij} \otimes F_{i} \qquad \text{ for certain } p_{ij}\in V. \]
Then we get a presentation of $\mathscr{A}_{k,L\otimes \alpha}$ via the matrix
\[ K_{1} \otimes \sO_{\nP^r}(-1) \xrightarrow{M_1} H^0(C_{k+1},A_{k+1,L\otimes \alpha})\otimes \sO_{\nP^r}, \qquad M_1 = \left( p_{ij} \right)_{ij}. \]

\begin{example}[Rational normal quartic]
	Consider the curve $C=\nP^1$ embedded in $\nP^4$ with complete linear system of $L=\sO_{\nP^1}(4)$. The first secant variety $\Sigma_1 \subseteq \nP^4$ is a cubic hypersurface,  and if we take  $\alpha=\sO_{\nP^1}(-1)$, the resolution of the Ulrich sheaf $\mathscr{A}_{1,L\otimes \alpha}$ will give us a determinantal representation for it. We compute the resolution using the procedure outlined before: we choose bases of $H^0(C,L)$ and $H^0(C,L\otimes \alpha)$ given respectively by
	\begin{align*}
		(x_0,x_1,x_2,x_3,x_4) = (s^4,s^3t,s^2t^2,st^3,t^4), \qquad
		(y_0,y_1,y_2,y_3) = (s^3,s^2t,st^2,t^3),
	\end{align*}
	and then a basis of $H^0(C_2,A_{2,L\otimes \alpha})$ is given by quadrics that generate the ideal of the curve $C$ embedded by $L\otimes \alpha = \sO_{\nP^1}(3)$. We choose:
	\[ F_1 = y_1y_3-y_2^2, \quad F_2 = y_1y_2-y_0y_3, \quad F_3 = y_0y_2-y_1^2. \]
	One can compute a basis of $K_1$ (here we use the notation of the previous discussion):
	\begin{align*}
		s_1 = x_0 \otimes F_1 + x_1\otimes F_2 + x_2 \otimes F_3 \\
		s_2  = x_1 \otimes F_1 + x_2\otimes F_2 + x_3 \otimes F_3 \\
		s_3  = x_2 \otimes F_1 + x_3\otimes F_2 + x_4 \otimes F_3
	\end{align*}
	and this gives a presentation of $\mathscr{A}_{1,L\otimes \alpha}$ and a corresponding determinantal representation:
	\[ 0 \longrightarrow \sO_{\nP^4}(-1)^{\oplus 4} \xrightarrow{ \begin{pmatrix} x_0 & x_1 & x_2 \\ x_1 & x_2 & x_3 \\ x_2 & x_3 & x_4 \end{pmatrix}}  \sO_{\nP^4}^{\oplus 3} \longrightarrow \sA_{1,L\otimes \alpha} \longrightarrow 0. \]
	This is symmetric, reflecting the fact that $\alpha$ is a theta characteristic.
\end{example}

\begin{example}[A real curve of genus two]
	Consider the real curve $C$ of genus $2$ defined by the hyperelliptic model
	\begin{equation*}
		y^2=(1-x^2)(4-x^2)(9-x^2)
	\end{equation*}
	and the divisor $D$ on $X$ defined as the sum of the following six points:
	\begin{equation*}
		(\pm2,0),(\pm1,0),(0,\pm6).
	\end{equation*}
	The line bundle $L=\sO_C(D)$ embeds $C$ as a curve of degree $6$ in $\nP^4$ and the first secant $\Sigma_1$ is a hypersurface  of degree $8$.
	Using the algorithm outlined above, we compute a determinantal representation of this hypersurface. To this end, we choose the theta characteristic $\alpha$ on $C$ that corresponds to the divisor
	\begin{equation*}
		(-1,0)+(1,0)-(-3,0).
	\end{equation*}
	One checks that it induces a complex orientation on $C(\nR)$ (see Section \ref{sec:signatures} for more details).
	Choosing appropriate bases of $H^0(C,L)$ and $H^0(C,L\otimes \alpha)$, the resulting determinantal representation is symmetric and equal to
	\[
	\scalebox{.45}{$
		\left(
		\begin{array}{cccccccc}
			-585 x_1-1014 x_2-124 x_4 & 13 (18 x_1-x_4) & 3 (39 x_0+8 x_3) & 15 (57 x_0+14 x_3) & -3 (201 x_0+47 x_3) & 54 x_0+13 x_3 & -13 (3 x_1+x_2) & -26 (9 x_1-3 x_2-x_4) \\
			13 (18 x_1-x_4) & 6 x_2-x_4 & -3 (3 x_0-x_3) & -15 (3 x_0-x_3) & 12 (3 x_0-x_3) & x_3-3 x_0 & 3 x_1-x_2 & -2 (9 x_1+3 x_2-x_4) \\
			3 (39 x_0+8 x_3) & -3 (3 x_0-x_3) & -4 (18 x_1-x_4) & -2 (267 x_1+x_2-25 x_4) & 231 x_1+79 x_2-31 x_4 & -39 x_1+5 x_2+3 x_4 & 3 x_0 & 6 (2 x_0-x_3) \\
			15 (57 x_0+14 x_3) & -15 (3 x_0-x_3) & -2 (267 x_1+x_2-25 x_4) & -2 (960 x_1+154 x_2-79 x_4) & 231 x_1+79 x_2-31 x_4 & -193 x_1-29 x_2+15 x_4 & 15 x_0 & 6 (4 x_0-3 x_3) \\
			-3 (201 x_0+47 x_3) & 12 (3 x_0-x_3) & 231 x_1+79 x_2-31 x_4 & 231 x_1+79 x_2-31 x_4 & 924 x_1+466 x_2-139 x_4 & 77 x_1+43 x_2-12 x_4 & -12 x_0 & 24 x_0 \\
			54 x_0+13 x_3 & x_3-3 x_0 & -39 x_1+5 x_2+3 x_4 & -193 x_1-29 x_2+15 x_4 & 77 x_1+43 x_2-12 x_4 & x_4-18 x_1 & x_0 & 2 (2 x_0-x_3) \\
			-13 (3 x_1+x_2) & 3 x_1-x_2 & 3 x_0 & 15 x_0 & -12 x_0 & x_0 & -x_1 & 2 x_2 \\
			-26 (9 x_1-3 x_2-x_4) & -2 (9 x_1+3 x_2-x_4) & 6 (2 x_0-x_3) & 6 (4 x_0-3 x_3) & 24 x_0 & 2 (2 x_0-x_3) & 2 x_2 & 2 (8 x_1-x_4) \\
		\end{array}
		\right).$
	}
	\]
	\noindent The real algebraic geometry of $\Sigma_1$ was examined in \cite[Example 4.17]{kummersinn}. For instance, the linear projection $\Sigma_1\to\nP^3$ from the point $(0: -3:  1: 0: -28)$ is real-fibered. This is certified by the fact that the above matrix is positive definite when evaluated at $(0, -3,  1, 0, -28)$, as predicted by Theorem \ref{thm:thmD}.
\end{example}

\section{Isometry classes of symmetric determinantal representations}\label{sec:signatures}
\noindent We now look at the isometry classes of the symmetric determinantal representations that we have obtained, and we complete the proof of Theorem \ref{thm:thmA} and finally show Theorem \ref{thm:thmD}.

\subsection{Square classes} Let $\kk$ be a field of characteristic zero\footnote{For this discussion, it would be enough that the characteristic is different from two.}. The set of square classes of $\kk$ is the quotient $\kk/(\kk^{\times})^2$ and we denote the square class of $\lambda\in \kk$ as $\langle \lambda \rangle \in \kk/(\kk^{\times})^2$.

\begin{example}
For example, the square class of a complex number $\lambda\in \nC$ is either $\langle 0 \rangle$ or $\langle 1 \rangle$, according to whether the number $\lambda=0$ or not. Instead, the square class of a real number $\lambda \in \nR$ is either $\langle -1 \rangle,\langle 0 \rangle,\langle 1 \rangle$, according to the sign of $\lambda$.
\end{example}

Let now $X$ be a $\kk$-variety, let $\sF$ be a coherent sheaf on $X$, and let $\psi \in H^0(X\times X,\sF\boxtimes \sF)$.
 If $x\in X(\kk)$ is a point where $\sF$ is locally free of rank one, we can define the \emph{square class of $\psi(x,x)$} as an element $\langle \psi(x,x) \rangle \in \kk/(\kk^{\times})^2$ as follows:  since $\sF$ is locally free of rank one at $x$, there is an isomorphism $F\colon \sF\otimes \kappa(x) \to \kk$ of $\kk$-vector spaces which induces an isomorphism $F\boxtimes F\colon (\sF\boxtimes \sF)\otimes \kappa(x,x) \to \kk$. The \emph{square class of $\psi(x,x)$} is
 \[ \langle \psi(x,x) \rangle =  \langle (F\boxtimes F)(\psi(x,x)) \rangle \qquad \text{in } \kk/(\kk^\times)^2 .\]
Notice that if we had chosen another isomorphism $F'\colon \sF\otimes \kappa(x) \to \kk$, then $F'=\mu\cdot F$ for $\mu\in \kk^{\times}$ so that
\[ \langle (F'\boxtimes F')(\psi(x,x)) \rangle = \langle \mu^2(F\boxtimes F)(\psi(x,x)) \rangle  = \langle (F\boxtimes F)(\psi(x,x)) \rangle \qquad \text{in } \kk/(\kk^\times)^2 \]
the corresponding square class would be unchanged. Another way to compute the square class is to consider the restriction to the diagonal $\psi_{|\Delta_X} \in H^0(X,\sF^{\otimes 2})$ and then
\[ \langle \psi(x,x) \rangle = \langle F^{\otimes 2}(\psi_{|\Delta_X}(x)) \rangle \qquad \text{ in } \kk/(\kk^{\times})^2. \]
A third way is to take a rational section $T$ of $\sF$ which is regular and nonzero around $x$, and take the isomorphism $F\colon \sF \otimes \kappa(x) \to \kk$ such that $F(T(x))=1$.
Then $T^2$ is a rational section of $\sF^{\otimes 2}$, regular and nonzero around $x$, so around $x$ we can write $\psi_{|\Delta_X} = f\cdot T^2$ for a rational function $f$ on $X$, regular at $x$. The square class of $\psi(x,x)$ is
\[ \langle \psi(x,x) \rangle = \langle F^{\otimes 2}(f(x)\cdot T^2(x)) \rangle = \langle f(x)\cdot F^{\otimes 2}( T^2(x)) \rangle  = \langle f(x)\rangle \qquad \text{ in } \kk/(\kk^{\times})^2. \]

\subsection{Isometry classes of quadratic forms}
Let $\kk$ be a field of characteristic zero\footnote{Again, it would be enough that the characteristic is different from two.}. Its Grothendieck--Witt ring $\GW(\kk)$ is the ring generated by the isomorphism classes of quadratic forms on $\kk$, with the addition and multiplication given by the direct sum and tensor product of quadratic forms, respectively.

\begin{example}
For example, over $\nC$, the isomorphism class of a quadratic form is uniquely determined by its rank, so that there is an isomorphism $\GW(\nC) \cong \mathbf{Z}$. Instead, over $\nR$ the isomorphism class of a quadratic form is determined by its rank and its signature, so that $\GW(\nR)\cong \nZ \times \nZ$. For this reason, if $Q$ is a quadratic form over $\nR$, we will sometimes call the class $\langle  Q\rangle \in \GW(\nR)$  the \emph{signature} of $Q$.
\end{example}

Any quadratic form of rank one is isomorphic to one of the form $\kk \to \kk,x\mapsto \lambda x^2$, with $\lambda\in \kk^{\times}$. We denote the corresponding class in $\GW(\kk)$ as $\langle \lambda \rangle$. We see that $\langle \lambda \rangle = \langle \lambda' \rangle$ if and only if there is $\mu\in \kk^{\times}$ such that $\lambda = \lambda'\mu^2$, so that we can also think of $\langle \lambda \rangle$ as the square class of $\lambda$ in $\kk/(\kk^{\times})^2$. In general, any quadratic form $Q$ can be diagonalized, so that there are square classes $\langle \lambda_i \rangle\in \kk/(\kk^{\times})^2$ such that
\[ \langle Q \rangle = \sum_{i=1}^r \langle  \lambda_i \rangle \qquad \text{in } \GW(\kk).\]

In particular, if $V$ is a finite-dimensional $\kk$-vector space we can think of any $Q\in \operatorname{Sym}^2 V$ as a quadratic form on the dual space $V^{\vee}$ and then we can take its isometry class, which can be identified with the class $\langle Q \rangle$ in $\GW(\kk)$. This can be computed concretely by fixing a basis of $V$, and expressing $Q$ as a symmetric matrix with respect to this basis.
\medskip

We want to apply this to the setting of Lemma \ref{lem:easymadecomplicated}: let $X$ be an irreducible $\kk$-variety and $\sF$ a coherent sheaf of rank one on $X$ with $h^0(X,\sF)=d$. If $\psi\in H^0(X\times X,\sF\boxtimes \sF)$ is symmetric with respect to the involution, we can also see it as an element in $\Sym^2H^0(X,\sF)$, which has a well-defined class in $\GW(\kk)$.

\begin{lemma}\label{lem:signaturepsi}
	Let $\psi\in H^0(X\times X,\sF\boxtimes \sF)$ be a symmetric section and assume that there are points $x_1,\ldots,x_d\in X(\kk)$ such that $\psi(x_i,x_j)=0$ if $i\ne j$ and $\psi(x_i,x_i)\ne 0$ for all $i$. Then the isometry class of $\psi$ is
	\[ \langle \psi \rangle  = \langle \psi(x_1,x_1) \rangle + \cdots + \langle \psi(x_d,x_d) \rangle \qquad \text{ in } \GW(\kk). \]
\end{lemma}
\begin{proof}
	Lemma \ref{lem:easymadecomplicated} shows that the sections  $s_i=\psi(x_i,-)$ form a basis of $H^0(X,\sF)$ such that $\psi = \lambda_1 s_1\otimes s_1 + \cdots +\lambda_d s_d\otimes s_d$ for certain $\lambda_i\in\kk^{\times}$. Hence, the class of $\psi$ is $\sum_{i=1}^d \langle \lambda_i \rangle$. What we need to prove is that $\langle \lambda_i \rangle=\langle \psi(x_i,x_i) \rangle$. Lemma \ref{lem:easymadecomplicated} shows also that $s_j(x_i)=0$ if $i\ne j$ and $s_i(x_i)\ne 0$. Fix then the  isomorphism $F_i\colon \sF\otimes \kappa(x_i) \to \kk$ such that $F_i(s_i(x_i))=1$. Then
	\[ \langle \psi(x_i,x_i) \rangle = \langle (F_i\boxtimes F_i)(\psi(x_i,x_i)) \rangle = \langle (F_i\boxtimes F_i)(\lambda_i\cdot s_i(x_i)\boxtimes s_i(x_i)) \rangle = \langle \lambda_i \rangle. \qedhere \]
\end{proof}

We want to apply this to admissible determinantal representations of secant varieties.

\subsection{Isometry classes of admissible determinantal representations}

We go back to our usual setting: let $C\subseteq \nP^r = \nP(V)$ be  a smooth curve of genus $g$,  over a field $\kk$. We assume that $C$ is embedded by a linear system $V\subseteq H^0(C,L)$ with $r\geq 2k+2$ and we denote by  $\Sigma_k$ the $k$-th secant variety of $C$. Let also $\alpha$ be a theta-characteristic on $C$ and  $\rho\colon \alpha\otimes\alpha \to \omega_C$ an algebraic orientation. If  $\deg(\Sigma_k)=h^0(C_{k+1},A_{k+1,L\otimes \alpha})$ then Theorem \ref{thm:detrep} gives an admissible symmetric determinantal representation of $\Sigma_k$ via the map
\[\phi\colon \wedge^{2k+2} V \longrightarrow H^0(C_{k+1}\times C_{k+1},A_{k+1,L\otimes\alpha}\boxtimes A_{k+1,L\otimes \alpha}). \]
This determinantal representation has an associated Ulrich sheaf $\sA_{k,L\otimes\alpha}$ which is symmetric of rank one. Then, for any $p_0,\dots,p_{2k+1}\in V$ linearly independent and with no common zero on $\Sigma_k$ we have an element
\[ \phi(p_0\wedge \dots \wedge p_{2k+1}) \in \operatorname{Sym}^2 H^0(C_{k+1},A_{k+1,L\otimes\alpha}) \]
and we want to compute its isometry class via Lemma \ref{lem:signaturepsi}. The first thing to do is to compute the square class
\[ \langle \phi(p_0\wedge \dots \wedge p_{2k+1})(\xi,\xi) \rangle \qquad \text{in } \kk/(\kk^{\times})^2 \]
for $\xi\in C_{k+1}(\kk)$. We assume here that $\xi = x_1+\dots+x_{k+1}$ where the $x_i$ are mutually distinct, but not necessarily defined over $\kk$ \footnote{ To do so, we might need to work over $\overline{\kk}$, but this does not matter at the end since $\xi$ is defined over $\kk$}. Around each point $x_j$, choose  rational sections $S_j,A_j$ of $L$ and $\alpha$ respectively, that are regular and nonzero at $x_j$. Via the chosen algebraic orientation, take $W_j = \rho(A_j^{2})$ , which is a rational section of $\omega_C$, regular and nonzero at $x_j$. Then we can write
\[ p_i = f_{ij}\cdot S_j, \quad df_{ij} = f'_{ij}\cdot S_j\cdot W_j  \]
for certain rational functions $f_{ij}$ and $f_{ij}'$ on $C$, regular around $x_j$.  We can think of the $f_{ij}'$ as the derivatives of the $f_{ij}$ with respect to $\alpha$ and the section $A_j$.

\begin{proposition}\label{prop:localdegree}
	With the previous notation, the class $\langle \phi(p_0\wedge \dots \wedge p_{2k+1})(\xi,\xi) \rangle$ in $\GW(\kk)$ is equal to the square class of the following determinant:
    \begin{equation}\label{eq:localdeterminant}
       \det\text{\tiny\(
        \begin{pmatrix}
		f_{01}(x_1) & f'_{01}(x_1) & f_{02}(x_2) & f'_{02}(x_2)& \dots & f_{0,k+1}(x_{k+1}) & f'_{0,k+1}(x_{k+1}) \\
		f_{11}(x_1) & f'_{11}(x_1) & f_{12}(x_2) & f'_{12}(x_2)& \dots & f_{1,k+1}(x_{k+1}) & f'_{1,k+1}(x_{k+1}) \\
		\vdots & \vdots & \vdots & \vdots & \ddots & \vdots & \vdots \\
		f_{2k+1,1}(x_1) & f'_{2k+1,1}(x_1) & f_{2k+1,2}(x_2) & f'_{2k+1,2}(x_2)& \dots & f_{2k+1,k+1}(x_{k+1}) & f'_{2k+1,k+1}(x_{k+1})
	\end{pmatrix}\)}.
	\end{equation}
\end{proposition}
\begin{proof}
	The rational section $\sum_{\sigma\in \mathfrak{S}_{k+1}} S_{\sigma(1)}A_{\sigma(1)} \otimes \dots \otimes S_{\sigma(k+1)}A_{\sigma(k+1)}$ of $(L\otimes \alpha)^{\boxtimes (k+1)}$ is $\mathfrak{S}_{k+1}$-invariant, so that it descends to a rational section of $S_{k+1,L\otimes \alpha}$ that we denote by $(S_1A_1)\dots  (S_{k+1}A_{k+1})$. We also fix a rational function $\Delta$ on $C_{k+1}$ which vanishes on the locus $\{\xi\in C_{k+1} \,|\, \xi \text{ non reduced}\}$ and that is regular and nonzero at $\xi$. We can see $\Delta$ as a rational section of $\sO_{C_{k+1}}(-2\delta)$ which is regular and nonzero at $\xi$. Then  $(S_1A_1)\dots (S_{k+1}A_{k+1})\cdot \Delta$ is a rational section of $A_{k+1,L\otimes\alpha} = S_{k+1,L\otimes \alpha}(-2\delta)$ that is regular and nonzero at $\xi$, and its square
	\[ (S^2_1A^2_1)\cdots (S^2_{k+1}A^2_{k+1})\cdot \Delta^2 = (S^2_1W_1)\ \cdots (S^2_{k+1}W_{k+1})\cdot \Delta^2 \]
	is a rational section of $A_{k+1,L\otimes\alpha}^{\otimes 2} \cong A_{k+1,L}^{\otimes 2}\otimes S_{k+1,\omega_C}$. To compute the class that we want, we need to write
	\[ \phi(p_0\wedge \dots \wedge p_{2k+1})_{|\Delta_{C_{k+1}}} = g\cdot (S^2_1W_1) \cdots (S^2_{k+1}W_{k+1})\cdot \Delta^2 \]
	for a certain rational function $g$ on $C_{k+1}$ and then
	\[ \langle \phi(p_0\wedge \dots \wedge p_{2k+1})(\xi,\xi) \rangle =  \langle g(\xi) \rangle .\]
	We compute the rational function $g$: by definition $\phi(p_0\wedge\dots \wedge p_{2k+1}) = P(p_0\wedge\dots\wedge p_{2k+1})\cdot S_{(\alpha,\alpha)}$, and by the properties of the Szeg\H{o} kernel, we have that ${S_{(\alpha,\alpha)}}_{|\Delta_{C_{k+1}}}=1$.  We are left with the restriction $P(p_0\wedge\dots \wedge p_{2k+1})_{|\Delta_{C_{k+1}}}$, which can be described in terms of the higher Gaussian map: by definition, $P(p_0\wedge \dots \wedge p_{2k+1}) \in H^0(C_{k+1}\times C_{k+1},(N_{k+1,L}\boxtimes N_{k+1,L})(D_{k+1,k+1}))$ is obtained by pulling back $\det_{2k+2,L}(p_0\wedge \dots \wedge p_{2k+1})\in H^0(C_{2k+2},N_{2k+2,L})$ along $\sigma_{k+1,k+1}\colon C_{k+1}\times C_{k+1}\to C_{2k+2}$. The restriction to the diagonal  $\Delta_{C_{k+1}}$ corresponds to the pullback along the doubling map $\sigma_{k+1,k+1}\circ \Delta_{C_{k+1}} = \operatorname{db}_{k+1}$. By definition, this is precisely the $(k+1)$-Gaussian map of \eqref{eq:k+1gaussianmap}:
	\[ P(p_0\wedge\dots \wedge p_{2k+1})_{|\Delta_{C_{k+1}}} = G_{2k+2,L}(p_0\wedge\dots\wedge p_{2k+1}) \in H^0(C_{k+1},A_{k+1,L}^{\otimes 2}\otimes S_{k+1,\omega_C}) . \]
	We now use the explicit expression for the Gaussian map given in \eqref{eq:highergaussian}:
	\begin{align*}
	 G(p_0\wedge\dots \wedge p_{2k+1}) &= \det \left( f_{ij}S_j \,\,\,\, df_{ij} S_i \right)_{ij} = \det \left( f_{ij}S_j \,\,\,\, f'_{ij} S_iW_i \right)_{ij} \\
	 &= \det\left( f_{ij} \,\,\,\, f'_{ij}  \right)_{ij}\cdot (S_1^2W_1)\dots (S^2_{k+1}W_{k+1}) \\
	 & = \frac{1}{\Delta^2}\det\left( f_{ij} \,\,\,\, f'_{ij}  \right)_{ij}\cdot (S_1^2W_1)\dots (S^2_{k+1}W_{k+1})\cdot \Delta^2,
	\end{align*}
	where in the last line we have used the rational function $\Delta$ defined before. Notice that $\Delta(\xi)\ne 0$ so that $\Delta(\xi)^2 \in (\kk^{\times})^2$, hence the previous discussion gives that
	\begin{align*}
	\langle \phi(p_0\wedge\dots\wedge p_{2k+1})(\xi,\xi) \rangle  &= \langle \frac{1}{\Delta(\xi)^2}\cdot \det \left( f_{ij} \,\,\,\, f'_{ij}  \right)_{ij}(\xi)\rangle = \langle \det \left( f_{ij} \,\,\,\, f'_{ij}  \right)_{ij}(\xi)\rangle\\
	&  = \langle \det \left( f_{ij}(x_j) \,\,\,\, f'_{ij}(x_j)  \right)_{ij}\rangle
	\end{align*}
	and this is precisely what we wanted to prove.
\end{proof}

Now it is straightforward to compute the isometry class of our symmetric determinantal representations: assume again that $d=\deg(\Sigma_k)=h^0(C_{k+1},A_{k+1,L\otimes\alpha})$, and consider the projection $\pi=(p_0:\dots:p_{2k+1})\colon \Sigma_{k} \longrightarrow \nP^{2k+1}$. Assume that there is a point  $y\in \nP^{2k+1}(\kk)$ with a fiber $\pi^{-1}(y)=\{y_1,\dots,y_d\}$ made up of $\kk$-rational points $y_i \in \Sigma_{k}(\kk)$. Assume also that  for all $i=1,\dots,d$ there is a unique reduced divisor $\xi_i\in C_{k+1}$ such that $y_i \in \ell(\xi_i)$.

\begin{theorem}\label{thm:signatureformula}
	With the previous notation the class $\langle \phi(p_0\wedge\dots\wedge p_{2k+1})\rangle$ in $\GW(\kk)$ is given by the sum
	\[ \sum_{i=1}^{d} \langle \phi(p_0\wedge\dots\wedge p_{2k+1})(\xi_i,\xi_i)\rangle, \]
	where each summand can be computed as in Proposition \ref{prop:localdegree}.
\end{theorem}
\begin{proof}
	This is a consequence of Lemma \ref{lem:signaturepsi} and of Proposition \ref{prop:localdegree}.
\end{proof}

We complete the proof of Theorem \ref{thm:thmA}.

\begin{proof}[Proof of Theorem \ref{thm:thmA}]
It follows from Theorem \ref{thm:admrepresentation} and Theorem \ref{thm:signatureformula}.
\end{proof}

\subsection{Secant varieties of real curves with maximal signature}
Now we assume that our smooth curve $C\subseteq \nP^r = \nP(V)$ of genus $g$ is defined over the real numbers $\nR$. In particular, for any point $x\in C(\nC)$ we denote by $\bar{x}\in C(\nC)$ its complex conjugate.  We denote by $\Gamma_1,\ldots,\Gamma_s$ the connected components of $C(\nR)$. Recall that $s\leq g+1$ and that $C$ is called an \emph{M-curve} if $s=g+1$.

\medskip

 We assume that $C$ is embedded by a linear system $V\subseteq H^0(C,L)$ with $r\geq 2k+2$ and denote by $\Sigma_k$ the $k$-th secant variety of $C$, which is also a real variety. Let now $p_0,\ldots,p_{2k+1}\in V$ have no common zero on $\Sigma_k$. Following the definition of \cite{shamovichvinnikov,kummershamovich}, the real variety $\Sigma_k$ is \emph{hyperbolic} with respect to $p_0,\ldots,p_{2k+1}$ if and only if the projection
 \[ \pi = (p_0:\cdots:p_{2k+1})\colon \Sigma_k \longrightarrow \nP^{2k+1} \]
 is \emph{real-fibered}, meaning that $\pi^{-1}(\nP^{2k+1}(\nR)) = \Sigma_k(\nR)$. Concretely, this means that the intersection
 \[ \{ \lambda_0p_{0}+\cdots+\lambda_{2k+1}p_{2k+1} = 0 \} \cap \Sigma_k \]
 consists only of real points for each $(\lambda_0:\cdots:\lambda_{2k+1}) \in \nP^{2k+1}(\nR)$. This condition can be restated purely in terms of the curve $C$, following \cite{kummersinn}:
 
\begin{definition}[Vastly real linear systems]
    A linear system $W \subseteq V$ of dimension $2k+2\geq 2$ on the real curve $C$ is \emph{vastly real} if every divisor in $|W|$ contains at most $2k$ nonreal points, counted with multiplicity.
\end{definition}

Then it was shown in \cite[Theorem 4.6]{kummersinn} that 
\[ \pi\colon \Sigma_k \to \nP^{2k+1} \text{ is real-fibered } \quad  \text{ if and only if } \quad  W=\langle p_0,\ldots,p_{2k+1} \rangle \text{ is vastly real. } \]
In general, the notion of hyperbolicity for arbitrary real projective varieties was introduced by Shamovich and Vinnikov in \cite{shamovichvinnikov} in order to extend the Lax conjecture of plane algebraic curves to higher dimension. They also connected it with the theory of admissible determinantal representations. In particular, their \cite[Proposition 3.12]{shamovichvinnikov} applied to our case shows that if $\Sigma_k$ has an admissible determinantal representation $\phi$ such that the symmetric matrix $\phi(p_0\wedge \ldots \wedge p_{2k+1})$ is (positive or negative) definite, then $\pi\colon \Sigma_k\to \nP^{2k+1}$ is real-fibered. Our main result in this section is a converse of this result for secant varieties, which is precisely Theorem \ref{thm:thmD} from the introduction.

\begin{theorem}[Theorem \ref{thm:thmD}]\label{thm:definitedetrep}
	Let $C\subseteq \nP^r = \nP(V)$ be a smooth curve of genus $g$ over $\nR$ embedded by a linear system $V\subseteq H^0(C,L)$ with $r\geq 2k+2$. Further let $p_0,\ldots,p_{2k+1}\in V$ be linearly independent without common zero on the secant variety $\Sigma_k$ and consider the projection $\pi = (p_0:\cdots:p_{2k+1})\colon \Sigma_k\to \nP^{2k+1}$. If $V$ is $k$-very ample, then the following are equivalent
	\begin{enumerate}
		\item The projection $\pi\colon \Sigma_k \rightarrow \nP^{2k+1}$ is real-fibered.
		\item The linear system $W=\langle p_0,\ldots,p_{2k+1} \rangle$ is vastly real.
		\item There is an admissible determinantal representation $\phi$ of $\Sigma_k$ of size $\deg(\Sigma_k)$ such that $\phi(p_0\wedge \cdots \wedge p_{2k+1})$ is (positive or negative) definite. 
	\end{enumerate}
Furthermore, if $C$ is an $M$-curve and $V=H^0(C,L)$, then $V$ is $k$-very ample.
\end{theorem}

We prove this theorem in the next subsection.

\begin{remark}
	We expect that vastly real linear systems of dimension at least $4$ can only exist on $M$-curves, see \cite[Remark 4.7]{kummersinn}. Hence, if this is true, then the conditions (1)-(2)-(3) in Theorem \ref{thm:definitedetrep} are equivalent for complete linear systems of dimension at least $4$.
\end{remark}
	
	\begin{remark} 
	 Theorem \ref{thm:definitedetrep} was already proved for $k=0$, i.e. for the curve itself, in \cite{shamovichvinnikov}.
\end{remark}

\subsection{Vastly real linear systems}

We start by studying vastly real linear systems. In the following $C$ will be a real curve, $L$ a line bundle on $C$ and  $W\subseteq H^0(C,L)$ a linear system. If $\xi\subseteq C$ is an effective divisor, we denote by $W(-\xi) \subseteq H^0(C,L(-\xi))$ the linear system corresponding to the sections of $W$ that vanish on $\xi$. First we see that the condition in the definition of a vastly real linear system is the best possible:

\begin{lemma}\label{lem:mrdim}
 Assume that $W$ is a linear system with $\dim W\geq 2k+1\geq 1$.  Then there is at least one divisor in $|W|$ containing $2k$ nonreal points, counted with multiplicity.
\end{lemma}
\begin{proof} Observe that for any choice of $x_1,\ldots,x_h\in C(\nC)\setminus C(\nR)$ the space $W(-\sum_{i=1}^k (x_i+\bar{x}_i))$  has dimension at least $2k+2-2k=2$. 
\end{proof}

\begin{lemma}\label{lem:vastlyrealclosed}
 Let $k\geq0$ and $\nG=\nG(2k+2,H^0(C,L))$ be the Grassmannian of $2k+2$ dimensional linear systems. The set of vastly real linear systems is closed in $\nG(\nR)$ with the analytic topology.
\end{lemma}

\begin{proof}
 Let $(W_n)_{n\in\nN}$ be a sequence in $\nG(\nR)$ of vastly real linear systems that converges to $W\in\nG(\nR)$. Every $0\neq s\in W$ is the limit of a sequence $(s_n)_{n\in\nN}$ with $0\neq s_n\in W_n$. By continuity of zeros and because $C(\nR)$ is closed in $C(\nC)$ with the analytic topology, it follows that $s$ does not have more than $2k$ nonreal zeros because otherwise $s_n$, for large enough $n$, would have more than $2k$ nonreal zeros, contradicting the assumption that each $W_n$ is vastly real.
\end{proof}

\begin{lemma}\label{lem:fixpair}
	Assume that $\dim W=2k+2 \geq 2$. Then $W$ is vastly real if and only if it separates any divisor of the form $x_1+\ldots+x_h+\bar{x}_1+\cdots+\bar{x}_{h}$, where $0\leq h\leq k+1$ and $x_1,\ldots,x_h\in C(\nC)\setminus C(\nR)$ are arbitrary points. In particular if $W$ is vastly real then $W(-x_1-\cdots-x_h-\bar{x}_1-\cdots-\bar{x}_h)$ is a vastly real linear system of dimension $2k+2-2h$.
\end{lemma}
\begin{proof}
	By definition $W$ is vastly real if and only if for any $x_1,\ldots,x_{k+1}\in C(\nC)\setminus C(\nR)$ it holds that $W(-x_1-\cdots- x_1-\bar{x}_{1}-\cdots - \bar{x}_{k+1}) = 0$. This means precisely that $W$ separates the divisor $x_1+\cdots+x_{k+1}+\bar{x}_{1}+\cdots + \bar{x}_{k+1}$, and then it must separate also all the divisors $x_1+\cdots+x_h+\bar{x}_1+\cdots+\bar{x}_h$ for $0\leq h\leq k+1$. 
\end{proof}

\begin{corollary}\label{cor:dividingtype}
    If there exists a vastly real linear system on $C$, then $C$ is of \emph{dividing type}, i.e., the set $C(\nC)\setminus C(\nR)$ has two connected components that are exchanged by complex conjugation.
\end{corollary}

\begin{proof}
 Lemma \ref{lem:fixpair} shows that there exists a vastly real linear system of dimension $2$ on $C$. After removing base-points if necessary, this defines a real-fibered morphism $C\to\nP^1$ which can only exist if $C$ is of dividing type.
\end{proof}

From Lemma \ref{lem:fixpair} we can also deduce a similar statement with the complex conjugate pair replaced by a real double point.

\begin{corollary}\label{cor:fixdouble}
    Let $W$ be a vastly real linear system of dimension $\dim W = 2k+2$, let $x_1,\ldots,x_h\in C(\nR)$ for $1\leq h\leq k$ and consider the linear system $W'=W(-2x_1-\cdots-2x_h)$.  Then there is a vastly real linear system $W''\subseteq W'$ of dimension $\dim W''=2k+2-2h$. 
\end{corollary}

\begin{proof}
   By induction, it is enough to prove the statement for $h=1$, and we set $x=x_1\in C(\nR)$.  Choose a sequence of points $y_n\in C(\nC)\setminus C(\nR)$ that converges to $x$ in the analytic topology and consider the subspaces $W_n=W(-y_n-\bar{y}_n)$. These are vastly real linear system of dimension $2k$ thanks to Lemma \ref{lem:fixpair}. Let $W''$ be an accumulation point of the sequence $(W_n)_{n\in\nN}$ in the Grassmannian of $2k$-dimensional linear subspaces of $W$. By continuity $W''$ has $2x$ in its base locus, so that $W''\subseteq W'$, and by Lemma \ref{lem:vastlyrealclosed} it is  vastly real of dimension $2k$.
\end{proof}

\begin{lemma}\label{lem:numbersconstant}
 Let $S\subseteq H^0(C,L)$ and assume that every $s\in S$ has only real zeros. For every connected component $S_0$ of $S$ with the analytic topology, the number of zeros (counted with multiplicities), that a section $s\in S_0$ has on each $\Gamma_i$, is constant.
\end{lemma}

\begin{proof}
 Fix some $d_1,\ldots,d_s\geq0$ such that $\deg(L)=d_1+\cdots+d_s$.
 Since the $\Gamma_i$ are closed subsets of $C(\nC)$, the set $S(d_1,\ldots,d_s)$ of all $s\in S$ that have at least (and thus exactly) $d_i$ zeros on $\Gamma_i$ for all $i=1,\ldots,s$ is closed inside $S$. On the other hand, the set $S$ is the disjoint union of finitely many such sets of the form $S(e_1,\ldots,e_s)$ with $\deg(L)=e_1+\cdots+e_s$. Hence, each such set is closed and open in $S$ which proves the claim.
\end{proof}

The next proposition is not needed in its full generality but we include it anyways as it might be of independent interest.

\begin{proposition}\label{prop:typex}
 Let $W\subseteq H^0(C,L)$ be a vastly real linear system of dimension $2k+2$. There are unique integers $d_1,\ldots,d_s>0$ with $d_1+\cdots+d_s=\deg(L)-2k$ such that every $s\in W$ has at least $d_i$ zeros on $\Gamma_i$ (counted with multiplicity) for all $i=1,\ldots,s$.
\end{proposition}

\begin{proof}
 We first prove the uniqueness. Assume for the sake of a contradiction that there are $d_1',\ldots,d_s'>0$ satisfying the same properties but $d_i\neq d_i'$ for some $i$. Then every $0\neq s\in W$ has at least $\max(d_i,d_i')$ zeros on $\Gamma_i$ for all $i=1,\ldots,s$ which implies that $s$ has more than $\deg(L)-2k$ zeros on $C(\nR)$. This contradicts Lemma \ref{lem:mrdim}.

\medskip

\noindent  Next we observe that it suffices to prove the existence for linear systems without base points. We do this by induction on $k$.
 For $k=0$ the linear system $W$ corresponds to a real-fibered morphism $C\to\nP^1$. Its restriction to $C(\nR)$ is a covering map onto $\nP^1(\nR)$ of degree $\deg(L)$ by \cite[Theorem 2.19]{kummershamovich}. Hence we can choose $d_i$ to be the degree of this covering map restricted to $\Gamma_i$.

\medskip

\noindent  Now let $k>0$ and $x_0\in C(\nC)\setminus C(\nR)$. We consider the two linear systems $W_0=W(-x_0-\overline{x}_0)$ and $W'_0=W(-kx_0-k\overline{x}_0)\subseteq W_0$.  
 By Lemma \ref{lem:fixpair} they are vastly real of dimensions $\dim W_0=2k$ and $\dim W'_0=2$. Then by induction there must be two tuples of positive numbers $(d_1,\ldots,d_r)$ and $(d_1',\ldots,d_r')$ such that $\sum_{i=1}^r d_i = \sum_{i=1}^{r} d_i' = \deg(L)-2k$ and such that any section in $W_0$ (resp. $W'_0$) has at least $d_i$ zeros (resp. $d'_i$ zeros) on $\Gamma_i$. Since $W'_0\subseteq W_0$ the uniqueness statement proved before shows that $d_i=d_i'$. We claim now that these numbers do not depend on $x_0$. Indeed, let $x_1\in C(\nC)\setminus C(\nR)$. After possibly replacing $x_1$ with its complex conjugate we can assume that $x_0$ and $x_1$ lie in the same connected component of $C(\nC)\setminus C(\nR)$ (see Corollary \ref{cor:dividingtype}). We can thus connect these two points by a path $x_t \in C(\nC)\setminus C(\nR)$, $t\in[0,1]$. This gives a continuous family of vastly real linear systems $W'_t=W(-kx_t-k\bar{x}_t)$ of dimension $2$. Lemma \ref{lem:numbersconstant} applied to the connected set $S=\cup_{t\in[0,1]}(W'_t\setminus\{0\})$ shows that every $s\in W'_1$ has degree $d_i$ on $\Gamma_i$ since this is the case for $W'_0$.

\medskip

\noindent So far we have shown that there are integers $d_1,\ldots,d_r>0$ with $d_1+\cdots+d_r=\deg(L)-2k$ such that every $s\in W$ with at least one pair of complex conjugate nonreal zeros has at least $d_i$ zeros on $\Gamma_i$ for all $i=0,\ldots,r$.
  It remains to consider $s_0\in W$ with only real zeros. Choose $0\neq s_1\in W$ such that $s_1$ does not have only real zeros.
 Let $t_m\geq0$ be maximal with the property that $s_0+ts_1$ has only real zeros for all $0\leq t\leq t_m$. By Lemma \ref{lem:numbersconstant} the section $s_2=s_0+t_ms_1$ has the same number of zeros on each $\Gamma_i$ as $s_0$. On the other hand, the section $s_2$ is also the limit $\lim_{t\to0}(s_2+ts_1)$ and $s_2+ts_1$ does not have only real zeros for small enough $t>0$. Thus $s_2+ts_1$ has at least $d_i$ zeros on $\Gamma_i$ for all $i=1,\ldots,s$ and all small enough $t>0$. Hence, by continuity of zeros, the same is true for $s_2$ and thus for $s_0$.
 \end{proof}

\begin{definition}
Let $W\subseteq H^0(C,L)$ be a vastly real linear system of dimension $2k+2$. We call the tuple $(d_1,\ldots,d_r)$ from Proposition \ref{prop:typex} the \emph{type} of $W$ and say that $W$ is \emph{vastly real of type $(d_1,\ldots,d_r)$}.
\end{definition}

\begin{remark}
    Let $W\subseteq H^0(C,L)$ be a vastly real linear system of dimension $2k+2$ and of type $(d_1,\ldots,d_r)$. If $k=0$ and $W$ is base-point free, then it corresponds to a real-fibered morphism $C\to\nP^1$ whose restriction to $S_i$ has degree $d_i$. The question of which tuples $(d_1,\ldots,d_r)$ can occur was studied in \cite{kumshaw}. For instance, it was shown that for an $M$-curve every $(g+1)$-tuple of positive integers can be realized. The situation when $k=1$ and $W$ is very ample was examined in \cite{mikhalkinorevkov}. It was shown that in this case $C$ must be an $M$-curve and that every tuple of positive integers can be realized. In \cite{kummersinn}, for every $M$-curve and all $k,d\geq1$ examples of vastly real linear systems of dimension $2k+2$ and of type $(d,1,\ldots,1)$ have been constructed.
\end{remark}

Now we are almost ready to prove the statement on $M$-curves in Theorem \ref{thm:definitedetrep}. We just need the following statement from Huisman:

\begin{lemma}[{\cite[Theorem 2.4]{huismannsd1}}]\label{lem:huisman}
	Let $0\neq s\in H^0(C,L)$ have a zero on at least $g$ different connected components of $C(\nR)$. Then $L$ is nonspecial.
\end{lemma}

\begin{proposition}\label{prop:vastlyrealveryample}
    If $C$ is an $M$-curve and $W\subseteq H^0(C,L)$ is a vastly real linear system of dimension $2k+2$, then $L$ is $k$-very ample.
\end{proposition}

\begin{proof}
     We prove this by induction on $k$. Let $k=0$ and $x\in C(\nR)$. Since the dimension of $W$ is $2$, there is a $0\neq s\in W$ which vanishes at $x$. By Proposition \ref{prop:typex} it also vanishes at a point on each of the $g$ connected components of $C(\nR)$ that do not contain $x$. It follows from Lemma \ref{lem:huisman} that both $L$ and $L(-x)$ are nonspecial. This shows that $h^0(L(-x))=h^0(L)-1$. Thus $L$ has no real base points. Since elements of $W$ have only real zeros, it has no nonreal base points either.

\medskip

\noindent     Now let $k>0$ and $x\in C(\nC)$. The linear system $W(-x-\overline{x})\subseteq H^0(C,L(-x-\overline{x}))$ corresponding to all $s\in W$, that vanish at $x$ and $\overline{x}$, contains a maximally real linear system of dimension $2k$ by Corollary \ref{cor:fixdouble} if $x$ is real and by Lemma \ref{lem:fixpair} if $x$ is nonreal. Thus $L(-x-\overline{x})$ is $(k-1)$-very ample by induction hypothesis for every $x$. But then $L(-x)$ is also $(k-1)$-very ample for every $x$ which implies that $L$ is $k$-very ample.
\end{proof}

\begin{example}
    Proposition \ref{prop:vastlyrealveryample} can fail if we do not assume that $C$ is an $M$-curve. Indeed, let $C$ be a curve of genus $2$ such that the canonical map $C\to\nP^1$ is real-fibered. Now for any $x\in C(\nR)$ the line bundle $L(x)$ has $x$ as a base point and $H^0(C,L(x))$ is vastly real of dimension $2$.
\end{example}

We now go back to the proof of Theorem \ref{thm:definitedetrep}:  to this end, we assume for the remainder of this section that $V$ is $k$-very ample and let $p_0,\ldots,p_{2k+1}\in V$ be linearly independent without common zero on the secant variety $\Sigma_k$ such that their span $W=\langle p_0,\ldots,p_{2k+1} \rangle \subseteq V$ is vastly real. Next, we have to choose a theta characteristic and an algebraic orientation on $C$. Note that every algebraic orientation naturally induces a topological orientation on $C(\nR)$, and, conversely, every topological orientation on $C(\nR)$ comes from an algebraic orientation. This was stated in \cite[Example 4.5]{agostinikummer} and follows by the same argument as in \cite[Corollary 2.4]{tottet}. Since $C$ is of dividing type by Corollary \ref{cor:dividingtype}, the real points $C(\nR)$ are the boundary of each of the two components of $C(\nC)\setminus C(\nR)$. Therefore, a topological orientation on one of these components induces an orientation on $C(\nR)$. Such an orientation is called \emph{complex orientation} of $C(\nR)$ and it is unique up to global reversal\footnote{An orientation is induced by a nowhere-vanishing volume form $\omega$. Its global reversal is induced by $-\omega$.} In the following we fix a theta characteristic $\alpha$ and an algebraic orientation $\rho\colon\alpha\otimes\alpha\to\omega_C$ which induces one of the two complex orientations on $C(\nR)$. It follows from Cauchy's theorem that $\alpha$ cannot have global sections, so that Theorem \ref{thm:detrep} gives an admissible symmetric determinantal representation $\phi$ of $\Sigma_k$ of size $\deg(\Sigma_k)$ whose associated Ulrich sheaf is $\sA_{k,L\otimes\alpha}$, which is symmetric of rank one. The only remaining step is to verify that $\phi(p_0\wedge\cdots\wedge p_{2k+1})$ is positive or negative definite, for which we will use Theorem \ref{thm:signatureformula}.

\begin{theorem}\label{thm:definite}
    With the notation as above the symmetric matrix $\phi(p_0\wedge\dots\wedge p_{2k+1})$ is positive or negative definite.
\end{theorem}

\begin{proof}
    Because the projection $\pi=(p_0:\dots:p_{2k+1})\colon \Sigma_{k} \longrightarrow \nP^{2k+1}$ is real-fibered, we can apply Theorem \ref{thm:signatureformula}. It therefore suffices to show that the determinant (\ref{eq:localdeterminant}) in Proposition \ref{prop:localdegree} has the same sign for every $\xi\in C_{k+1}(\nR)$ where it does not vanish. We use the same notation as in Proposition \ref{prop:localdegree}.

\medskip

\noindent    In the case $k=0$ the map $\pi\colon C\to\nP^1$ is real-fibered and given by the rational function $g=\frac{f_{11}}{f_{01}}$. The determinant in Proposition \ref{prop:localdegree} is $g'(x_1)\cdot f_{01}(x_1)^2$ where $g'$ is the derivative with respect to the rational differential $W_1$, which, by our choice of $\alpha$, induces the complex orientation on $C(\nR)$. Hence the sign of the determinant in question is the local degree of $\pi$ at $x_1$ with respect to the complex orientation. Since $\pi$ is a covering map by \cite[Theorem 2.19]{kummershamovich} and the complex orientation of $C(\nR)$ can be obtained as pull-back of an orientation of $\nP^1(\nR)$ under $\pi$, these must all have the same sign.

\medskip

\noindent    In the case $k=1$ we first note that the symmetric admissible determinantal representation from \cite[Theorem 6.5]{agostinikummer} is, up to some constant factor, the same as $\phi$ because they both correspond to the same symmetric rank one Ulrich sheaf. Now the statement of \cite[Theorem 6.5]{agostinikummer} combined with \cite[Remark 6.3]{agostinikummer} says that the signature of $\phi(p_0\wedge\dots\wedge p_{3})$ is (up to sign) the so-called encomplexed writhe of the space curve $\pi(C)\subseteq\nP^3$. Because $W$ is vastly real, the encomplexed writhe of $\pi(C)\subseteq\nP^3$ is equal to $\deg(\Sigma_1)$ by \cite[Theorem 2]{mikhalkinorevkov}. Thus $\phi(p_0\wedge\dots\wedge p_{3})$ is either positive or negative definite. For later reference we note that this also implies, by Theorem \ref{thm:signatureformula}, that (\ref{eq:localdeterminant}) has the same sign for every $\xi\in C_{k+1}(\nR)$ where it does not vanish.

\medskip

\noindent    Finally, let $k>1$ and $U\subseteq C_{k+1}$ be the Zariski open subset where (\ref{eq:localdeterminant}) does not vanish. We will prove that (\ref{eq:localdeterminant}) takes only one sign on $U\cap C_{k+1}(\nR)$ which by Theorem \ref{thm:signatureformula} proves the claim. Assume for the sake of a contradiction that there are $\xi,\xi'\in U\subseteq C_{k+1}$ such that (\ref{eq:localdeterminant}) takes different signs on $\xi$ and $\xi'$. We write
    \begin{equation*}
     \xi=x_1+\cdots+x_{k+1}\textrm{ and }\xi'=x'_1+\cdots+x'_{k+1}
    \end{equation*} We assume that we have chosen $\xi$ and $\xi'$ such that the number $s$ of indices $j$ with $x'_j\not\in\{x_1,\ldots,x_{k+1}\}$ is minimal among all such examples.

\medskip

\noindent    We first consider the case $s=1$. We can write $\xi'=x'_1+x_2+\cdots+x_{k+1}$ with $x'_1$ not in the support of $\xi$. Then $x_1$ and $x_1'$ are necessarily real points of $C$. A multiple application of Corollary \ref{cor:fixdouble} shows that there are linearly independent sections $q_0,q_1\in W$ which vanish with multiplicity at least $2$ on $x_2,\ldots,x_{k+1}$ such that every nonzero linear combination $\lambda q_0+\mu q_1$, $\lambda,\mu\in\nR$, has only real zeros on $C$. Applying an endomorphism of $W$ with determinant $1$ we can assume that $q_i=p_i$, $i=0,1$. Then (\ref{eq:localdeterminant}) evaluated at $x+x_2+\cdots+x_{k+1}$ is equal to
\begin{equation}\label{eq:localblockdet1}
       \text{\footnotesize\(\det
        \begin{pmatrix}
		f_{01}(x) & f'_{01}(x)  \\
		f_{11}(x) & f'_{11}(x)
	\end{pmatrix}\cdot\det
        \begin{pmatrix}
		 f_{22}(x_2) & f'_{22}(x_2)& \dots & f_{2,k+1}(x_{k+1}) & f'_{2,k+1}(x_{k+1}) \\
		 f_{32}(x_2) & f'_{32}(x_2)& \dots & f_{3,k+1}(x_{k+1}) & f'_{3,k+1}(x_{k+1}) \\
		\vdots & \vdots & \ddots & \vdots & \vdots \\
		 f_{2k+1,2}(x_2) & f'_{2k+1,2}(x_2)& \dots & f_{2k+1,k+1}(x_{k+1}) & f'_{2k+1,k+1}(x_{k+1})
	\end{pmatrix}\)},
	\end{equation}
    where we have chosen $f_{01}$ and $f_{11}$ to meet the requirements above Proposition \ref{prop:localdegree} for $x_1$ and $x_1'$ simultaneously. By construction the morphism $C\to\nP^1$ defined by $g=\frac{f_{11}}{f_{01}}$ is real-fibered. Hence by the same argument as in the case $k=0$ above, the first determinant, which is $g'(x)\cdot f_{01}(x)^2$, has the same sign for every $x$ where it is defined and where $f_{01}$ is nonzero. In particular, since the second determinant in (\ref{eq:localblockdet1}) does not depend on $x$, this shows that (\ref{eq:localdeterminant}) has the same sign when evaluated at $\xi$ and $\xi'$. This contradicts our choice of $\xi$ and $\xi'$.

\medskip

\noindent        Finally, we consider the case that $s\geq2$. After relabeling if necessary, we can assume that $x_{i} \notin \{x'_1,\ldots,x'_{k+1}\},x'_{i}\notin\{x_1,\ldots,x_{k+1}\}$ for $i=1,2$ and that both $x_1+x_2$ and $x_1'+x_2'$ are defined over $\nR$.
    We consider the maps $\sigma,\sigma'\colon C_{2}\to C_{k+1}$ that send $\eta\in C_{2}$ to $\eta+(x_3+\cdots+x_{k+1})$ and $\eta+(x'_3+\cdots+x'_{k+1})$, respectively. We have
    \begin{equation*}
        x_1+x_2\in\sigma^{-1}(U) \textrm{ and } x'_1+x'_2\in\sigma'^{-1}(U),
    \end{equation*}
    so that $\sigma^{-1}(U)$ and $\sigma'^{-1}(U)$ are non-empty Zariski open subsets. Since the real points of $C_2$ are Zariski dense, there exists $y_1+y_2\in C_2(\nR)$ with $y_1+y_2\in\sigma^{-1}(U)\cap\sigma'^{-1}(U)$. We proceed analogously to the case $s=1$. A multiple application of Corollary \ref{cor:fixdouble} shows that there are linearly independent sections $q_0,\ldots,q_3\in W$ which vanish with multiplicity at least $2$ on $x_3,\ldots,x_{k+1}$ such that every nonzero real linear combination of $q_0,\ldots,q_3$ has at most two nonreal zeros on $C$, counted with multiplicities. Applying an endomorphism of $W$ with determinant $1$ we can assume that $q_i=p_i$, $i=0,\ldots,3$. Then (\ref{eq:localdeterminant}) evaluated at $z_1+z_2+x_3+\cdots+x_{k+1}$ is equal to
\begin{equation*}
       \text{\footnotesize\(\det
        \begin{pmatrix}
		f_{01}(z_1) & f'_{01}(z_1) & f_{02}(z_2) & f'_{02}(z_2) \\
		f_{11}(z_1) & f'_{11}(z_1) & f_{12}(z_2) & f'_{12}(z_2) \\
		f_{21}(z_1) & f'_{21}(z_1) & f_{22}(z_2) & f'_{22}(z_2) \\
		f_{31}(z_1) & f'_{31}(z_1) & f_{32}(z_2) & f'_{32}(z_2) \\
	\end{pmatrix}\cdot\det
        \begin{pmatrix}
		 f_{33}(x_3) & f'_{33}(x_3)& \dots  & f'_{3,k+1}(x_{k+1}) \\
		 f_{43}(x_3) & f'_{43}(x_3)& \dots  & f'_{4,k+1}(x_{k+1}) \\
		\vdots & \vdots & \ddots & \vdots  \\
		 f_{2k+1,3}(x_3) & f'_{2k+1,3}(x_3)& \dots & f'_{2k+1,k+1}(x_{k+1})
	\end{pmatrix}\)},
	\end{equation*}
    where we have chosen the $f_{ij}$, $i=0,\ldots,3$ and $j=1,2$, to meet the requirements above Proposition \ref{prop:localdegree} for $x_j$ and $z_j$ simultaneously. As shown in the case $k=1$, the first factor has the same sign for every real $z_1+z_2\in C_2(\nR)$ whenever it is nonzero. In particular, since the second factor does not depend on $z_1+z_2$, this shows that (\ref{eq:localdeterminant}) has the same sign when evaluated at $\xi=x_1+x_2+x_3+\cdots+x_{k+1}$ and $\nu=y_1+y_2+x_3+\cdots+x_{k+1}$. By the same argument, the determinant (\ref{eq:localdeterminant}) has the same sign when evaluated at $\xi'=x'_1+x'_2+x'_3+\cdots+x'_{k+1}$ and $\nu'=y_1+y_2+x'_3+\cdots+x'_{k+1}$. In particular, it has different signs at $\nu$ and $\nu'$ because this is the case for $\xi$ and $\xi'$. However, this contradicts the minimality of $s$, and thus completes the proof.
\end{proof}

We write out the proof of Theorem \ref{thm:definitedetrep} for completeness.

\begin{proof}[Proof of Theorem \ref{thm:definitedetrep}]
  Assume first that $V$ is $k$-very ample. The fact that condition (1) implies condition (3) is Theorem \ref{thm:definite}. The fact that condition (3) implies condition (1) is \cite[Proposition 3.12]{shamovichvinnikov} and the fact that condition (1) and (2) are equivalent is \cite[Theorem 4.6]{kummersinn}. Finally, if $C$ is an $M$-curve and if the complete linear system $H^0(C,L)$ is vastly real, then Proposition \ref{prop:vastlyrealveryample} shows that $H^0(C,L)$ is $k$-very ample.
\end{proof}

\begin{remark}\label{rem:hypcone}
	Consider the case $r=2k+2$ in Theorem \ref{thm:definitedetrep}. Then $\Sigma_k$ is a hypersurface defined by a homogeneous polynomial $H\in\nR[x_0,\ldots,x_r]$ and the zero locus of the linear forms  $p_0,\ldots,p_{2k+1}\in V$ is a point $[e]\in\nP^r(\nR^{r+1})$. Without loss of generality we can assume $H(e)>0$. Condition (3) in Theorem \ref{thm:definitedetrep} then translates into $H$ being \emph{hyperbolic with respect to $e$}, meaning that $H(te+a)\in\nR[t]$ has only real zeros for all $a\in\nR^{r+1}$. Condition (1) means that there are real symmetric matrices $A_0,\ldots,A_r$ such that
	\begin{equation}\label{eq:detrephyp}
	 H=\det(A(x))\textrm{ and }A(e)\textrm{ is positive definite,}
	\end{equation}
 where $A(x)=x_0A_0+\cdots+x_rA_r$. Of particular interest in convex optimization is the \emph{hyperbolicity cone} of $H$ \cite{guler}, defined as the closure of the connected component of $\{a\in\nR^{r+1}\mid H(a)\neq0\}$ that contains $e$. If we have (\ref{eq:detrephyp}), then it can be described as
 \begin{equation*}
  \{a\in\nR^{r+1}\mid A(e)\textrm{ is positive semidefinite}\},
 \end{equation*}
meaning that it is a so-called \emph{spectrahedron} \cite{BPT}.
If $C$ is an $M$-curve embedded via a complete linear system such that the components $\Gamma_1,\ldots,\Gamma_{g}$ of $C(\nR)$ are non-trivial in $H_1(\nP^r(\nR),\nZ/2)$, then $\Sigma_k$ is hyperbolic and its hyperbolicity cone is the conic hull of $\Gamma_{g+1}$ by \cite[Proposition 4.19]{kummersinn}.
By Theorem \ref{thm:definitedetrep}, this hyperbolicity cone is a spectrahedron, which answers \cite[Question (4)]{kummersinn} affirmatively.
\end{remark}

\bibliographystyle{ams}

\end{document}